\documentclass[a4paper,reqno, 10pt]{amsart}

\usepackage{amsmath,amssymb,amsfonts,amsthm,mathrsfs}
\usepackage{tikz}

\usepackage{verbatim,wasysym,cite}

\usepackage[latin1]{inputenc}
\usepackage{microtype}
\usepackage{color,enumitem,graphicx}
\usepackage[colorlinks=true,urlcolor=blue, citecolor=red,linkcolor=blue,
linktocpage,pdfpagelabels, bookmarksnumbered, bookmarksopen]{hyperref}
\usepackage[english]{babel}
\renewcommand{\epsilon}{{\varepsilon}}
\usepackage{enumitem}
\numberwithin{equation}{section}
\newtheorem{theorem}{Theorem}[section]
\newtheorem{lemma}[theorem]{Lemma}
\newtheorem{remark}[theorem]{Remark}
\newtheorem{definition}[theorem]{Definition}
\newtheorem{proposition}[theorem]{Proposition}
\newtheorem{corollary}[theorem]{Corollary}

\newcommand{\C}{\mathbb C}

\newcommand{\R}{\mathbb R}
\newcommand{\N}{\mathbb N}

\newcommand{\Z}{\mathbb Z}

\def\({\left(}
\def\){\right)}
\def\<{\left\langle}
\def\>{\right\rangle}

\def\O{\mathcal O}

\def\F{\mathcal F}
\def\K{\mathcal K}

\def\L{\mathcal L}
\def\EE{\mathcal E}

\def\eps{\varepsilon}

\def\M{\mathcal M}

\DeclareMathOperator{\RE}{Re}
\DeclareMathOperator{\IM}{Im}

\newcommand{\qtq}[1]{\quad\text{#1}\quad}

\begin{document}
	
	\title[Dynamics of the combined NLS with inverse-square potential]{Dynamics of the combined nonlinear Schr\"odinger equation with inverse-square potential}
	
	\author{Zuyu Ma}
	\address{Zuyu Ma
		\newline \indent The Graduate School of China Academy of Engineering Physics,
		Beijing 100088,\ P. R. China}
	\email{mazuyu23@gscaep.ac.cn}
	
\author{Yilin Song}
\address{Yilin Song
	\newline \indent The Graduate School of China Academy of Engineering Physics,
	Beijing 100088,\ P. R. China}
\email{zwqsylsj@163.com}

\author{Jiqiang Zheng}
\address{Jiqiang Zheng
\newline \indent Institute of Applied Physics and Computational Mathematics,
Beijing, 100088, China.
\newline\indent
National Key Laboratory of Computational Physics, Beijing 100088, China}
\email{zheng\_jiqiang@iapcm.ac.cn, zhengjiqiang@gmail.com}

	\begin{abstract} We consider the long-time dynamics of  focusing energy-critical Schr\"odinger equation perturbed by the $\dot{H}^\frac{1}{2}$-critical nonlinearity and with inverse-square potential(CNLS$_a$) in dimensions $d\in\{3,4,5\}$
		\begin{align}\label{NLS-ab}
			i\partial_tu-\mathcal{L}_au=-|u|^{\frac{4}{d-2}}u+|u|^{\frac{4}{d-1}}u, \quad (t,x)\in\R\times\R^d,\tag{CNLS$_a$}
		\end{align}
		 where $\mathcal{L}_a=-\Delta+a|x|^{-2}$ and the energy is below and equal to the threshold  $m_a$, which is given by the ground state $W_a$ satisfying $\mathcal{L}_aW_a=|W_a|^{\frac{4}{d-2}}W_a$.   When the energy is below the threshold, we utilize the concentration-compactness argument as well as the variatonal analysis to characterize the scattering and blow-up region. When the energy is equal to the threshold, we use the modulation analysis associated to the equation \eqref{NLS-ab} to classify the dynamics of $H_a^1$-solution. In both regimes of scattering results, we do not need the radial assumption in $d=4,5$. Our result generalize the scattering results of \cite{MXZ2013,MXZ-high,MiaoZhaoZheng-CVPDE} and \cite{AJZ}  in the setting of standard combined NLS.
	\end{abstract}
	
	
	\maketitle
	
	\section{Introduction}
	\label{sec:intro}
	In this article, we investigate the global well-posedness and the asymptotic behavior of solutions to the focusing-defocusing  combined nonlinear Schr\"odinger equation(NLS) with an inverse-square potential:
	\begin{equation}\tag{CNLS$_{a}$}\label{NLS}
		\begin{cases}
			(i\partial_{t}-\L_{a})u=-|u|^{\frac{4}{d-2}}u+|u|^{\frac{4}{d-1}}u,\\
			u|_{t=0}=u_{0}\in H_a^{1}(\R^{d}),
		\end{cases} 
	\end{equation}
	where $u: \R\times\R^{d}\rightarrow \C$, and the operator
	$$
	\mathcal{L}_{a}=-\Delta+a|x|^{-2}
	$$
	is defined via the Friedrichs extension of the following quadratic form $Q(f)$ on $C^{\infty}_{c}(\R^{d}\backslash\left\{0\right\})$
	$$
	Q(f)=\int_{\R^d}|\nabla f|^2+\frac{a}{|x|^2}|f(x)|^2dx
	$$
	with  $a>-\tfrac{(d-2)^2}{4}$ and $d\geqslant3$, which  guarantees the positivity of $\L_a$ and the equivalence of Sobolev norms
	\begin{equation}\label{equiNorms}
		\|u\|_{\dot H^1} \sim \|u\|_{\dot H_a^1}:=\|\mathcal{L}_a^\frac{1}{2}u\|_{L^2}.
	\end{equation}
	from the sharp Hardy inequality. The operator $\mathcal{L}_a^\frac{1}{2}$ is defined via the standard functional calculus. By using the inner-product of $L^2$, we can rewrite the $\dot{H}_a^1$ as
	\begin{align*}
		\|u\|_{\dot{H}_a^1(\R^d)}^2=\|\mathcal{L}_a^\frac{1}{2}u\|_{L^2(\R^d)}^2=\langle\mathcal{L}_a^\frac{1}{2}u,\mathcal{L}_a^\frac{1}{2}u\rangle_{L^2}=Q(u).
	\end{align*}

	The equation \eqref{NLS} has two conserved quantities, namely, the mass and energy:
	\begin{align*}
		\mbox{Mass }&:M(u)=\int_{\R^{d}}|u|^{2}\,dx,\\
		\mbox{Energy }&: E_{a}(u)=\int_{\R^{d}}\left(\frac{1}{2}|\nabla u|^{2}+\frac{a}{2|x|^{2}}|u|^{2}
		+\frac{d-1}{2(d+1)}|u|^{\frac{2(d+1)}{d-1}}-\frac{d-2}{2d}|u|^{\frac{2d}{d-2}}\right)\,dx.
	\end{align*}
	By \eqref{equiNorms} and Sobolev embedding, we see that  $M(u_{0})$, $|E_{a}(u_{0})|<\infty$ if $u_{0}\in H_a^1$.
	
	An important model related to \eqref{NLS} is the focusing energy-critical NLS with inverse-square potential:
	\begin{align}\tag{NLS$_a$}\label{NLS-Potential}
		\begin{cases}
			i\partial_tu-\mathcal{L}_au=-|u|^{\frac{4}{d-2}}u, &(t,x)\in\Bbb{R}\times\R^d,\\
			u(0,x)=u_0\in\dot{H}_a^1(\R^d).
		\end{cases}
	\end{align}
	The solution to $\eqref{NLS-Potential}$ enjoys two conservation laws 
	\begin{gather*}
		M(u)=\int_{\R^d}|u|^2dx,\\
		E_a^c(u)=\frac{1}{2}\int_{\R^d}|\nabla u|^2dx+\frac{1}{2}\int_{\R^d}\frac{a|u|^2}{|x|^2}dx-\frac{d-2}{2d}\int_{\R^d}|u|^\frac{2d}{d-2}dx.
	\end{gather*}
	
	The equation $\eqref{NLS-Potential}$ has radial positive static solution $W_a(x)$, where $W_a$ is the solution to the following elliptic equation
	\begin{align}\label{W_a}
		\mathcal{L}_aW_a(x)=|W_a|^{\frac{4}{d-2}}W_a.
	\end{align}
	More specifically, a radial explicit solution of \eqref{W_a}  is given by
	\begin{equation}\label{E:Wa}
		W_a(x) :=    [d(d-2)\beta^2]^{\frac{d-2}{4}} \biggl[ \frac{ |x|^{\beta-1} }{ 1+|x|^{2\beta} }\biggr]^{\frac{d-2}{2}},
	\end{equation}
	where  $1\geqslant\beta>0$ and $a=(\frac{d-2}2)^2[\beta^2-1]$. We refer to this solution as the \emph{ground state}. 
	
	In particular, when $a=0$, the equation \eqref{NLS}  turns to the classical focusing-defocusing combined NLS: 
	\begin{align*}\label{Combine-NLS}
		\begin{cases}
			i\partial_tu+\Delta u=-|u|^{\frac{4}{d-2}}u+|u|^{\frac{4}{d-1}}u,&(t,x)\in\R\times\R^d,\\
			u(0,x)=u_0\in H^1(\R^d)\tag{CNLS}
		\end{cases}
	\end{align*}
	where its solution enjoys two conservation laws as
	\begin{align*}
		&  M_0(u)=\int_{\R^d}|u|^2dx\\
		& E_0(u)=\frac{1}{2}\int_{\R^d}|\nabla u|^2dx-\frac{d-2}{2d}\int_{\R^d}|u|^{\frac{2d}{d-2}}dx+\frac{d-1}{2(d+1)}\int_{\R^d}|u|^{\frac{2(d+1)}{d-1}}dx.
	\end{align*}
	
	The combined NLS can be understood as the perturbation of the standard focusing energy-critical NLS:
	\begin{align}\tag{NLS}\label{NLS-critical}
		\begin{cases}
			i\partial_tu+\Delta u=-|u|^{\frac{4}{d-2}}u, &(t,x)\in\Bbb{R}\times\R^d,\\
			u(0,x)=u_0\in\dot{H}^1(\R^d)
		\end{cases}
	\end{align}
	with associated mass and energy
	\begin{gather*}
		M(u)=\int_{\R^d}|u|^2dx,\\
		E_0^c(u)=\frac{1}{2}\int_{\R^d}|\nabla u|^2dx-\frac{d-2}{2d}\int_{\R^d}|u|^{\frac{2d}{d-2}}dx.
	\end{gather*}
	The equation $\eqref{NLS-critical}$ also has a radial positive static solution which satisfies 
	\begin{align*}
		-\Delta W_0=|W_0|^{\frac{4}{d-2}}W_0.
	\end{align*}
	The  explicit solution of \eqref{W_a}  is given by
	\begin{align*}
		W_0(x) :=    [d(d-2)]^{\frac{d-2}{4}} \biggl( \frac{ 1 }{ 1+|x|^{2} }\biggr)^{\frac{d-2}{2}}.
	\end{align*}
	We call $W_0$ as the \emph{ground state} related to \eqref{NLS-critical}. 
	
	\subsection{History}
	\subsubsection{Sub-threshold scattering/blow-up dichotomy}
	In this part, we will survey the related research in the study of the scattering/blow-up dichotomy  of the focusing energy-critical NLS with or without potential and the focusing-defocusing combined NLS below the energy threshold.

	For equation $\eqref{NLS-critical}$(i.e. \eqref{NLS-Potential} with $a=0$), the global well-posedness and scattering/blow-up dichotomy  below the energy threshold $E_0^c(W_0)$ is highly
	developed.  Kenig and Merle \cite{KenigMerle2006} developed the concentration-compactness argument \cite{Keraani1} to establish the scattering/blow-up dichotomy of radial solutions in dimension $d\in\{3,4,5\}$. Later, Killip and Visan \cite{KiiVisan2008} used the double-Duhamel trick to remove the radial assumption of scattering theory in higher dimensions $d\geqslant5$. However, their method fails when $d=3,4$ because of the weaker dispersive estimates. Inspired by the work on solving the mass-critical problem \cite{D2,D3,D4}, Dodson\cite{D5}  overcame the logarithmic loss in double-Duhamel argument and prove the scattering result for non-radial $\dot{H}^1$ initial data  in four dimension.   But, the global well-posedness and scattering for $\eqref{NLS-critical}$ for non-radial  $\dot{H}^1$ initial data in $\Bbb{R}^3$ is still  open. We summary  above to the following theorem(see also in Figure 1):
	\begin{theorem}[Scattering/blow-up dichotomy for $\eqref{NLS-critical}$,\cite{D5,KenigMerle2006,KiiVisan2008}]\label{thm-NLS-critical}Let $d\geqslant3$. Suppose that $u_0\in \dot{H}^1(\Bbb{R}^d)$ and $u_0$ is radial if $d=3$. Let $u$ be the corresponding maximal-lifespan solution to $\eqref{NLS-critical}$ with $u|_{t=0}=u_0$.
	\begin{enumerate}
	\item {\bf Scattering part:}
		If $u_0$ obeys 
		\begin{align}\label{condition1}
			E_0^c(u_0)<E_0^c(W_0),\quad \Vert u_0\Vert_{\dot{H}^1(\R^d)}<\Vert W_0\Vert_{\dot{H}^1(\R^d)},
		\end{align}
	 then $u$ is global. Moreover, $u\in L_{t,x}^{\frac{2(d+2)}{d-2}}(\Bbb{R}\times\R^{d})$ and scatters in $\dot{H}^1(\R^d)$ when $t\to\pm\infty$, i.e. there exist $u_\pm\in \dot{H}^1(\R^d)$ such that
	 \begin{align*}
	 	\lim_{t\to\pm\infty}\big\|u-e^{it\Delta}u_\pm\big\|_{\dot{H}^1(\R^d)}=0.
	 \end{align*} 
	\item {\bf Blow-up part:} If $u_0\in H^1(\R^d)$ is radial and  obeys 
		\begin{align}\label{equ:condlarge}
			E_0^c(u_0)<E_0^c(W_0),\quad \Vert u_0\Vert_{\dot{H}^1(\R^d)}>\Vert W_0\Vert_{\dot{H}^1(\R^d)},
		\end{align}
		then, the solution $u$ blows up in finite time in both time direction.
	\end{enumerate}
	\end{theorem}
	
\begin{figure}
\begin{center}
 \begin{tikzpicture}[scale=1,domain=0:3]\label{figure1} 

\draw[->] (-0.2,0) -- (6.6,0) node[anchor=north] {$\eta^2$};
\draw[->] (0,-1) -- (0,3.5)  node[anchor=east] {$y$};

\draw (-0.2,0) node[anchor=north] {O}
       (2.4,0) node[anchor=north] {$1$}
       (0.8,0) node[anchor=north] {$\frac2d$}
       (5.2,0) node[anchor=north] {$\sqrt{\frac{d}{d-2}}$}
       (0, 2) node[anchor=east] {$\frac{E_0^c(u)}{E_0^c(W_0)}$}
        (0, 2.4) node[anchor=east] {$1$}
        (0.65,2.7) node[anchor=north] {A}
        (2.4,2.9) node[anchor=north] {B}
         (0.8,3.6) node[anchor=north] {C}
        (5.7,0.4) node[anchor=north] {D}
         (5.7,2.7) node[anchor=north] {E}
        (1,2.1) node[anchor=north] {{\small sca}}
        (4.8,2.1) node[anchor=north] {{\small blow-up}}
        (0.5,3.2) node[anchor=north] {{\small forbid}}
         (2.5,1.5) node[anchor=north] {{\small forbidden}};

\draw[thick] (0,0) -- (1.1,3.3) node[ right] {$y=\tfrac{d}2\eta^2$}
        (0.8,-0.03) -- (0.8,0.03)
        (2.4,-0.03) -- (2.4,0.03)
        (5.5,-0.03) -- (5.5,0.03)
        (-0.03, 2.4) -- (0.03, 2.4)
      (0.8,2.4)-- (5.5,2.4)
     (5.5,0) -- (5.5,3.5) ;



\draw[red,thick] (0,0) parabola bend (2.4,2.4) (5.8,-0.5)
        node[below right] {$y=\tfrac{d}2\eta^2-\tfrac{d-2}2\eta^6$};






\end{tikzpicture}
\end{center}
\caption[A plot]{ A plot of $\tfrac{E_0^c(u)}{E_0^c(W_0)}$ versus $\eta^2$ with $\eta(t)=\frac{\|u(t)\|_{\dot{H}^1}}{\|W_0\|_{\dot{H}^1}}$. The area to 
the left of line OAC and inside region OBD are excluded by 
$$\tfrac{d}2\eta^2-\tfrac{d-2}2\eta^6\leq \tfrac{E_0^c(u)}{E_0^c(W_0)}\leq \tfrac{d}2\eta^2,$$
from sharp Sobolev embedding.
The region inside OAB corresponds to \eqref{condition1} (solutions scatter). The region BDE corresponds to \eqref{equ:condlarge}
(blow-up in finite time).}

\end{figure}

	For the standard combined NLS, i.e. $\eqref{Combine-NLS}$, 
	 Miao-Xu-Zhao\cite{MXZ2013} first characterized the scattering and blow-up regions for focusing energy-critical NLS with $\dot{H}^\frac{1}{2}$ perturbation where the $H^1$-initial data is radial and   energy is below $E_0^c(W_0)$  by the variation method and concentration-compactness argument
	 in dimension three. Xu-Yang\cite{XuYang2016} extended the result to  any mass-supercritical and energy-subcritical perturbation, i.e. $\frac{4}{3}<p<5$.  Later, Miao-Xu-Zhao\cite{MXZ-high} and Miao-Zhao-Zheng\cite{MiaoZhaoZheng-CVPDE} gives the  scattering/blow-up dichotomy for general $\dot{H}^1$ initial values in four and higher dimensions. But the  scattering/blow-up dichotomy for general $\dot{H}^1$ initial values in three dimension is still open. In general, we have:
	\begin{theorem}[Scattering/blow-up dichotomy for for $\eqref{Combine-NLS}$,\cite{MXZ2013,MXZ-high,MiaoZhaoZheng-CVPDE,XuYang2016}]\label{thm-Combine-NLS}Let $d\geqslant3$. Suppose that $u_0\in H^1(\Bbb{R}^d)$ and $u_0$ is radial when $d=3$. Let $u$ be the corresponding maximal-lifespan solution to $\eqref{NLS-critical}$ with $u|_{t=0}=u_0$.
	\begin{enumerate}
	\item {\bf Scattering part:}
		If $u_0$ obeys 
		\begin{align}
			E_0(u_0)<E_0^c(W_0),\quad \Vert u_0\Vert_{\dot{H}^1(\R^d)}<\Vert W_0\Vert_{\dot{H}^1(\R^d)},\label{condition1}
		\end{align}
		then $u$ is global. Moreover, $u\in L_{t,x}^{\frac{2(d+2)}{d-2}}(\Bbb{R}\times\R^{d})$ and  scatters in $\dot{H}^1(\R^d)$ when $t\to\pm\infty$, i.e. there exist $u_\pm\in {H}^1(\R^d)$ such that
		\begin{align*}
			\lim_{t\to\pm\infty}\big\|u-e^{it\Delta}u_\pm\big\|_{{H}^1(\R^d)}=0.
		\end{align*} 
		\item {\bf Blow-up part:} If $u_0$ is radial and  obeys 
		\begin{align}\label{equ:condlarge2}
			E_0^c(u_0)<E_0^c(W_0),\quad \Vert u_0\Vert_{\dot{H}^1(\R^d)}>\Vert W_0\Vert_{\dot{H}^1(\R^d)},
		\end{align}
		then, the solution $u$ blows up in finite time in both time direction.
	\end{enumerate}
	\end{theorem} 
	
	Next, we recall the result of the energy-critical NLS with inverse-square potential, i.e. \eqref{NLS-Potential} with $a<0$. 
	Killip-Visan-Miao-Zhang-Zheng\cite{KillipMiaVisanZhangZheng} first studied the global well-posedness and scattering of defocusing energy-critical NLS with inverse-square potential
	\begin{align}\label{DNLS-potential}
		\begin{cases}
			i\partial_tu-\mathcal{L}_a u=|u|^{4}u, &(t,x)\in\Bbb{R}\times\R^3,\\
			u(0,x)=u_0\in\dot{H}_a^1(\R^3)
		\end{cases}  \tag{DNLS$_a$}
	\end{align} 
	by  developing the linear profile decomposition in $\dot{H}_a^1$ and  using Kenig-Merle's concentration-compactness argument.  They also carry out the variational analysis needed to treat the equation $\eqref{NLS-Potential}$. Moreover, they give the proof of finite-time blow-up when energy below that of the ground state $W_a$ and $\|u_0\|_{\dot{H}_a^1}>\|W_a\|_{\dot{H}_a^1}$.
	For the focusing energy-critical NLS with inverse-square potential, i.e. \eqref{NLS-Potential},  Yang\cite{YANG2020124006} proved the global well-posedness and scattering for Cauchy problem $\eqref{NLS-Potential}$ with energy below that of ground state and $\|u_0\|_{\dot{H}_a^1}<\|W_a\|_{\dot{H}_a^1}$ when $3\leqslant d\leqslant6$.\footnote[1]{When $d\geq7$, the stability lemma is difficult to establish due to the lack of the $L^\infty-L^1$ dispersive estimate. We also remark that it is possible to derive the stablity lemma by using the $L^p-L^{p'}$ dispersive estimate with $p>1$, see Miao-Su-Zheng\cite{Miao-Su-Zheng} for  details. This is also the unique reason why we restrict the dimension  $d\in\{3,4,5\}$.  } It is worth to note that if $d=3$, they also assume $u_0$ is radial.\footnote[2]{When $d=3$, the scattering result of the \eqref{NLS-critical} for general initial data is
		still open. This is the unique reason why he restrict the radial assumption in three-dimension.} 
	We summarize Yang's works in the following theorem:
	\begin{theorem}[Global well-posedness and scattering for $\eqref{NLS-Potential}$,\cite{KillipMiaVisanZhangZheng,YANG2020124006}]\label{thm-NLS-Potential}Let $3\leqslant d\leqslant6$ and $0>a>-\frac{(d-2)^2}{4}+\left(\frac{d-2}{d+2}\right)^2$. Suppose that $u_0\in \dot{H}_a^1(\Bbb{R}^d)$ and $u_0$ is radial when $d=3$.  Let $u$ be the corresponding maximal-lifespan solution to $\eqref{NLS-Potential}$ with $u|_{t=0}=u_0$.
		If $u_0$ obeys 
		\begin{align}
			E_a^c(u_0)<E_a^c(W_a),\quad \Vert u_0\Vert_{\dot{H}_a^1(\R^d)}<\Vert W_a\Vert_{\dot{H}_a^1(\R^d)},\label{condition1}
		\end{align}
	then $u$ is global. Moreover, $u\in L_{t,x}^{\frac{2(d+2)}{d-2}}(\Bbb{R}\times\R^{d})$ and  scatters in $\dot{H}_a^1(\R^d)$ when $t\to\pm\infty$, i.e. there exist $u_\pm\in \dot{H}_a^1(\R^d)$ such that
	\begin{align*}
		\lim_{t\to\pm\infty}\big\|u-e^{-it\mathcal{L}_a}u_\pm\big\|_{\dot{H}_a^1(\R^d)}=0.
	\end{align*} 
	\end{theorem}  
	\begin{remark}
		In fact, Yang\cite{YANG2020124006} also treated the case of $a>0$ in \cite{YANG2020124006}. But the energy threshold is $E_0^c(W_0)$.
	\end{remark}

	\subsubsection{Dynamic classification of threshold solution}
	Regarding the threshold dynamics, i.e. the energy of $u_0$ is equal to that of ground state. Duckyaerts-Merle \cite{DuyckaMerle2009} classified all possible behaviors for $\eqref{NLS-critical}$ with radial $\dot{H}^1$ data in $d=3,4,5$. Besides scattering and blowing up, the authors also constructed two special solutions $W^\pm$ that are the ``heteroclinic orbits" connecting the ground state to blowing up/scattering. It is worth mentioning that Su and Zhao \cite{Zhao} gave the sub-threshold dynamics classification with general $\dot{H}^1$ initial values for $d\geq 5$. We summary their results into the following two theorems.
	\begin{theorem}[Existence of heteroclinic orbits for $\eqref{NLS-critical}$,\cite{DuyckaMerle2009,LiZhang09}]
		\label{th.exist}
		Let $d\geqslant3$. There exist radial solutions $W_0^-$ with maximal life span $(-\infty,+\infty)$ and $W_0^+$ with maximal life span $(T_+(W_0^+),+\infty)$ of \eqref{NLS-critical}   such that
		\begin{gather}
			\label{ex.energy}
			E_0^c(W_0)=E_0^c(W_0^+)=E_0^c(W_0^-),\lim_{t\rightarrow +\infty} W_0^{\pm}(t)=W_0 \text{ in } \dot{H}^1,\\
			\label{ex.sub}
			\big\|W_0^{-}\big\|_{\dot{H}^1}<\|W_0\|_{\dot{H}^1},\quad   \|W_0^-\|_{L_{t,x}^\frac{2(d+2)}{d-2}((-\infty,0])}<\infty,\\
			\label{ex.super}
			\big\|W_0^{+}\big\|_{\dot{H}^1}>\|W_0\|_{\dot{H}^1},\text{ and, if }d\geqslant5,\;T_-(W_0^+)<+\infty.
		\end{gather}
	\end{theorem}
	\begin{theorem}[Threshold dynamics of $\eqref{NLS-critical}$, \cite{DuyckaMerle2009,LiZhang09,Zhao}]
		\label{th.classif}
		Let $d\geqslant3$ and $u_0\in \dot{H}^1(\R^d)$  such that 
		\begin{equation}
			\label{threshold}
			E_0^c(u_0)=E_0^c(W_0)
		\end{equation}
		Suppose that $u$ is the solution of \eqref{NLS-critical} with initial condition $u_0$ 
		and $I$ is the maximal lifespan of $u$. Then the followings hold:
		\begin{enumerate}
			\item \label{theorem.sub} If $ \| u_0\|_{\dot{H}^1}<\| W_0\|_{\dot{H}^1}$ and  $u_0$ is radial  when $d=3,4$ then $I=\R$. Furthermore,  either 
			$$ u(t,x)\in\bigg\{\frac{e^{i\theta_0}}{\lambda_0^{(d-2)/2}}W_0^-\Big(\frac{t_0+t}{\lambda_0^2},\frac{x-x_0}{\lambda_0}\Big) ,\quad \frac{e^{i\theta_0}}{\lambda_0^{(d-2)/2}}\overline{W_0^-}\Big(\frac{t_0-t}{\lambda_0^2},\frac{x-x_0}{\lambda_0}\Big)\bigg\}$$
			for some parameters $\theta_{0}\in\Bbb{S}^1$, $t_0\in\R$, $x_0\in \R^d$ and $\lambda_0\in\R\setminus\{0\}$,
			or $u$  scatters in $\dot{H}^1(\R^d)$ when $t\to\pm\infty$, i.e. there exist $u_\pm\in \dot{H}^1(\R^d)$ such that
			\begin{align*}
				\lim_{t\to\pm\infty}\big\|u-e^{it\Delta}u_\pm\big\|_{\dot{H}^1(\R^d)}=0.
			\end{align*} 
			\item \label{theorem.crit} If $ \| u_0\|_{\dot{H}^1}=\| W_0\|_{\dot{H}^1}$ then $$ u(t,x)=\frac{e^{i\theta_0}}{\lambda_0^{(d-2)/2}}W_0\Big(\frac{x-x_0}{\lambda_0}\Big)$$
			for some parameters $\theta_{0}\in\Bbb{S}^1$,  $x_0\in \R^d$ and $\lambda_0\in\R\setminus\{0\}$,
			\item \label{theorem.super} If $ \| u_0\|_{\dot{H}^1}>\| W_0\|_{\dot{H}^1}$, and $u_0\in L^2$ is radial, then either 
			$$ u(t,x)\in\bigg\{\frac{e^{i\theta_0}}{\lambda_0^{(d-2)/2}}W_0^+\Big(\frac{t_0+t}{\lambda_0^2},\frac{x}{\lambda_0}\Big) ,\quad \frac{e^{i\theta_0}}{\lambda_0^{(d-2)/2}}\overline{W_0^+}\Big(\frac{t_0-t}{\lambda_0^2},\frac{x}{\lambda_0}\Big)\bigg\}$$
			for some parameters $\theta_{0}\in\Bbb{S}^1$, $t_0\in\R$ and $\lambda_0\in\R\setminus\{0\}$,
			or $I$ is finite.
		\end{enumerate}
	\end{theorem}

	For the threshold dynamics of  equation \eqref{NLS-Potential}, Yang-Zeng-Zhang\cite{KYang-SIAM} developed the modulation analysis near the ground state $W_a$ and then using the concentration-compactness argument to classificate the dynamics  with radial data in $d=3$ and non-radial data in $d=4,5$. 
	Now we invoke the result of \cite{KYang-SIAM} as the following theorem:\begin{theorem}[Existence of heteroclinic orbits for $\eqref{NLS-Potential}$,\cite{KYang-SIAM}]
		\label{th.exist}
		Let $d\in\{3,4,5\}$ and $0>a>-\frac{(d-2)^2}{4}+\left(\frac{2(d-2)}{d+2}\right)^2$. There exist radial solutions $W_a^-$ with maximal life span $(-\infty,+\infty)$ and $W_a^+$ with maximal life span $(T_+(W_a^+),+\infty)$ of \eqref{NLS-Potential}   such that
		\begin{gather}
			\label{threshold-a}
			E_a^c(W_a)=E_a^c(W_a^\pm),\quad\lim_{t\rightarrow +\infty} W_a^{\pm}(t)=W_a \text{ in } \dot{H}_a^1(\R^d),\\
			\label{ex.sub-a}
			\big\|W_a^{-}\big\|_{\dot{H}_a^1}<\|W_a\|_{\dot{H}_a^1},\quad   \|W_a^-\|_{L_{t,x}^\frac{2(d+2)}{d-2}((-\infty,0])}<\infty,\\
			\label{ex.super}
			\big\|W_a^{+}\big\|_{\dot{H}_a^1}>\|W_a\|_{\dot{H}_a^1},\text{ and, if  }d=5,\;T_-(W_a^+)<+\infty.
		\end{gather}
	\end{theorem}
	\begin{theorem}[Threshold dynamics of $\eqref{NLS-Potential}$, \cite{KYang-SIAM}]
		\label{th.classif}
		Let $d\in\{3,4,5\}$ and $0>a>-\frac{(d-2)^2}{4}+\left(\frac{2(d-2)}{d+2}\right)^2$. Suppose that $u_0\in \dot{H}^1_a(\R^d)$  such that 
		\begin{equation}
			\label{threshold-a}
			E_a^c(u_0)=E_a^c(W_a)
		\end{equation}
		We assume that $u$ is the solution of \eqref{NLS-Potential} with initial data $u_0$ 
		and $I$ is the maximal lifespan of $u$. Then the following holds:
		\begin{enumerate}
			\item \label{theorem-sub-a} If $ \| u_0\|_{\dot{H}_a^1(\R^d)}< \|W_a\|_{\dot{H}_a^1(\R^d)}$ and  $u_0$ is radial  when $d=3$, then $I=\R$. Furthermore,  either 
			$$ u(t,x)\in\bigg\{\frac{e^{i\theta_0}}{\lambda_0^{(d-2)/2}}W_a^-\Big(\frac{t_0+t}{\lambda_0^2},\frac{x}{\lambda_0}\Big) ,\quad \frac{e^{i\theta_0}}{\lambda_0^{(d-2)/2}}\overline{W_a^-}\Big(\frac{t_0-t}{\lambda_0^2},\frac{x}{\lambda_0}\Big)\bigg\}$$
			for some parameters $\theta_{0}\in\Bbb{S}^1$, $t_0\in\R$,  and $\lambda_0\in\R\setminus\{0\}$,
			or $u$  scatters in $\dot{H}^1(\R^d)$ when $t\to\pm\infty$, i.e. there exist $u_\pm\in \dot{H}_a^1(\R^d)$ such that
			\begin{align*}
				\lim_{t\to\pm\infty}\big\|u-e^{-it\mathcal{L}_a}u_\pm\big\|_{\dot{H}_a^1(\R^d)}=0.
			\end{align*}
			\item \label{theorem-crit-a} If $ \| u_0\|_{\dot{H}_a^1(\R^d)}= \|W_a\|_{\dot{H}_a^1(\R^d)}$, then $$ u(t,x)=\frac{e^{i\theta_0}}{\lambda_0^{(d-2)/2}}W_a\Big(\frac{x}{\lambda_0}\Big)$$
			for some parameters $\theta_{0}\in\Bbb{S}^1$ and $\lambda_0\in\R\setminus\{0\}$.
			\item \label{theorem.super} If $\| u_0\|_{\dot{H}_a^1(\R^d)}> \|W_a\|_{\dot{H}_a^1(\R^d)}$, and $u_0\in L^2$ is radial, then either 
			$$ u(t,x)\in\bigg\{\frac{e^{i\theta_0}}{\lambda_0^{(d-2)/2}}W_a^+\Big(\frac{t_0+t}{\lambda_0^2},\frac{x}{\lambda_0}\Big) ,\quad \frac{e^{i\theta_0}}{\lambda_0^{(d-2)/2}}\overline{W_a^+}\Big(\frac{t_0-t}{\lambda_0^2},\frac{x}{\lambda_0}\Big)\bigg\}$$
			for some parameters $\theta_{0}\in\Bbb{S}^1$, $t_0\in\R$ and $\lambda_0\in\R\setminus\{0\}$,
			or $I$ is finite.
		\end{enumerate}
	\end{theorem}

	For the threshold dynamics of the energy-critical focusing-defocusing perturbed equation, i.e. $E_0(u_0)=E_0^c(W_0)$, Ardila-Murphy-Zheng \cite{AJZ} solved the scattering/blow-up dichotomy for radial $H^1$ initial data in three dimension. Compared to the energy-critical threshold problem, the heteroclinic orbits no longer exist. However, as pointed out by \cite{DuyckaMerle2009}, if one considers the non-homogeneous space $H^1$, heteroclinic orbits do not exist naturally. 
	\begin{theorem}[Threshold dynamics of $\eqref{Combine-NLS}$,\cite{AJZ}]\label{Th2} 
		 Let  $u_{0}\in H^{1}(\R^{3})$ be radial. Suppose that $u$ is the corresponding maximal-lifespan solution to \eqref{Combine-NLS} with $u|_{t=0}=u_0$.
		\begin{enumerate}
			\item If $u_{0}$ obeys
			\begin{equation*}
				E_0(u_{0})=E_0^{c}(W_0) \quad \text{and} \quad \Vert u_0\Vert_{\dot{H}^1(\R^3)}<\Vert W_0\Vert_{\dot{H}^1(\R^3)},
			\end{equation*}
			 then $u$ is global, $u\in L_{t,x}^{10}(\R\times\R^3)$, and  scatters in ${H}^1(\R^3)$ when $t\to\pm\infty$, i.e. there exist $u_\pm\in \dot{H}^1(\R^3)$ such that
			 \begin{align*}
			 	\lim_{t\to\pm\infty}\big\|u-e^{it\Delta}u_\pm\big\|_{{H}^1(\R^3)}=0.
			 \end{align*}
			\item If $u_{0}$ is radial and obeys
			\begin{equation*}\label{Blow-up}
				E_0(u_{0})=E_0^{c}(W_0) \quad \text{and} \quad \Vert u_0\Vert_{\dot{H}^1(\R^3)}>\Vert W_0\Vert_{\dot{H}^1(\R^3)},
			\end{equation*}
			then $u$ blows up in both time directions in finite time.
		\end{enumerate}
	\end{theorem}
	\begin{remark}
We expect that Theorem \ref{Th2} holds in higher-dimensions $d\geqslant4$ without radial assumption. We shall study this issue in the future work.  
	\end{remark}

	In this paper, we  extend Theorem \ref{thm-Combine-NLS} and \ref{Th2} to the case of inverse-square potential. Compared to  these theorems(cf. Theorem \ref{thm-Combine-NLS} and \ref{Th2}), the new ingredient  in our paper is that the additional potential breaks down the translation-invariance with respect to the spatial variable. It needs more delicate analysis in the construction of nonlinear profile decomposition.

	\subsection{Main results and discussion}
	In this section, we present our main results and give the outline of the proof. 
	
	Our first result is to give the scattering/blow-up dichotomy under the energy threshold, which extend Theorem \ref{thm-Combine-NLS}($a=0$) to the case with inverse-square potential. We state it as the following theorem.
	
	\begin{theorem}[Sub-threshold solution]\label{sub-threshold} Let $d\in\{3,4,5\}$ and $-\tfrac{(d-2)^2}{4}+\big(\tfrac{d-2}{d+2}\big)^2<a<0$. Suppose that $u_0\in H_a^1(\Bbb{R}^d)$ and let $u$ be the corresponding maximal-lifespan solution to $\eqref{NLS}$ with $u|_{t=0}=u_0$.
		\begin{enumerate}
			\item[$(i)$] {\bf Scattering part:} If $u_0$ obeys 
			\begin{align}
				E_a(u_0)<E_a^c(W_a),\quad \Vert u_0\Vert_{\dot{H}_a^1(\R^d)}<\Vert W_a\Vert_{\dot{H}_a^1(\R^d)},\label{condition1}
			\end{align}
			and we further assume that $u_0$ is radial when $d=3$,
		then $u$ is global. Moreover, $u\in L_{t,x}^{\frac{2(d+2)}{d-2}}(\Bbb{R}\times\R^{d})$ and scatters in ${H}^1(\R^d)$ when $t\to\pm\infty$, i.e. there exist $u_\pm\in \dot{H}_a^1(\R^d)$ such that
		\begin{align*}
			\lim_{t\to\pm\infty}\big\|u-e^{-it\mathcal{L}_a}u_\pm\big\|_{{H}_a^1(\R^d)}=0.
		\end{align*}
			\item[$(ii)$] {\bf Blow-up part:} If $u_0$ is radial and obeys 
			\begin{align}
				E_a(u_0)<E_a^c(W_a),\quad\|u_0\|_{\dot{H}_a^1}>\|W_a\|_{\dot{H}_a^1},\label{condition2}
			\end{align}
			then $u$ blow up in both time directions in finite time.
		\end{enumerate}
	\end{theorem}
	Our second result is to  give the scattering/blow-up dichotomy at the energy threshold, which extend  Theorem \ref{Th2}($a=0$) to the case with inverse-square potential.  We state it as the following theorem.
	\begin{theorem}[Threshold solution]\label{Threshold}
		Let $d\in\{3,4,5\}$ and $-\frac{(d-2)^2}{4}+\left(\frac{2(d-2)}{d+2}\right)^2<a<0$. Suppose that $u_0\in H_a^1(\R^d)$ and let $u$ be the corresponding maximal-lifespan solution to $\eqref{NLS}$ with $u|_{t=0}=u_0$.
		\begin{enumerate}
			\item[$(i)$] If $u_0$ obeys 
			\begin{align}\label{energy-threshold}
				E_a(u_0)=E_a^c(W_a),\quad\Vert u_0\Vert_{\dot{H}_a^1}<\Vert W_a\Vert_{\dot{H}_a^1},
			\end{align}
			and we further assume that $u_0$ is radial when $d=3$,   then $u$ is global. Moreover, $u\in L_{t,x}^{\frac{2(d+2)}{d-2}}(\Bbb{R}\times\R^{d})$  and scatters in ${H}^1(\R^d)$ when $t\to\pm\infty$, i.e. there exist $u_\pm\in \dot{H}_a^1(\R^d)$ such that
			\begin{align*}
				\lim_{t\to\pm\infty}\big\|u-e^{-it\mathcal{L}_a}u_\pm\big\|_{{H}_a^1(\R^d)}=0.
			\end{align*}
			\item[$(ii)$] If $u_0$ is radial and obeys 
			\begin{align}\label{blow-up-thre}
				E_a(u_0)=E_a^c(W_a),\quad\Vert u_0\Vert_{\dot{H}_a^1}>\Vert W_a\Vert_{\dot{H}_a^1},
			\end{align}
			then $u$ blows up in both time direction in finite time.
		\end{enumerate}
	\end{theorem}
	\begin{remark}
	Compared to  Theorem \ref{sub-threshold}, we further  need that $-\left(\frac{d-2}{2}\right)^2+\left(\frac{2(d-2)}{d+2}\right)^2<a<0$ in Theorem \ref{Threshold} where such  restriction  is to ensure better regularity of the ground state $W_a(x)$.  We refer to \cite{KYang-SIAM} for details.
	\end{remark}  
	
	In the remainder of this introduction, we outline the structure of the paper and the arguments used to prove Theorem \ref{sub-threshold} and \ref{Threshold}.
	
	In Section \ref{S1:preli}, we  revisit some  harmonic analysis tools associated to the operator $\mathcal{L}_a$, including the followings: the equivalence of Sobolev spaces and the Littlewood-Paley theory along with Bernstein inequalities; Strichartz estimates; Kato's local smoothing estimates. Next, we study the local theory  for  equation $\eqref{NLS}$, which contains the local well-posedness, stability lemma, small-data scattering and the persistence of regularity.  We also show the linear profile decomposition for bounded sequence $\{f_n\}\subset H_a^1(\R^d)$. We end up this section by establishing the local virial identity, which is important in both proofs of Theorem \ref{sub-threshold} and \ref{Threshold}.
	
	In Section \ref{sec:VarationalGNI}, we first collect the sharp Sobolev embedding associated to $\mathcal{L}_a$. By using this sharp inequality, we establish the energy trapping lemma, which is crucial in the proof of Theorem \ref{sub-threshold}. We  associate the energy threshold $E_a^c(W_a)$ to a following minimization problem(see $\eqref{minimization}$ for details). The scattering and blow-up regimes can be rewritten to the two regions associated to this minimization problem. Then we establish the energy trapping lemma for these two regions. 
	
	In Section \ref{sec:Modulation}, we give the modulation analysis around the ground state $W_a$, which is essential in the proof of Theorem \ref{Threshold}(1). We emphasize that the mass conservation plays an important role in establishing the  modulation analysis. 
	
	In Section \ref{sec:blowup-sub} and Section \ref{sec:blowup-threshold}, we give the proof of the blow up part in Theorem \ref{sub-threshold} and \ref{Threshold} throughout the local virial argument. For the threshold case, we need the additional modulation analysis. These tools are already established in  previous sections.
	
	In Section~\ref{sec:minimal-blowup},  we construct solutions to \eqref{NLS} associated to profiles $\phi_n$ living either at small length scales (i.e. in the regime $\lambda_n\to 0$) or far from the origin relative to their length scale (i.e. in the regime $|\tfrac{x_{n}}{\lambda_{n}}|\to \infty$) or both. These solutions are also called the nonlinear profiles.  The  challenge lies in the fact that the translation and scaling symmetries in \eqref{NLS} are broken by the potential and the  nonlinearity, respectively.  In particular, we must consider several limiting regimes and use approximation by the suitable underlying model in each case.   Using the nonlinear profiles, we can prove that the solution $u(t)$ is pre-compact in $H_a^1(\R^d)$. In Section \ref{sec:preclude-sub}, we preclude the possibility of minimal blow-up solution. We use the local virial argument to preclude the infinite-time blow-up and the continuous method to preclude the possibility of finite-time blow-up.  
	
	In Section~\ref{sec:threshold-solution}, we demonstrate that the failure of Theorem \ref{Threshold}(i) implies the existence of a non-scattering solution $u$ to $\eqref{NLS}$ satisfying $\eqref{SN}$ such that 
	\begin{align*}
		\bigg\{\lambda(t)^{-\frac{d-2}{2}}u\left(t,\frac{x}{\lambda(t)}\right):t\in[0,T_{max})\bigg\}\mbox{ is pre-compact in }\dot{H}_a^1(\R^d),
	\end{align*}
	where $\lambda(t)\in \R\setminus\{0\}$ is the scaling parameter.
	In Section~\ref{sec:preclude-threshold}, we preclude the possibility of finite-time blow-up and infinite-time blow-up separately. 
	In the case of finite-time blowup, we find that the solution must move to arbitrarily small spatial scales as one approaches the finite blowup time via the local virial identity.  In this regime, we are able to prove that the $L^2$ norm of the solution is identically zero, which leads to a contradiction.  
	
	In the case of infinite-time blowup, we briefly give the sketch of the proof here. Using the standard localized virial argument, we can find a sequence of times along which the solution approaches the orbit of the ground state.  This is made precise using the quantity 
	\[
	\delta(t) = \big| \| W_a\|_{\dot{H}_a^1}^2 - \| u(t)\|_{\dot{H}_a^1}^2\big|. 
	\]
	We can further show that if $\delta(t_n)\to 0$ for some $t_n\to \infty$, then we have $\lambda(t_n)\rightarrow\infty$, which implies that the solution is in small spacial scales.  However, the modulation analysis indicates that the difference in spatial scale may be controlled by the integral of $\delta(t)$.  We  control such time integrals by the value of $\delta(\cdot)$ by using the  virial estimates.  It follows that the solution cannot  move to arbitrarily small scales, and hence we obtain a contradiction.   

	\subsection{Notations}
	We write $A\lesssim B$ or $A=O(B)$ when $A\leq CB$ for some $C>0$. If  $A \lesssim B \lesssim A$, we write $A\sim B$.  We also write $A\ll1$ in the sense that there exists $\varepsilon>0$ sufficiently small such that $A<\varepsilon$.
	\subsection*{Acknowledgements}

	\section{Preliminaries}\label{S1:preli} 
	In this section, we collect the harmonic analysis tools associated with  $\mathcal{L}_a$, the local well-posedness/stablity result, the linear profile decomposition of the bounded sequence in ${H}_a^1(\R^d)$ and local virial identities.
	
	For a function 
	$u:I\times \R^{d}\rightarrow \C$ we use the notation
	\[ 
	\|  u \|_{L_{t}^{q}L^{r}_{x}(I\times \R^{d})}=\|  \|u(t) \|_{L^{r}_{x}(\R^{d})}  \|_{L^{q}_{t}(I)}
	\]
	with $1\leq q\leq r\leq\infty$. When $q=r$ we abbreviate $L_{t}^{q}L^{r}_{x}$ by $L_{t,x}^{q}$.
	
	We define the Sobolev spaces associated with $\L_{a}$ via
	\[
	\|f\|_{\dot{H}^{s, r}_{a}(\R^{d})}=\|(\L_{a})^{\frac{s}{2}}f\|_{L^{r}_{x}(\R^{d})}
	\quad \text{and}\quad 
	\|f\|_{{H}^{s,r}_{a}(\R^{d})}=\|(1+\L_{a})^{\frac{s}{2}}f\|_{L^{r}_{x}(\R^{d})}.
	\]
	We abbreviate $\dot{H}^{s}_{a}(\R^{d})=\dot{H}^{s,2}_{a}(\R^{d})$ and ${H}^{s}_{a}(\R^{d})={H}^{s,2}_{a}(\R^{d})$.
	
	Given a function $u\in L^2(\Bbb{R}^d)$, we defined the mass quantity
	\[
	M(u)=\int_{\R^{d}}|u|^{2}\,dx.
	\] 
	
	If $u\in \dot{H}_a^1(\Bbb{R}^d)$, we define four energy quantities 
	\begin{align}
		&E_0(u)=\int_{\R^{d}}\bigg(\frac{1}{2}|\nabla u|^{2}
		+\frac{d-1}{2d+2}|u|^{\frac{2d+2}{d-1}}-\frac{d-2}{2d}|u|^{\frac{2d}{d-2}}\bigg)\,dx,\label{E_0}\\
		&E_{a}(u)=\int_{\R^{d}}\bigg(\frac{1}{2}|\nabla u|^{2}+\frac{a}{2|x|^{2}}|u|^{2}
		+\frac{d-1}{2d+2}|u|^{\frac{2d+2}{d-1}}-\frac{d-2}{2d}|u|^{\frac{2d}{d-2}}\bigg)\,dx,\label{E_a}\\
		&E_a^c(u)=\int_{\R^{d}}\bigg(\frac{1}{2}|\nabla u|^{2}+\frac{a}{2|x|^{2}}|u|^{2}-\frac{d-2}{2d}|u|^{\frac{2d}{d-2}}\bigg)\,dx,\label{E_a^c}\\
		&E_0^c(u)=\int_{\R^{d}}\bigg(\frac{1}{2}|\nabla u|^{2}-\frac{d-2}{2d}|u|^{\frac{2d}{d-2}}\bigg)\,dx.\label{E_0^c}
	\end{align}
	As mentioned in the introduction, $\eqref{E_a^c}$ and $\eqref{E_0^c}$ are conserved energies for focusing energy-critical problem $\eqref{NLS-Potential}$ and $\eqref{NLS-critical}$ while  $\eqref{E_0}$ and $\eqref{E_a}$ are conserved energies for the focusing-defocusing NLS with combined-power nonlinearity $\eqref{Combine-NLS}$ and $\eqref{NLS}$.

	We summarize the equivalence of Sobolev spaces in the following lemma, which was first proved in \cite{KMVZZ2018}.
	\begin{lemma}\label{EquiSobolev}
		Fix $d\geq 3$, $a\geq -\( \frac{d-2}{2}\)^{2}$ and $0<s<2$. We let 
		\begin{align*}
			\sigma=\frac{d-2}{2}-\big[\big(\frac{d-2}{2}\big)^2+a\big]^\frac{1}{2}.
		\end{align*} If $1<p<\infty$ satisfies 
		$\frac{\sigma+s}{d}<\frac{1}{p}<\min\left\{1, \frac{d-\sigma}{d}\right\}$, then 
		\[
		\|(-\Delta)^{\frac{s}{2}}f\|_{L_{x}^{p}}\lesssim_{d,p,s}\|\L_{a}^{\frac{s}{2}}f\|_{L_{x}^{p}}
		\quad \text{for $f\in C^{\infty}_{c}(\R^{d}\setminus\left\{{0}\right\})$}.
		\]
		If $\max\left\{\frac{s}{d}, \frac{\sigma}{d}\right\}<\frac{1}{p}<\min\left\{1, \frac{d-\sigma}{d}\right\}$, then 
		\[
		\|\L_{a}^{\frac{s}{2}}f\|_{L_{x}^{p}}\lesssim_{d,p,s}\|(-\Delta)^{\frac{s}{2}}f\|_{L_{x}^{p}}
		\quad \text{for  $f\in C^{\infty}_{c}(\R^{d}\setminus\left\{{0}\right\})$}.
		\]
		Especially, if $a>-\frac{(d-2)^2}{4}+\left(\frac{d-2}{d+2}\right)^2$, then 
		\[
		\|\L_{a}^{\frac{1}{2}}f\|_{L_{x}^{p}}\sim\|(-\Delta)^{\frac{1}{2}}f\|_{L_{x}^{p}}
		\quad \text{for all  $\tfrac{2d}{d+2}\leq p\leq \tfrac{2d(d+2)}{d^2+4}$}.
		\]
	\end{lemma}

	We will also need the  Littlewood--Paley theory adapted to $\mathcal{L}_a$.  Let $\phi\in C^{\infty}_{0}(\R^{d})$ be a smooth positive radial function satisfying
	$\phi(x)=1$ if $|x|\leq 1$  and $\phi(x)=0$ if $|x|\geq \frac{11}{10}$. For $N\in 2^{\Z}$, we define 
	\[
	\phi_{N}(x):=\phi(x/N)  \quad \text{and}\quad      \psi_{N}(x)=\phi_{N}(x)-\phi_{N/2}(x).
	\]
	We define the Littlewood-Paley projector as
	\begin{align*}
		f_{\leq N}&:=P^{a}_{\leq N}f:=\phi_{N}(\sqrt{\L_{a}})f, \quad f_ {N}:=P^{a}_{N}f:=\psi_{N}(\sqrt{\L_{a}})f,\\
		&\quad \text{and}\quad f_{> N}:=P^{a}_{> N}f:=(I-P^{a}_{\leq N})f
	\end{align*}
	
	The Littlewood-Paley projectors  satisfy the following Bernstein estimates.
	\begin{lemma}[Bernstein inequalities, \cite{KMVZZ2018}]\label{BIN}
		Let $s\in \R$. For $q_{0}<q\leq r<q_0^\prime:=\frac{\sigma}{d}$ and $f: \R^{d}\to \C$ we have
		\begin{align*}
			\| P^{a}_{N}f   \|_{L_{x}^{r}}&\lesssim N^{\frac{d}{q}-\frac{d}{r}}\| P^{a}_{N}f   \|_{L_{x}^{q}},\\
			\| P^{a}_{\geq N}f   \|_{L_{x}^{r}}&\lesssim N^{-s}\| \L_{a}^{\frac{s}{2}} P^{a}_{\geq N}f   \|_{L_{x}^{r}},\\
			N^{s}\| P^{a}_{N}f   \|_{L_{x}^{r}}&\sim \|\L^{\frac{s}{2}}_{a} P^{a}_{N}f   \|_{L_{x}^{r}}.
		\end{align*}
	\end{lemma}

	
	We have the following global-in-time Strichartz estimates.
	\begin{lemma}[Strichartz estimates, \cite{BPSTZ}]\label{STRICH}
		Fix $a>-\frac{(d-2)^2}{4}$. Then the solution $u$ of $(i\partial_{t}-\L_{a}) u=F$ on an interval $t_0\in I$
		satisfying
		\[\| u  \|_{L_{t}^{q}L^{r}_{x}(I\times \R^{d})} \lesssim \| u(t_{0})\|_{L^{2}_{x}(\R^{d})}+ \| F\|_{L_{t}^{\tilde{q}'}L^{\tilde{r}'}_{x}(I\times \R^{d})}, \]
		where $2\leq \tilde{q},\tilde{r}\leq\infty$ with 
		$\frac{2}{q}+\frac{d}{r}=\frac{2}{\tilde{q}}+\frac{d}{\tilde{r}}=\frac{d}{2}$.
	\end{lemma}
	\begin{remark}
		Bouclet-Mizutani\cite{Mizutani} proved the double end-point case, i.e. $q=\tilde{q}=2$.
	\end{remark}
	We also collect the following Kato's local smoothing result for the linear Schr\"odinger group $e^{-it\L_{a}}$, see \cite{KillipMiaVisanZhangZheng} for the  proof.
	\begin{lemma}[Local smoothing,\cite{KillipMiaVisanZhangZheng}]\label{LocalSmoothing}
		Let $a>-\left(\frac{d-2}{4}\right)^2$ and $w=e^{-it\mathcal{L}_a}w_0$, then we have
		\begin{align}
			&    \int_{\mathbb{R}}\int_{\mathbb{R}^d}\frac{|\nabla w(t,x)|^2}{R\langle R^{-1}x\rangle^3}+\frac{|w(t,x)|^2}{R|x|^2} dx dt\lesssim\|w_0\|_2\|\nabla w_0\|_2+R^{-1}\|w_0\|_2^2,\label{local1}\\&\int_{\mathbb{R}}\int_{|x-z|\leq R}\frac{1}{R}|\nabla w(t,x)|^{2} dx dt\lesssim\|w_{0}\|_{2}\|\nabla w_{0}\|_{2}+R^{-1}\|w_{0}\|_{2}^{2}\label{local2}
		\end{align}
		uniformly for $z\in\Bbb{R}^d$.
	\end{lemma}
	Finally, we have the following estimate of the linear flow $e^{-it\mathcal{L}_a}f$ on
	compact domains.  This estimate is used  to control the error term in an approximate solution in the construction of minimal blowup solution(see \eqref{Bound33}). 
	\begin{lemma}\label{Kato}
		Let $a>-\frac{(d-2)^2}{4}+\left(\frac{d-2}{d+2}\right)^2$ and $\phi\in \dot{H}^{1}_{a}(\R^{d})$, then we have
		\begin{align*}
			\|\nabla e^{-it\mathcal{L}_{a}} \phi\|_{L_{t,x}^2([\tau-T,\tau+T]\times\{|x-z|\leq R\})}^{2}&\lesssim T^{\frac{d^2+8d-4}{2(d+2)(3d-2)}}R^\frac{2d(d-2)}{(3d-2)(d+2)}\|\phi\|_{\dot{H}_{a}^{1}}^{\frac{2d}{3d-2}}\|e^{-it\mathcal{L}_{a}}\phi\|_{L_{t,x}^\frac{2(d+2)}{d-2}(\R\times\R^d)}^\frac{d-2}{3d-2}\\
			&\hspace{2ex}+T^{\frac{d^2+8d-4}{8(d-1)(d+2)}}R^\frac{3d^2-4d-4}{4(d-1)(d+2)}\|\phi\|_{\dot{H}_{a}^{1}}^{\frac{3d-2}{4(d-1)}}\|e^{-it\mathcal{L}_{a}}\phi\|_{L_{t,x}^\frac{2(d+2)}{d-2}(\R\times\R^d)}^\frac{d-2}{4(d-1)}.
		\end{align*}
		uniformly in $\phi$ and the parameters $R$, $T>0$, $z\in \R^{d}$ and $\tau\in\R$ and $R,T>0$.
	\end{lemma}
	\begin{proof}
		We write $$e^{-it\mathcal{L}_a}\phi=e^{-it\mathcal{L}_a}(P_{\leq N}^a\phi+P_{>N}^a\phi).$$
		First, we estimate the low frequency part. By the equivalence of Sobolev norms,  H\"older's inequality,  Bernstein's inequality and Strichartz estimate, we have
		\begin{align*}
			&\hspace{3ex}	\big\|\nabla e^{-it\mathcal{L}_a}P_{\leq N}^{a}\phi\big\|_{L_{t, x}^{2}([\tau-T,\tau+T]\times\{|x-z|\leq R\})} \\
			&\lesssim T^{\frac{d^2+8d-4}{4d(d+2)}}R^{\frac{d-2}{d+2}}\left\|\nabla e^{-it\mathcal{L}_{a}}P_{\leq N}^{a}\phi\right\|_{L_{t}^{\frac{4d(d+2)}{(d-2)^2}}L_{x}^{\frac{2d(d+2)}{d^2+4}}(\mathbb{R}\times\mathbb{R}^{d})} \\
			&\lesssim T^{\frac{d^2+8d-4}{4d(d+2)}}R^{\frac{d-2}{d+2}}N^{\frac{d-2}{2d}}\|(\mathcal{L}_a)^{\frac{d+2}{4d}}e^{-it\mathcal{L}_a}P_{\leq N}^a\phi\|_{L_t^{\frac{4d(d+2)}{(d-2)^2}}L_x^{\frac{2d(d+2)}{d^2+4}}(\mathbb{R}\times\mathbb{R}^d)} \\
			&\lesssim T^{\frac{d^2+8d-4}{4d(d+2)}}R^{\frac{d-2}{d+2}}N^{\frac{d-2}{2d}}\|e^{-it\mathcal{L}_{a}}\phi\|_{L_{t,x}^{\frac{2(d+2)}{d-2}}(\mathbb{R}\times\mathbb{R}^{d})}^{\frac{d-2}{2d}}\|(\mathcal{L}_{a})^{\frac{1}{2}}e^{-it\mathcal{L}_{a}}\phi\|_{L_{t}^{\infty}L_{x}^{2}(\mathbb{R}\times\mathbb{R}^{d})}^{\frac{d+2}{2d}} \\
			&\lesssim T^{\frac{d^2+8d-4}{4d(d+2)}}R^{\frac{d-2}{d+2}}N^{\frac{d-2}{2d}}\|e^{-it\mathcal{L}_{a}}\phi\|_{L_{t,x}^{\frac{2(d+2)}{d-2}}(\mathbb{R}\times\mathbb{R}^{d})}^{\frac{d-2}{2d}}\|\phi\|_{\dot{H}^{1}(\mathbb{R}^{d})}^{\frac{d+2}{2d}}.
		\end{align*}
		For the high frequency,  the local smoothing estimate $\eqref{local2}$ and Bernstein's inequality imply 
		\begin{align*}        
			\|\nabla e^{-it\mathcal{L}_{a}}P_{>N}^{a} \phi\|_{L_{t,x}^2([\tau-T,\tau+T]\times\{|x-z|\leq R\}}^{2}& \lesssim R\|P_{>N}^{a}\phi\|_{L_{x}^{2}}\|\nabla P_{>N}^{a}\phi\|_{L_{x}^{2}}+\|P_{>N}^{a}\phi\|_{L_{x}^{2}}^{2} \\
			&\lesssim(RN^{-1}+N^{-2})\|\phi\|_{\dot{H}_{a}^{1}}^{2}.
		\end{align*}
		Then optimizing the choice of $N$, we have
		\begin{align*}
			&\hspace{5ex}\|\nabla e^{-it\mathcal{L}_{a}} \phi\|_{L_{t,x}^2([\tau-T,\tau+T]\times\{|x-z|\leq R\})}^{2}\\
			&\lesssim T^{\frac{d^2+8d-4}{2(d+2)(3d-2)}}R^\frac{2d(d-2)}{(3d-2)(d+2)}\|\phi\|_{\dot{H}_{a}^{1}}^{\frac{2d}{3d-2}}\|e^{-it\mathcal{L}_{a}}\phi\|_{L_{t,x}^\frac{2(d+2)}{d-2}(\R\times\R^d)}^\frac{d-2}{3d-2}\\
			&\hspace{2ex}+T^{\frac{d^2+8d-4}{8(d-1)(d+2)}}R^\frac{3d^2-4d-4}{4(d-1)(d+2)}\|\phi\|_{\dot{H}_{a}^{1}}^{\frac{3d-2}{4(d-1)}}\|e^{-it\mathcal{L}_{a}}\phi\|_{L_{t,x}^\frac{2(d+2)}{d-2}(\R\times\R^d)}^\frac{d-2}{4(d-1)}.
		\end{align*}
	\end{proof}

	In the rest of the paper, we shall use the notation:
	\[
	S^{s}_{a}(I)=L^{q}_{t}H^{s,r}_{a}\cap C^0_{t}H^{s}_{a}(I\times\R^{d})
	\quad \text{and}\quad 
	\dot{S}^{s}_{a}(I)=L^{q}_{t}\dot{H}^{s,r}_{a}\cap C_t^0\dot{H}^{s}_{a}(I\times\R^{d}).
	\]
	where $s\in\{0,1\}$ and $(q,r)$ satisfies
	\begin{align*}
		\frac{2}{q}=d\bigg(\frac{1}{2}-\frac{1}{r}\bigg),\quad2\leqslant q,r\leqslant\infty.
	\end{align*}
	\subsection{Local well-posedness and stability theory}\label{Sec:smalldata}

	In this subsection, we present the local well-posedness and stability theory for \eqref{NLS} in the space $H_a^1$.  First, we give the local  well-posedness result. When $a=0$, we refer to Tao-Visan-Zhang\cite{Tao-Visan-Zhang} for $d\geqslant3$. When $a\neq0$, we refer to Ardila-Murphy\cite{Ardila-Murphy} for the three-dimensional case. 
	
	Before presenting the local well-posedness, we first give some definitions of the function space which will be used in  the proof of local theory.
	\begin{definition}[Function Spaces] We define the following spaces
		\begin{align*}
			\dot{X}_a^0(I)=L_{t}^{\gamma_1}L_x^{\rho_1}(I\times\R^d)\cap  L_t^{\frac{2(d+2)}{d-2}}L_x^{\frac{2d(d+2)}{d^2+4}}(I\times\R^d)\cap L_{t,x}^{\frac{2(d+2)}{d}}(I\times\R^d),
		\end{align*}
		where $\gamma_1=\frac{4(p_1+2)}{p_1(d-2)}$ and $\rho_1=\frac{d(p_1+2)}{d+p_1}$ and $p_1=\frac{d+3}{d-1}$.
		\begin{align*}
			\dot{X}_a^1(I)=\left\{u:\sqrt{\mathcal{L}_a}u\in\dot{X}_a^0(I)\right\},\quad X_a^1(I)=\dot{X}_a^0(I)\cap\dot{X}_a^1(I).
		\end{align*}
	\end{definition}
	Now we are in position to give the local well-posedness result. The proof is very close to Tao-Visan-Zhang\cite{Tao-Visan-Zhang}. The one thing we need to be cautious is that in the inverse-square setting, we must ensure that the exponent is in the range of Lemma   \ref{EquiSobolev}. The $L_{t,x}^{\frac{2(d+2)}{d}}$ norm can be interpolated by $L_{t,x}^{\frac{2(d+2)}{d-2}}$ and $L_t^\infty L_x^2$. Then using Sobolev embedding, we can bound $L_{t,x}^{\frac{2(d+2)}{d-2}}$ by $L_t^{\frac{2(d+2)}{d-2}}\dot{H}_a^{1,\frac{2d(d+2)}{d^2+4}}$, which is the endpoint of Lemma \ref{EquiSobolev}.
	\begin{theorem}[Local well-posedness]\label{Th1}
		Given $a>-\left(\frac{d-2}{2}\right)^2+\left(\frac{d-2}{d+2}\right)^2$ and $u_{0}\in H_{a}^{1}(\R^{d})$. Then for every $T>0$, there exists $\eta=\eta(T)$ such that if 
		\begin{align}\label{local}
			\big\|e^{it\mathcal{L}_a}u_0\big\|_{X_a^1([-T,T])}\leqslant\eta,
		\end{align}
		then $\eqref{NLS}$ admits a unique strong $H_a^1$-solution  on time interval $[-T,T]$. Let $(T_{min},T_{max})$ be the maximal lifespan of $u$. Then $u\in \dot{S}_a^1(I\times\R^d)$ for every compact time interval $I\subset(T_{min},T_{max})$ and the following properties hold
		\begin{enumerate}
			\item If $T_{max}<+\infty$, then
			\begin{align*}
				\big\|u\big\|_{{X}_a^1((0,T_{max}))}=+\infty.
			\end{align*}
			Similarly, if $T_{min}>-\infty$, then
			\begin{align*}
				\big\|u\big\|_{{X}_a^1((T_{min},0))}=+\infty.
			\end{align*}
			\item If $T_{max}=\infty$ and $T_{min}=-\infty$ with
			\begin{align*}
				\Vert u\Vert_{X_a^1(\Bbb{R})}<\infty,
			\end{align*}
			then $u$ scatters in $H_a^1(\R^d)$.
		\end{enumerate}
	\end{theorem}

	To obtain global well-posedness in $H_a^1$, we will need a following results establishing scattering in $H_a^1$ for sufficiently small initial data, along with a persistence of regularity result and a stability result.  All of these are analogues of results in \cite[Section~6]{KillipOhPoVi2017}.  As the proofs rely primarily on Strichartz estimates, which are readily available in the setting of the inverse-square potential, we omit them here. 
	
	\begin{proposition}[Small data scattering]\label{SDC}
		Let $d\in\{3,4,5\}$, $a>-\left(\frac{d-2}{2}\right)^2+\left(\frac{d-2}{d+2}\right)^2$ and $u_{0}\in H_{a}^{1}(\R^{d})$. There exists $\delta>0$ such that if $\|u_0\|_{\dot{H}_a^1}<\delta$, then the corresponding solution $u$ of \eqref{NLS} is global and scatters, with
		\[
		\|u\|_{L^{\frac{2(d+2)}{d-2}}_{t,x}(\R\times\R^{d})} \leqslant C({M(u_0)}) \|\sqrt{\L_{a}}u_{0}\|_{L^{2}(\R^{d})}.
		\]
	\end{proposition}
	\begin{lemma}[Persistence of regularity]\label{PRegularity}
		Suppose that $u: \R\times \R^{d}\rightarrow \C$ is a solution to \eqref{NLS} such that 
		$S(u):=\|u\|_{L^{\frac{2(d+2)}{d-2}}_{t,x}(\R\times\R^{d})}<\infty$. 
		Then for $t_{0}\in \R$ we have
		\begin{align}
			\|u\|_{\dot{S}_{a}^{0}(I )}&\leq C(S(u), M(u_{0}))\|u(t_{0})\|_{L^{2}_{x}(\R^{d})}, \label{L^2}\\
			\|u\|_{\dot{S}_a^1(I)}
			&\leq C(S(u), M(u_{0}))\|u(t_{0})\|_{\dot{H}^{1}_{a}(\R^{d})}.\label{H^1}
		\end{align}
	\end{lemma}
	\begin{proof}
		Indeed, first note that by Strichartz and H\"older
		\begin{equation}\label{RegulComp}
			\begin{split}
				\|u\|_{S_{a}^{0}(I)}& \lesssim \|  u(t_{0}) \|_{L^{2}_{x}}+\| u \|^{\frac{4}{d-2}}_{L^{\frac{2(d+2)}{d-2}}_{t, x}}\| u \|_{L_{t}^{2}L^{\frac{2d}{d-2}}_{x}}
				+\| u \|_{L_{t}^{\infty}L^{2}_{x}}^\frac{2}{d-1}\| u \|_{L^{\frac{2(d+2)}{d-2}}_{t, x}}^\frac{2}{d-1}\| u \|_{L_{t}^{2}L^{\frac{2d}{d-2}}_{x}}\\
				& \lesssim \|  u(t_{0}) \|_{L^{2}_{x}}+(\| u \|^{\frac{4}{d-2}}_{L^{\frac{2(d+2)}{d-2}}_{t, x}}+ [M(u)]^{\frac{1}{d-1}}\| u \|^{\frac{4}{d-2}}_{L^{\frac{2(d+2)}{d-2}}_{t, x}})
				\|u\|_{S_{a}^{0}(I)}
			\end{split}
		\end{equation}
		for any interval $ t_{0}\in I$, where space-time norms are over $I\times \R^{d}$. Since $\|u\|_{L^{\frac{2(d+2)}{d-2}}_{t,x}}<\infty$,  a standard bootstrap argument implies estimate $\eqref{L^2}$. 
		Similarly, we can obtain that
		\begin{align*}
			& \big\Vert \sqrt{\mathcal{L}_a} (|u|^\frac{4}{d-1}u)\big\Vert_{L_t^{\tfrac{d^2+d+2}{2(d-1)(d+2)}}L_x^{\tfrac{d^3+3d^2-12}{2d(d-1)(d+2)}}}\lesssim\Vert u\Vert_{L_t^{\frac{2(d+2)}{d-2}}\dot{H}_a^{1,\frac{2d(d+2)}{d^2+4}}}\Vert u\Vert_{L_t^{\frac{2(d+2)}{d}}L_x^{\frac{2d(d+2)}{d^2-d-2}}}^{\frac{4}{d-1}}\\
			&\lesssim\Vert u\Vert_{L_t^{\frac{2(d+2)}{d-2}}\dot{H}_a^{1,\frac{2d(d+2)}{d^2+4}}}\Vert u\Vert_{L_t^{2}L_x^{\frac{2d}{d-2}}}^{\frac{2}{d-1}}\Vert u\Vert_{L_{t,x}^{\frac{2(d+2)}{d-2}}}^{\frac{2}{d-1}}
		\end{align*}
		and
		\begin{align*}
			& \big\Vert \sqrt{\mathcal{L}_a} (|u|^\frac{4}{d-2}u)\big\Vert_{L_t^2L_x^{\frac{2d}{d+2}}}\lesssim\Vert u\Vert_{L_{t,x}^{\frac{2(d+2)}{d-2}}}^{\frac{4}{d-2}}\Vert u\Vert_{L_{t}^{\frac{2(d+2)}{d-2}}\dot{H}_a^{1,\frac{2d(d+2)}{d^2+4}}}.
		\end{align*}
		By the Strichartz estimates and H\"older's inequality, we obtain $\eqref{H^1}$.
	\end{proof}
	
	Using the standard bootstrap argument as in \cite{Ardila-Murphy}, we have the following stability result:
	\begin{lemma}[Stability]\label{StabilityNLS}
		Let $d\in\{3,4,5\}$. Fix $a>-\left(\frac{d-2}{2}\right)^2+\left(\frac{d-2}{d+2}\right)^2$. Let $I\subset \R$ be a time interval containing $t_{0}$ and  let $\tilde{u}$ satisfy
		\[ 
		(i\partial_{t}-\L_{a}) \tilde{u}=-|\tilde{u}|^{\frac{4}{d-2}}\tilde{u}+|\tilde{u}|^{\frac{4}{d-1}}\tilde{u}+e, \quad 
		\tilde{u}(t_{0})=\tilde{u}_{0}
		\]
		on $I\times\R^d$ for some $e:I\times\R^{d}\rightarrow \C$. Assume the conditions
		\[
		\| \tilde{u}  \|_{L_{t}^{\infty}H^{1}_{a}(I\times\R^{d})}\leq E\qtq{and} 
		\| \tilde{u}  \|_{L_{t,x}^{\frac{2(d+2)}{d-2}}(I\times\R^{d})}+	\| \tilde{u}  \|_{L_{t,x}^{\frac{2(d+2)}{d-1}}(I\times\R^{d})}\leq L
		\]
		for some $E$, $L>0$. Let $t_{0}\in I$ and $u_{0}\in H_{a}^{1}(\R^{d})$ such that $\|u_{0}\|_{L_{x}^{2}}\leq M$
		for some  positive constant $M$.  
		Assume also the smallness conditions 
		\[
		\|u_{0}-\tilde{u}_{0} \|_{\dot{H}^{1}_{a}}\leq \epsilon\qtq{and}	\|\sqrt{\L_{a}} e  \|_{N(I)}\leq \epsilon,
		\]
		for some $0<\epsilon<\epsilon_{0}=\epsilon_{0}(\mbox{A,L,M})>0$, where
		\[
		N(I)=L^{1}_{t}L^{2}_{x}(I\times \R^{d})+L_{t,x}^{\frac{2(d+2)}{d+4}}(I\times \R^{d})+L_t^2L_x^\frac{2d}{d+2}(I\times\R^d).
		\]
		Then there exists a unique global solution $u$ to Cauchy problem \eqref{NLS} with initial data $u_{0}$ at the time $t=t_{0}$ satisfying
		\[
		\|u-\tilde{u}\|_{\dot{S}_{a}^{1}(I)}\leq C(E,L,M)\epsilon\qtq{and}	\|u\|_{\dot{S}_{a}^{1}(I)}\leq C(E,L,M).
		\]
	\end{lemma}

		\subsection{Linear profile decomposition}\label{Sec:LinearProfile}
	
	In this subsection, we give the linear  profile decomposition associated to the propagator $e^{-it\L_a}$.  In fact, the result follows by combining the techniques of \cite{KillipMiaVisanZhangZheng,YANG2020124006}, which described the bubble decomposition of the Schr\"odinger group $e^{-it\mathcal{L}_a}f$ where $f\in \dot{H}_a^1(\R^d)$, with those of \cite{KillipOhPoVi2017,Luo2021}, which developed the bubble decomposition of the Schr\"odinger group without potential $e^{it\Delta}f$  with initial data $f\in H^1(\R^d)$.  Thus, we will focus on stating the main results and providing suitable references to the analogous results in the references just mentioned. 
	
	First, given a sequence $\left\{x_{n}\right\}\subset \R^{d}$, we define
	\begin{equation}\label{DefOperator}
		\L^{n}_{a}:=-\Delta+\tfrac{a}{|x+x_{n}|^{2}}
		\quad \text{and}\quad 
		\L^{\infty}_{a}:=
		\begin{cases} 
			-\Delta+\tfrac{a}{|x+x_{\infty}|^{2}} & \text{if $x_{n}\to x_{\infty}\in \R^{d}$},\\
			-\Delta & \text{if $|x_{n}|\to \infty$}.
		\end{cases} 
	\end{equation}
	In particular, $\L_{a}[\phi(x-x_{n})]=[\L^{n}_{a}\phi](x-x_{n})$, and for any $x_{n}\in \R^{d}$ and $N_{n}>0$,
	\[
	N^{\frac{d-2}{2}}_{n} e^{-it\L_{a}}[\phi(N_{n}x-x_{n})]=N^{\frac{d-2}{2}}_{n}[e^{-iN^{2}_{n}t\L^{n}_{a}}\phi](N_{n}x-y_{n}).
	\]
	
	To establish the linear profile decomposition associated to $e^{-it\mathcal{L}_a}$, we need the following results related to the convergence of the operator $\mathcal{L}_a^n$ to $\mathcal{L}_a^\infty$, which was initially proved in \cite{KillipMiaVisanZhangZheng, KillipMurphyVisanZheng2017}: 
	\begin{lemma}\label{ConverOpera}
		Fix $a>-\frac{(d-2)^2}{4}$. 
		\begin{itemize}
			\item If $t_{n}\to t_{\infty}\in \R$ and $\left\{x_{n}\right\}\subset \R^{d}$ satisfies $x_{n}\to x_{\infty}$ or $|x_{n}|\to \infty$, then
			\begin{align}\label{Conver11}
				&\lim_{n\to\infty}\|\L^{n}_{a}\psi-\L^{\infty}_{a}\psi\|_{\dot{H}^{-1}_{x}}=0 \quad \text{for all $\psi\in \dot{H}^{1}_{x}$},\\
				\label{Conver22}
				&\lim_{n\to\infty}\|(e^{-it_{n}\L^{n}_{a}}-e^{-it_{\infty}\L^{\infty}_{a}})\psi\|_{\dot{H}^{-1}_{x}}=0
				\quad \text{for all $\psi\in \dot{H}^{-1}_{x}$},\\
				\label{Conver33}
				&\lim_{n\to\infty}\|[\sqrt{\L^{n}_{a}}-\sqrt{\L^{\infty}_{a}}]\psi\|_{L^{2}_{x}}=0
				\quad \text{for all $\psi\in \dot{H}^{1}_{x}$}.
			\end{align}
			If $\frac{2}{q}+\frac{d}{r}=\frac{d}{2}$ with $2<q\leq \infty$, then we have
			\begin{equation}\label{Conver44}
				\lim_{n\to\infty}\|(e^{-it\L^{n}_{a}}-e^{-it\L^{\infty}_{a}})\psi\|_{L^{q}_{t}L^{r}_{x}(\R\times\R^{d})}=0
				\quad \text{for all $\psi\in L^{2}_{x}$}.
			\end{equation}
			Finally, if $x_{\infty}\neq 0$, then for any $t>0$,
			\begin{equation}\label{Conver55}
				\lim_{n\to\infty}\|(e^{-it\L^{n}_{a}}-e^{-it\L^{\infty}_{a}})\delta_{0}\|_{H^{-1}}=0.
			\end{equation}
			\item Given $\psi\in \dot{H}^{1}_{x}$, $t_{n}\to \pm\infty$ and any sequence 
			$\left\{x_{n}\right\}\subset \R^{d}$, we have
			\begin{equation}\label{DecayingES1}
				\lim_{n\to\infty}\|e^{-it_{n}\L^{n}_{a}}\psi\|_{L^{\frac{2d}{d-2}}_{x}}=0.
			\end{equation}
			Moreover, if $\psi\in {H}^{1}_{x}$, then
			\begin{equation}\label{DecayingES2}
				\lim_{n\to\infty}\|e^{-it_{n}\L^{n}_{a}}\psi\|_{L^{\frac{2d+2}{d-1}}_{x}}=0.
			\end{equation}
			\item Finally, fix $a>-\left(\frac{d-2}{4}\right)^2$. Then for any sequence $\{x_n\}$,
			\begin{equation}\label{Convercritical}
				\lim_{n\to\infty}\|(e^{-it\L^{n}_{a}}-e^{-it\L^{\infty}_{a}})\psi\|_{L^{\frac{2(d+2)}{d-2}}_{t,x}(\R\times\R^{d})}=0
				\quad \text{for all $\psi\in \dot{H}^{1}_{x}$}.
			\end{equation}
		\end{itemize}
	\end{lemma}
	
	The linear profile decomposition is stated as follows: 
	
	\begin{theorem}[Linear profile decomposition]\label{LinearProfi}Let $d\geqslant3$ and $a>-\frac{(d-2)^2}{4}$.
		Suppose that $\left\{f_{n}\right\}$ is a bounded sequence in $H^{1}(\R^{d})$. Then, up to subsequence, there exist
		$J^{*}\in \left\{0,1,2,\ldots\right\}\cup\left\{\infty\right\}$, non-zero profiles 
		$\{\phi^{j}\}^{J^{*}}_{j=1}\subset \dot{H}^{1}(\R^{d})$ and parameters
		\[
		\left\{\lambda^{j}_{n}\right\}_{n\in \N}\subset (0,1],\quad 
		\left\{t^{j}_{n}\right\}_{n\in \N}\subset \R
		\quad \text{and}\quad 
		\left\{x^{j}_{n}\right\}_{n\in \N}\subset \R^{d}
		\]
		so that for each finite $1\leq J\leq J^{*}$, we have the decomposition
		\begin{equation}\label{Dcom}
			f_{n}=\sum^{J}_{j=1}\phi_{n}^{j}+W^{J}_{n},
		\end{equation}
		where
		\begin{equation}\label{fucti}
			\phi_{n}^{j}(x):=
			\begin{cases} 
				[e^{-it^{j}_{n}\L^{n_{j}}_{a}}\phi^{j}](x-x^{j}_{n}), &\mbox{if $\lambda^{j}_{n}\equiv 1$},\\
				(\lambda^{j}_{n})^{-\frac{d-2}{2}}[e^{-it^{j}_{n}\L^{n_{j}}_{a}} P^{a}_{\geq (\lambda^{j}_{n})^{\theta}}\phi^{j}]\( \frac{x-x^{j}_{n}}{\lambda^{j}_{n}}\),
				&\mbox{if $\lambda^{j}_{n}\rightarrow 0$},
			\end{cases}
		\end{equation}
		for some $0<\theta<1$ $($ with $\L^{n_{j}}_{a}$ as in \eqref{DefOperator} corresponding to sequence 
		$\{\tfrac{x^{j}_{n}}{\lambda^{j}_{n}}\})$, satisfying
		\begin{itemize}
			\item $\lambda^{j}_{n}\equiv 1$ or $\lambda^{j}_{n}\rightarrow 0$ and $t^{j}_{n}\equiv 0$ or $t^{j}_{n}\rightarrow\pm\infty$,
			\item if $\lambda^{j}_{n}\equiv 1$ then $\left\{\phi^{j}\right\}^{J^{*}}_{j=1}\subset L_{x}^{2}(\R^{d})$
		\end{itemize}
		for each $j$.  Furthermore, we have:
		\begin{itemize}[leftmargin=5mm]
			\item Smallness of the reminder: 
			\begin{equation}\label{Reminder}
				\lim_{J\to J^{*}}\limsup_{n\rightarrow\infty}\|e^{-it\L_{a}}W^{J}_{n}\|_{L^{\frac{2(d+2)}{d-2}}_{t,x}(\R\times\R^{d})}=0.
			\end{equation}
			\item Weak convergence property:
			\begin{equation}\label{WeakConver}
				e^{it^{j}_{n}\L_{a}}[(\lambda^{j}_{n})^{\frac{1}{2}}W^{J}_{n}(\lambda^{j}_{n}x+x^{j}_{n})]\rightharpoonup 0\quad \mbox{in}\,\,
				\dot{H}^{1}_{a}, \quad \mbox{for all $1\leq j\leq J$.}
			\end{equation}
			\item Asymptotic orthogonality: for all $1\leq j\neq k\leq J^{*}$
			\begin{equation}\label{Ortho}
				\lim_{n\rightarrow \infty}\left[ \frac{\lambda^{j}_{n}}{\lambda^{k}_{n}}+\frac{\lambda^{k}_{n}}{\lambda^{j}_{n}} 
				+\frac{|x^{j}_{n}-x^{k}_{n}|^{2}}{\lambda^{j}_{n}\lambda^{k}_{n}}+
				\frac{|t^{j}_{n}(\lambda^{j}_{n})^{2}-t^{k}_{n}(\lambda^{k}_{n})^{2}|}{\lambda^{j}_{n}\lambda^{k}_{n}}\right]=\infty.
			\end{equation}		
			\item Asymptotic Pythagorean expansions:
			\begin{align}\label{MassEx}
				&\sup_{J}\lim_{n\to\infty}\big[M(f_{n})-\sum^{J}_{j=1}M(\phi^{j}_{n})-M(W^{J}_{n})\big]=0,\\\label{EnergyEx}
				&\sup_{J}\lim_{n\to\infty}\big[E_{a}(f_{n})-\sum^{J}_{j=1}E_{a}(\phi^{j}_{n})-E_{a}(W^{J}_{n})\big]=0.
			\end{align}
		\end{itemize}
	\end{theorem}
	\begin{remark}
		We notice that the parameters $x_n^j\equiv0$ if $\{f_n\}\subset\dot{H}_a^1(\R^d)$ are radial functions. 
	\end{remark}
	In order to prove Theorem \ref{LinearProfi}, 
	we  need the  refined Strichartz estimate and inverse Strichartz estimate. 
	
	The refined Strichartz estimates for linear Schr\"odinger flow $e^{-it\mathcal{L}_a}f$ was first proved in  \cite{KillipMiaVisanZhangZheng} for $d=3$ and then generalized to the higher dimensions by \cite{YANG2020124006}.
	
	\begin{lemma}[Refined Strichartz]\label{RefinedStrichartz}
		Let $d\geqslant3$ and $a>-\frac{(d-2)^2}{4}+\left(\frac{d-2}{d+2}\right)^2$. For $f\in \dot{H}^{1}_{a}(\R^{d})$ we have
		\[
		\| e^{-it\L_{a}}f\|_{L^{\frac{2(d+2)}{d-2}}_{t,x}(\R\times\R^{d})}\lesssim
		\|  f \|^{\frac{d-2}{d+2}}_{\dot{H}^{1}_{a}(\R^{d})}\sup_{N\in 2^{\Z}}
		\| e^{-it\L_{a}}f_{N}\|^{\frac{4}{d+2}}_{L^{\frac{2(d+2)}{d-2}}_{t,x}(\R\times\R^{d})}.
		\]
	\end{lemma}
	
	Using the similar argument as in \cite{KillipOhPoVi2017,Luo2022} and Lemma \ref{RefinedStrichartz}, we can  prove the following inverse Strichartz inequality which implies that there is only one bubble concentration.
	\begin{proposition}[Inverse Strichartz inequality]\label{InverseSI}
		Let $a>-\left(\frac{d-2}{2}\right)^2+\left(\frac{d-2}{d+2}\right)^2$. Let $\left\{f_{n}\right\}_{n\in \N}$ be a sequence such that
		\[
		\limsup_{n\to\infty}\|  f_{n} \|_{{H}^{1}_{a}}=A<\infty
		\quad \text{and}\quad 
		\liminf_{n\to\infty}\| e^{-it\L_{a}}f\|_{L^{\frac{2(d+2)}{d-2}}_{t,x}(\R\times\R^{d})}=\epsilon>0.
		\]
		Then, after passing to a subsequence in $n$, there exist $\phi\in \dot{H}^{1}_{a}$,
		\[
		\left\{\lambda_{n}\right\}_{n\in \N}\subset (0,\infty), \quad\left\{t_{n}\right\}_{n\in \N}\subset \R,
		\quad \left\{x_{n}\right\}_{n\in \N}\subset \R^{d}
		\]
		such that the following statements hold:
		\begin{enumerate}[label=\rm{(\roman*)}]
			\item $\lambda_{n}\to \lambda_{\infty}\in [0, \infty)$, and if $\lambda_{\infty}>0$ then
			$\phi\in{H}^{1}_{a}$.
			\item Weak convergence property:
			\begin{equation}\label{weakprofile}
				\lambda^{\frac{d-2}{2}}_{n}( e^{-it\L_{a}}f_{n})(\lambda_{n}x+x_{n})\rightharpoonup \phi(x)
				\quad \text{weakly in}\quad 
				\begin{cases} 
					H_a^{1}(\R^{d}), \quad \text{if $\lambda_{\infty}>0$} \\
					\dot{H}_a^{1}(\R^{d}), \quad \text{if $\lambda_{\infty}=0$}.
				\end{cases} 
			\end{equation}
			\item Decoupling of norms:
			\begin{align}\label{ConverH1}
				&	\lim_{n\to\infty}\left\{\|f_{n}\|^{2}_{\dot{H}^{1}_{a}}-\|f_{n}-\phi_{n}\|^{2}_{\dot{H}^{1}_{a}}\right\}
				\gtrsim_{\epsilon,A} 1\\\label{ConverLp}
				&\lim_{n\to\infty}\left\{\|f_{n}\|^{2}_{L_{x}^{2}}-\|f_{n}-\phi_{n}\|^{2}_{L_{x}^{2}}-\|\phi_{n}\|^{2}_{L_{x}^{2}}\right\}=0,
			\end{align}
			where {\small
				\[
				\phi_{n}(x):=\begin{cases} 
					\lambda^{-\frac{d-2}{2}}_{n}e^{-it_{n}\L_{a}}\left[\phi\(\frac{x-x_{n}}{\lambda_{n}}\)\right] 
					& \text{if $\lambda_{\infty}>0$}, \\
					\lambda^{-\frac{d-2}{2}}_{n}e^{-it_{n}\L_{a}}\left[(P^{a}_{\geq \lambda^{\theta}_{n}}\phi)\(\frac{x-x_{n}}{\lambda_{n}}\)\right]                 
					& \text{if $\lambda_{\infty}=0$},
				\end{cases} 
				\]
				with $0<\theta<1$.
			}
			\item We may choose the parameters $\left\{\lambda_{n}\right\}_{n\in \N}$, $\left\{t_{n}\right\}_{n\in \N}$  and $\left\{x_{n}\right\}_{n\in \N}$ such that either 
			$\frac{t_{n}}{\lambda^{2}_{n}}\to \pm \infty$ or $t_{n}\equiv 0$ and either $\frac{|x_{n}|}{\lambda_{n}}\to  \infty$ 
			or $x_{n}\equiv 0$.
		\end{enumerate}
	\end{proposition}
	
	Arguing as in  \cite[Proposition 3.7]{KillipMiaVisanZhangZheng},  \cite[Lemma 7.4]{KillipOhPoVi2017} and \cite[Lemma 3.6]{Luo2021}, we also have the following: 
	\begin{lemma}\label{ScalingPara}
		Under the hypotheses of Proposition \ref{InverseSI}, we have:
		\begin{enumerate}[label=\rm{(\roman*)}]
			\item Passing to subsequence, we may assume that either $\lambda_{n}\equiv1$ or $\lambda_{n}\to0$.
			\item \begin{align}\label{ConverL6}
				&\lim_{n\to\infty}\left\{\|f_{n}\|^{\frac{2d}{d-2}}_{L_{x}^{\frac{2d}{d-2}}}-\|f_{n}-\phi_{n}\|^{\frac{2d}{d-2}}_{L_{x}^{\frac{2d}{d-2}}}-\|\phi_{n}\|^{\frac{2d}{d-2}}_{L_{x}^{\frac{2d}{d-2}}}\right\}=0,\\\label{ConverL4}
				&\lim_{n\to\infty}\left\{\|f_{n}\|^{\frac{2d+2}{d-1}}_{L_{x}^{\frac{2d+2}{d-1}}}-\|f_{n}-\phi_{n}\|^{\frac{2d+2}{d-1}}_{L_{x}^{\frac{2d+2}{d-1}}}-\|\phi_{n}\|^{\frac{2d+2}{d-1}}_{L_{x}^{\frac{2d+2}{d-1}}}\right\}=0.
			\end{align}
		\end{enumerate}
	\end{lemma}
	
	Using the standard argument as in \cite{Luo2021} along with Proposition \ref{InverseSI} and Lemma \ref{ScalingPara}, one can finish the proof of linear profile decomposition, i.e. Theorem \ref{LinearProfi}.
	
	
	\subsection{Local virial argument}\label{virialargument}
	In this subsection, we give the computation of the local  virial identity. Let $\phi$ be a radial, smooth function which satisfy
	\[
	\phi(x)=
	\begin{cases}
		|x|^{2},& \quad |x|\leq 1\\
		0,& \quad |x|\geq 2,
	\end{cases}
	\quad \text{with}\quad 
	|\partial^{\alpha}\phi(x)|\lesssim |x|^{2-|\alpha|}
	\]
	for all multi-indices $\alpha$. For a given $R>1$, we define
	\begin{equation}\label{WR}
		w_{R}(x)=R^{2}\phi\(\tfrac{x}{R}\)
	\end{equation}
	and introduce the localized virial quantity
	\begin{gather}\label{V_R}
		V_R[u]=\int_{\R^d}w_R|u|^2dx,\\
		I_{R}[u]:=\frac{d}{dt}V_R[u]=2\IM\int_{\R^{d}} \nabla w_{R} \cdot\nabla u \,\overline{u} \,dx.
	\end{gather}
	
	By a direct calculation, we can deduce the following virial identity:
	\begin{lemma}\label{VirialIden}
		Let $R\in [1, \infty)$. Assume that $u$ solves \eqref{NLS}. Then we have
		\begin{equation}\label{LocalVirial}
			\tfrac{d}{d t}I_{R}[u]=F_{R}[u],
		\end{equation}
		where
		\begin{align*}
			\frac{d}{dt} I_R[u]
			&= \; 4  \int_{\R^d} w_R''(r) \big|\nabla u(t,x)\big|^2+\frac{ax}{|x|^4}\cdot\nabla w_R|u|^2 \;dx -\int_{\R^d} (\Delta^2w_R)(x) |u(t,x)|^2 \; dx  \\
			& \hspace{2ex} - \frac4d \int_{\R^d} (\Delta w_R)  |u(t,x)|^\frac{2d}{d-2}\; dx +\frac{4}{d+1}
			\int_{\R^d} (\Delta w_R)  |u(t,x)|^\frac{2d+2}{d-1}\; dx\\
			&:=F_R^c[u]+\frac{4}{d+1}\int_{\R^d} (\Delta w_R)  |u(t,x)|^\frac{2d+2}{d-1}\; dx.
		\end{align*}
	\end{lemma}
	When $R=\infty$, we denote $F^{c}_{\infty}[u]=8G[u]$,
	where
	\[
	G[f]:=\| f\|_{\dot{H}_a^1}^2-\|f\|^{\frac{2d}{d-2}}_{L_{x}^{\frac{2d}{d-2}}}.
	\]
	\begin{lemma}\label{Virialzero}
		Fix $R\in [1, \infty]$, $\theta\in \R$ and $\lambda >0$. Then we have
		\[
		F^{c}_{R}[e^{i\theta}\lambda^{\frac{d-2}{2}}W_a(\lambda\cdot)]=0.
		\]
	\end{lemma}
	\begin{proof}
		If $R=\infty$, then by a change of variable and \eqref{PoQ} we get
		\[
		F^c_{\infty}[e^{i\theta}\lambda^{\frac{d-2}{2}}W_a(\lambda\cdot)]
		=8G[e^{i\theta}\lambda^{\frac{d-2}{2}}W_a(\lambda\cdot)]
		=8G[W_a]=0.
		\]
		Now assume $R\in [1, \infty)$. Fix $\theta\in \R$ and $\lambda >0$. Since 
		\[
		u(t,x)=e^{i\theta}\lambda^{\frac{d-2}{2}}W_a(\lambda x)
		\]
		is a solution to \eqref{NLS-Potential} and $I_{R}[e^{i\theta}\lambda^{\frac{d-2}{2}}W_a(\lambda\cdot)]=0$, Lemma~\ref{VirialIden} implies
		\[
		F^{c}_{R}[e^{i\theta}\lambda^{\frac{d-2}{2}}W_a(\lambda\cdot)]=0.
		\]
	\end{proof}

	Combining Lemmas~\ref{VirialIden} and \ref{Virialzero} yields the following, which we will use to incorporate the modulation analysis into the virial analysis. 
	
	\begin{lemma}\label{VirialModulate}
		Consider $R\in [1, \infty]$, $\chi: I\to \R$, $\theta: I\to \R$ and
		$\lambda: I\to \R$. If $u$ is a solution  to \eqref{NLS}, then we have
		\begin{align}
			\tfrac{d}{d t}I_{R}[u]=&F^{c}_{\infty}[u(t)]\nonumber\\
			&+F_{R}[u(t)]-F^{c}_{\infty}[u(t)] \label{Modu11}\\ 
			&-\chi(t)\big\{F^{c}_{R}[e^{i\theta(t)}\lambda(t)^{\frac{d-2}{2}}W_a(\lambda(t)\cdot)]-F^{c}_{\infty}[e^{i\theta(t)}\lambda(t)^{\frac{d-2}{2}}W_a(\lambda(t)\cdot)]\big\}.\label{Modu22}
		\end{align}
	\end{lemma}
\section{Variational analysis}
\label{sec:VarationalGNI}
In this section, we give some basic properties for the ground states $W_a(x)$ and $W_0(x)$ satisfying the following elliptic equations
\begin{gather*}
	\mathcal{L}_aW_a=|W_a|^{\frac{4}{d-2}}W_a,\\
	-\Delta W_0=|W_0|^{\frac{4}{d-2}}W_0.
\end{gather*}
To simplify the notation, we assume that $d\geq3$ and $0>a>-\left(\frac{d-2}{2}\right)^2$ in this section.

In the double-power setting, we still use the ground states $W_a(x)$ and $W_0(x)$ as the threshold for scattering. 
\begin{definition}  We define $1\geqslant\beta>0$ via $a=(\frac{d-2}2)^2[\beta^2-1]$, or equivalently, $\sigma=\frac{d-2}2(1-\beta)$.  We then define the \emph{ground state soliton}  by
	\begin{equation}\label{E:Wa}
		W_a(x) :=    [d(d-2)\beta^2]^{\frac{d-2}{4}} \biggl[ \frac{ |x|^{\beta-1} }{ 1+|x|^{2\beta} }\biggr]^{\frac{d-2}{2}}.
	\end{equation}
\end{definition}

By the direct calculation, we have
$$
\mathcal{L}_a W_a = |W_a|^\frac{4}{d-2} W_a. 
$$
Then using the  Euler's Beta integral  formula(cf. (1.1.20) in \cite{AAR}),  we can show that
\begin{equation}\label{WaPoho}
	\|W_a\|_{\dot H^{1}_a(\R^d)}^2 = \int_{\R^d} |W_a(x)|^{\frac{2d}{d-2}}\,dx = 
	\tfrac{\pi d(d-2)}{4} \Bigl[\tfrac{2\sqrt{\pi}\beta^{d-1}}{\Gamma(\frac{d+1}{2})} \Bigr]^\frac2d.
\end{equation}
Thus $W_a$ is a ground state soliton in the sense of being a radial non-negative static solution to \eqref{NLS-Potential}.   
\begin{proposition}[Sharp Sobolev embedding] \label{P:Wa}
	{\upshape(i)} We have the following sharp inequality
	\begin{equation}\label{Sharp-Sobolev}
		\|f\|_{L^{\frac{2d}{d-2}}_x(\R^d)} \leq C_{GN} \|f\|_{\dot H^{1}_a(\R^d)},
	\end{equation}
	where $$C_{GN}:=\|W_a\|_{L^{\frac{2d}{d-2}}_x(\R^d)} \|W_a\|_{\dot H^{1}_a(\R^d)}^{-1}.$$
	Moreover, equality holds in \eqref{Sharp-Sobolev} if and only if $f(x)=\alpha W_a(\lambda x)$ for some $\alpha\in\C$ and some $\lambda>0$.\\
	{\upshape(ii)} The inequality \eqref{Sharp-Sobolev} is valid also when $a=0$; however, equality now holds if and only if $f(x)=\alpha W_0(\lambda x+y)$ for some $\alpha\in\C$, some $y\in\R^d$,
	and some $\lambda>0$.
\end{proposition}
The ground state $W_a$ also satisfies the following Pohozaev identities:
\begin{equation}\label{PoQ}
	\| W_a\|_{\dot{H}_a^1}^2=\| W_a\|^{\tfrac{2d}{d-2}}_{L^{\tfrac{2d}{d-2}}} \mbox{ and } E_a^{c}(W_a)=\tfrac{1}{d}\| W_a\|_{\dot{H}_a^1}^2.
\end{equation}

We will need the following two elementary lemmas.
\begin{lemma}
	\label{lem.var}
	Let $f\in \dot{H}_a^1$ such that $\|f\|_{\dot{H}_a^1}\leq \|W_a\|_{\dot{H}_a^1}$. Then
	$$ \frac{\|f\|_{\dot{H}_a^1}^2}{\|W_a\|_{\dot{H}_a^1}^2}\leq \frac{E_a^c(f)}{E_a^c(W_a)}.$$
	In particular, $E_a^c(f)$ is positive.
\end{lemma}

\begin{proof}
	Let $\Phi(y)=\frac{1}{2} y - \frac{C_{GN}^{\frac{2d}{d-2}}}{\frac{2d}{d-2}}y^{\frac{2d}{d-2}/2}$.  Then by Sobolev embedding
	$$\Phi\big(\|f\|^2_{\dot{H}_a^1}\big) \leq \frac{1}{2}\|f\|_{\dot{H}_a^1}^2-\frac{1}{\frac{2d}{d-2}} \|f\|^{\frac{2d}{d-2}}_{L^{\frac{2d}{d-2}}}=E_a^c(f).$$
	Note that $\Phi$ is concave on $\Bbb{R}_+$, $\Phi(0)=0$ and $\Phi\big(\|W_a\|_{\dot{H}_a^1}^2\big)=E_a^c(W_a)$. Thus
	\begin{align*}
		\forall s\in(0,1),\quad \Phi\big(s\|W_a\|_{\dot{H}_a^1}^2\big) \geq s \Phi(\|W_a\|_{\dot{H}_a^1}^2)=sE_a^c(W_a).
	\end{align*}
	We take $s=\frac{\|f\|_{\dot{H}_a^1}^2}{\|W_a\|_{\dot{H}_a^1}^2}$ and hence complete the proof.
\end{proof}

\begin{lemma}\label{GlobalS}
	Suppose that $u_{0}\in H_a^{1}(\R^{d})$ satisfies
	\begin{equation}\label{Condition11}
		E_a(u_{0})\leq E_a^{c}(W_a)
		\quad \text{and}\quad 
		\| u_{0}\|_{\dot{H}_a^1}<\| W_a\|_{\dot{H}_a^1}.
	\end{equation}
	Then the solution to \eqref{NLS} with $u|_{t=0}=u_0$ satisfies 
	\begin{equation}\label{PositiveP}
		\| u\|_{\dot{H}_a^1}<\| W_a\|_{\dot{H}_a^1}
	\end{equation}
	throughout  maximal lifespan.  Similarly, if $\| u_{0}\|_{\dot{H}_a^1}>\| W_a\|_{\dot{H}_a^1}$, then $\| u\|_{\dot{H}_a^1}>\| W_a\|_{\dot{H}_a^1}$ throughout the maximal lifespan of $u$. 
\end{lemma}

\begin{proof} Suppose that there exists $t_{0}$ such  that $\| u(t_{0})\|_{\dot{H}_a^1}^2=\| W_a\|^{2}_{\dot{H}_a^1}$. By Lemma \ref{lem.var}, we obtain
	\[
	1=\frac{\| u(t_{0})\|_{\dot{H}_a^1}^2}{\| W_a\|^{2}_{\dot{H}_a^1}}\leq \frac{E_a^{c}(u(t_{0}))}{E_a^{c}(W_a)}
	<\frac{E_a(u(t_{0}))}{E_a^{c}(W_a)},
	\]
	which contradicts  to \eqref{Condition11}.  A similar argument treats the remaining case. \end{proof}

We also need the following comparison lemma which is crucial for us to prove Theorem \ref{sub-threshold} and \ref{Threshold}.
\begin{lemma}
	We  have the following comparsion estimate for the ground state $W_0$ and $W_a$,
	\begin{align*}
		E_a^c(W_a)< E_0^c(W_0),\quad \Vert W_a\Vert_{\dot{H}_a^1}<\Vert W_0\Vert_{\dot{H}^1}.
	\end{align*} 
\end{lemma}   
\begin{proof}
	In fact, by the sharp Sobolev inequality and $\|W_0\|_{\dot{H}_a^1}<\|W_0\|_{\dot{H}^1}$, we have
	\begin{align*}
		\|W_0\|_{\frac{2d}{d-2}}< C_{GN}\|W_0\|_{\dot{H}_a^1}\leq C_{GN}\|W_0\|_{\dot{H}^1}=\|W_a\|_{\dot{H}_a^1}^{\frac{-2}{d}}\|W_0\|_{\dot{H}^1}.
	\end{align*}
	Since $\|W_0\|_{\frac{2d}{d-2}}/\|W_0\|_{\dot{H}^1}=\|W_0\|_{\dot{H}^1}^{\frac{-2}d}$, we finish the proof.
\end{proof}
Inspired by the work of Ibrahim-Masmoudi-Nakanishi\cite{Nakanishi} and Miao-Xu-Zhao\cite{MXZ2013}, we give the following variational characterization in order to prove Theorem \ref{sub-threshold}(i)(ii). 
We first give some notation and definition associated with the variational analysis. For $\varphi\in H_a^1$, we denote the scaling quantity
$\varphi^{\lambda}_{d, -2}$ by
\begin{align*}
	\varphi^{\lambda}_{d, -2} (x)= e^{d\lambda}\varphi(e^{2\lambda}x).
\end{align*}
We denote the variation derivative of $E_a(\varphi)$ by $K_a(\varphi)$
\begin{align}\label{scaling deriv:special}
	K_a(\varphi)=  \mathcal{T} E_a(\varphi)
	&:=  \dfrac{d}{d \lambda } \Big|_{\lambda =0 } E_a(
	\varphi^{\lambda}_{d, -2}) \notag\\& =  \int_{\R^d}\left( 2 |\nabla
	\varphi|^2+2a\frac{|\varphi|^2}{|x|^2} - 2 |\varphi|^{\frac{2d}{d-2}} + \frac{2d}{d+1} |\varphi|^{\frac{2d+2}{d-1}}\right) \; dx.
\end{align}
Then we denote the quadratic and nonlinear parts of $K_a$ by $K_a^Q$ and
$K_a^N$, that is,
\begin{align*}
	K_a(\varphi)=K_a^Q(\varphi) + K_a^N(\varphi),
\end{align*}
where 
\begin{gather*}
	K_a^Q(\varphi)= \displaystyle 2 \; \int_{\R^d} \left(|\nabla \varphi|^2+a\frac{|\varphi|^2}{|x|^2}\right) \;
	dx,\\
	K_a^N (\varphi)= \displaystyle \int_{\R^d} \left(-2 |  \varphi|^{\frac{2d}{d-2}} + \frac{2d}{d+1}  \varphi|^{\frac{2d+2}{d-1}} \right)\;
	dx.
\end{gather*}

Next, we define the  threshold for Theorem \ref{sub-threshold} as the following constrained minimization problem associated to the energy $E_a(\varphi)$:
\begin{align}\label{minimization}
	m_a= \inf \{ E_a(\varphi)\; |\; 0\not=\varphi \in H_a^1(\R^d), \;
	K_a(\varphi)=0 \},
\end{align}
and two regions
\begin{align*}
	\K_a^+&=\{f \in H_a^1(\R^d)| E_a(f)<E_a^c(W_a),\ K_a(f)\geq0 \},\\
	\K_a^-&=\{f \in H_a^1(\R^d)| E_a(f)<E_a^c(W_a),\ K_a(f)<0 \}.
\end{align*}
In this section, we will prove that 
$$m_a=E_a^c(W_a)$$
and the scattering and blow-up region in Theorem \ref{sub-threshold} are actually  the sets $\K_a^+$ and $\K_a^-$ respectively. Finally, we give the crucial convex lemma adpated to the proof of Theorem \ref{sub-threshold}.

We start with the following basic lemmas.
\begin{lemma}\label{positive near origin:KQ }
	For any $\varphi \in H_a^1(\R^d)$, we have
	\begin{align}\label{asymptotic:KQ}
		\lim_{\lambda \rightarrow -\infty}
		K_a^Q(\varphi^{\lambda}_{d,-2}) =0.
	\end{align}
\end{lemma}
By the definition of $K_a^Q$, one can prove this lemma directly. Thus we omit the details. 

The next lemma  show the positivity of $K_a$ for sequences which is  near $0$ in the energy
space.
\begin{lemma}\label{Postivity:K} For any bounded sequence
	$\{\varphi_n\}\subset H_a^1(\R^d) \backslash\{0\}$ with
	\begin{align*}
		\lim_{n\rightarrow +\infty}K_a^Q(\varphi_n)=0,
	\end{align*}
	then for large $n$, we have
	\begin{align*}
		K_a(\varphi_n)>0.
	\end{align*}
\end{lemma}
\begin{proof}Since $K_a^Q(\varphi_n) \rightarrow 0$, we obtain
	$\displaystyle
	\lim_{n\rightarrow+\infty}\big\|\nabla\varphi_n\big\|^2_{L^2}\sim\lim_{n\rightarrow+\infty}\big\|\varphi_n\big\|^2_{\dot{H}_a^1}=0.
	$ Then using the Sobolev  embedding, we have for large
	$n$
	\begin{align*}
		\big\|\varphi_n\big\|^{\frac{2d}{d-2}}_{L^{\frac{2d}{d-2}}_x} \lesssim   \big\|\nabla
		\varphi_n\big\|^{\frac{2d}{d-2}}_{L^2_x} = & o(\big\|\nabla\varphi_n\big\|^2_{L^2}),\\
		\big\|\varphi_n\big\|^{\frac{2d+2}{d-1}}_{L^{\frac{2d+2}{d-1}}_x}\lesssim
		\big\|\varphi_n\big\|_{L^2}^{\frac{2d+2}{d^2-1}}\big\|\nabla \varphi_n\big\|^{\frac{2d^2+2d}{d^2-1}}_{L^2} & =
		o(\big\|\nabla\varphi_n\big\|^2_{L^2}),
	\end{align*}
	where we use the fact that  $\varphi_n$ is $L^2$ bounded. Hence for large $n$, we have
	\begin{align*} K_a(\varphi_n)= &  \int_{\R^d}
		\left( 2 |\nabla \varphi_n|^2+2a\frac{|\varphi|^2}{|x|^2} -2|\varphi_n|^{\frac{2d}{d-2}} +\frac{2d}{d+1} |\varphi_n|^{\frac{2d+2}{d-1}}
		\right)\; dx \thickapprox   \int_{\R^d}  |\nabla \varphi_n|^2\; dx > 0.
	\end{align*}
	This complete the proof.
\end{proof}

By the definition of $K_a$, we denote two real numbers by
\begin{align*}
	\bar{\mu} = \max\{\tfrac{2d+2}{d-1},0, \tfrac{2d}{d-2}\}, \quad \underline{\mu}=\min\{\tfrac{2d+2}{d-1},0,
	\tfrac{2d}{d-2}\}.
\end{align*}
Note that for $d\in \{3,4,5\}$,
\begin{align*}
	\bar{\mu} = \max\{\tfrac{2d+2}{d-1},0, \tfrac{2d}{d-2}\}=\tfrac{2d}{d-2}, \quad \underline{\mu}=\min\{\tfrac{2d+2}{d-1},0,
	\tfrac{2d}{d-2}\}=0.
\end{align*}

Next, we show the behavior of the scaling derivative functional $K_a$.
\begin{lemma}\label{structure:J}
	For any $\varphi \in H_a^1$, we have
	\begin{align*}
		\left(\bar{\mu}-\mathcal{L}\right)E_a(\varphi) = & \int_{\R^d} \frac{2}{d-1}\left(
		\big|\nabla \varphi\big|^2+a\frac{|\varphi|^2}{|x|^2}+
		\big|\varphi\big|^{\frac{2d}{d-2}} \right)\; dx, \\
		\mathcal{L}  \left(\bar{\mu}-\mathcal{L}\right)E_a(\varphi) = &
		\int_{\R^d}\left(\frac{8}{d-1}\big|\nabla \varphi\big|^2 +\frac{8a}{d-1} \frac{|\varphi|^2}{|x|^2}+\frac{8d}{(d-1)(d-2)}\big|\varphi\big|^{\frac{2d}{d-2}}
		\right) \; dx.
	\end{align*}
\end{lemma}
\begin{proof} By the definition of $\mathcal{T}$,
	we have
	\begin{align*}
		&\mathcal{T} \big\|\nabla \varphi\big\|^2_{L^2} = 4 \big\|\nabla
		\varphi\big\|^2_{L^2}+,4a \big\|
		\frac{\varphi}{x}
		\big\|^2_{L^2}  \quad \mathcal{T} \big\|  \varphi\big\|^{\frac{2d}{d-2}}_{L^{\frac{2d}{d-2}}} =\frac{4d}{d-2} \big\|
		\varphi\big\|^{\frac{2d}{d-2}}_{L^{\frac{2d}{d-2}}},\\
		&  \mathcal{T} \big\|  \varphi\big\|^{\frac{2d+2}{d-1}}_{L^{\frac{2d+2}{d-1}}} =\frac{4d}{d-1}
		\big\| \varphi\big\|^{\frac{2d+2}{d-1}}_{L^{\frac{2d+2}{d-1}}},
	\end{align*}
	which implies that
	\begin{align*}
		&\left(\bar{\mu}-\mathcal{T}\right)E_a(\varphi) = \frac{2d}{d-2}E_a(\varphi)-K_a(\varphi) = \frac{2}{d-1}
		\int_{\R^d} \left( \big|\nabla \varphi\big|^2+a\frac{|\varphi|^2}{|x|^2}
		+ \big|\varphi\big|^{\frac{2d}{d-2}} \right)\; dx,
		\\
		&\mathcal{T}  \left(\bar{\mu}-\mathcal{T}\right)E_a(\varphi) =   \frac{2}{d-1}\mathcal{T} \big\|\nabla
		\varphi\big\|^2_{L^2} + \frac{2}{d-1}\mathcal{T}
		\big\|  \varphi\big\|^{\frac{2d}{d-2}}_{L^{\frac{2d}{d-2}}} \\
		&= \int_{\R^d} \left( \frac{8}{d-1}  \big|\nabla
		\varphi\big|^2 +\frac{8a}{d-1} \frac{|\varphi|^2}{|x|^2}+ \frac{8d}{(d-1)(d-2)} \big| \varphi\big|^{\frac{2d}{d-2}} \right) \; dx.
	\end{align*}
	This completes the proof.
\end{proof}

According to the above analysis, we will replace the functional
$E_a$ in \eqref{minimization} with a positive functional $H_a$, while
extending the minimizing region from ``$K_a(\varphi)=0$'' to ``$K_a(\varphi)\leq 0$''. Let
\begin{align*}
	H_a(\varphi):= \left(1 - \frac{\mathcal{T}}{\bar{\mu}}\right) E_a(\varphi)
	=&\int_{\R^d} \left( \frac{1}{2d} \big|\nabla \varphi\big|^2 +\frac{a}{2d}\frac{|\varphi|^2}{|x|^2}
	+\frac{1}{2d} \big|\varphi\big|^{\frac{2d}{d-2}} \right)\;
	dx,
\end{align*}
then for any $\varphi \in H_a^1 \backslash\{0\}$, we have
\begin{align*}
	H_a(\varphi) > 0 , \quad  \mathcal{T} H_a(\varphi) \geq 0.
\end{align*}

Now we can characterization the minimization problem \eqref{minimization} by
use of $H_a$.
\begin{lemma}\label{minimization:H} For the minimization $m_a$ in \eqref{minimization}, we
	have
	\begin{align}
		m_a  =& \inf  \{ H(\varphi)\; |\;  \varphi \in H_a^1(\R^d), \; \varphi\not =0,\;
		K_a(\varphi) \leq 0  \} \nonumber\\
		=& \inf \{ H(\varphi)\; |\;\varphi \in H_a^1(\R^d), \; \varphi\not =0, \;
		K_a(\varphi)< 0 \}. \label{JEqualH}
	\end{align}
\end{lemma}
\begin{proof} For any $ \varphi\in H_a^1$, $\varphi \not=0$ with $K_a(\varphi)=0$, we have
	$E_a(\varphi)=H_a(\varphi)$, this implies that
	\begin{align}
		m_a= & \inf \{ E_a(\varphi)\; |\; \varphi \in H_a^1(\R^d),\; \varphi\not=0,\;
		K_a(\varphi)=0 \}  \nonumber\\
		\geq & \inf \{ H(\varphi)\; |\; \varphi \in
		H_a^1(\R^d), \; \varphi\not=0, \; K_a(\varphi) \leq 0 \}.\label{JLargeH}
	\end{align}
	
	On the other hand, for any $\varphi \in H_a^1$, $\varphi\not=0$ with
	$K_a(\varphi)<0$, by Lemma \ref{positive near origin:KQ }, Lemma \ref{Postivity:K}
	and the continuity of $K_a$ in $\lambda$, we know that there
	exists a $\lambda_0<0$ such that
	\begin{align*}
		K_a(\varphi^{\lambda_0}_{d,-2})=0,
	\end{align*}
	then by $\mathcal{T} H_a \geq 0$, we have
	\begin{align*}
		E_a(\varphi^{\lambda_0}_{d,-2}) = H_a (\varphi^{\lambda_0}_{d,-2}) \leq
		H_a(\varphi^{0}_{d,-2})=H_a(\varphi).
	\end{align*}
	Therefore,
	\begin{align}
		&  \inf \{ E_a(\varphi)\; |\; \varphi \in H_a^1(\R^d), \; \varphi\not=0, \;
		K_a(\varphi)=0 \}  \nonumber\\
		& \leq \inf \{ H_a(\varphi)\; |\;  \varphi \in
		H_a^1(\R^d),\; \varphi\not=0, \; K_a(\varphi) < 0 \}.\label{JSmallH}
	\end{align}
	By \eqref{JLargeH} and \eqref{JSmallH}, we have
	\begin{align*}
		& \inf  \{ H(\varphi)\; |\;  \varphi \in H_a^1(\R^d), \; \varphi\not =0,\;
		K_a(\varphi) \leq 0  \} \\
		&\leq  m_a
		\leq
		\inf \{ H_a(\varphi)\; |\;\varphi \in H_a^1(\R^d), \; \varphi\not =0, \;
		K_a(\varphi)< 0 \}.
	\end{align*}
	
	In order to show \eqref{JEqualH}, it suffices to show that
	\begin{align}
		&\inf \{ H(\varphi)\; |\;  \varphi \in H_a^1(\R^d),\; \varphi\not=0, \; K_a(\varphi)
		\leq 0 \} \nonumber\\
		& \geq \inf \{ H(\varphi)\; |\; \varphi \in H_a^1(\R^d),\; \varphi\not=0,
		\; K_a(\varphi)< 0 \}. \label{minimization:H:larger}
	\end{align}
	For any $\varphi \in H_a^1$, $\varphi\not=0$ with $K_a(\varphi)\leq 0$. By Lemma
	\ref{structure:J}, we know that
	\begin{align*}
		\mathcal{T} K_a(\varphi)= \bar{\mu}K_a(\varphi)- \int_{\R^d}\left(\frac{8}{d-1}\big|\nabla
		\varphi\big|^2+\frac{8a}{d-1}\frac{|\varphi|^2}{|x|^2} + \frac{8d}{(d-1)(d-2)}\big|\varphi\big|^{\frac{2d}{d-2}} \right) \; dx <0,
	\end{align*}
	then for any $\lambda>0$ we have
	\begin{align*}
		K_a(\varphi^{\lambda}_{d,-2})<0,
	\end{align*}
	and as $\lambda\rightarrow 0$
	\begin{align*}
		H_a(\varphi^{\lambda}_{d,-2})=\int_{\R^d} \left(\frac{e^{4\lambda}}{2d}
		\big|\nabla \varphi \big|^2  + \frac{e^{\frac{4d}{d-2}\lambda}}{2d}  \big|\varphi\big|^{\frac{2d}{d-2}} \right) \; dx
		\longrightarrow H_a(\varphi).
	\end{align*}
	This shows \eqref{minimization:H:larger}, and completes the proof.
\end{proof}

Next we will use the ($\dot H_a^1$-invariant) scaling argument to remove the $L^{\frac{2d+2}{d-1}}$ term (the lower regularity quantity than $\dot H_a^1$) in $K_a$,
that is, to replace  the constrained
condition $K_a(\varphi) < 0$ with $K_a^c(\varphi) < 0$, where
\begin{align*}
	K_a^{c}(\varphi):= \int_{\R^d} \left( 2 |\nabla \varphi|^2+2a\frac{|\varphi|^2}{|x|^2}
	-2|\varphi|^{\frac{2d}{d-2}} \right)\; dx.
\end{align*}

In fact, we have

\begin{lemma}\label{minimization:Hc}For the minimization $m_a$ in \eqref{minimization}, we
	have
	\begin{align*} m_a =& \inf \{ H_a(\varphi)\; |\; \varphi \in
		H_a^1(\R^d),\; \varphi\not=0, \;
		K_a^c(\varphi) < 0 \}\\
		=& \inf \{ H_a(\varphi)\; |\; \varphi \in   H_a^1(\R^d),\; \varphi\not =0, \;
		K_a^c(\varphi) \leq  0 \}.
	\end{align*}
\end{lemma}
\begin{proof}Since
	$
	K_a^c(\varphi) \leq K_a(\varphi)$, it is obvious that
	\begin{align*}
		m_a = & \inf \{ H_a(\varphi)\; |\;  \varphi \in H_a^1(\R^d),\; \varphi\not=0, \;
		K_a(\varphi)< 0 \}  \\
		\geq & \inf \{ H_a(\varphi)\; |\; \varphi \in
		H_a^1(\R^d),\; \varphi\not =0, \; K_a^c(\varphi) < 0 \}.
	\end{align*}
	Hence in order to show the first equality, it suffices to show that
	\begin{align}
		& \inf \{ H_a(\varphi)\; |\; \varphi \in H_a^1(\R^d),\; \varphi\not=0,  \; K_a(\varphi)<
		0 \} \nonumber\\
		&  \leq \inf \{ H_a(\varphi)\; |\;  \varphi \in   H_a^1(\R^d),\; \varphi\not=0,
		\; K_a^c(\varphi) < 0 \}.\label{minimization:Hc:smaller}
	\end{align}
	To do so, for any $ \varphi \in H_a^1$, $\varphi\not=0$ with $ K_a^c(\varphi) < 0
	$, taking
	\begin{align*}
		\varphi^{\lambda}_{d-2,-2}(x)=e^{(d-2)\lambda}\varphi(e^{2\lambda}x),
	\end{align*}
	we have $  \varphi^{\lambda}_{(d-2),-2} \in H_a^1$ and $ \varphi^{\lambda}_{d-2,-2}\not=0$ for any
	$\lambda>0$. In addition,   we have
	\begin{align*}
		K_a(\varphi^{\lambda}_{d-2,-2})=  \int_{\R^d} \left( 2\big|\nabla \varphi\big|^2+2a\frac{|\varphi|^2}{|x|^2}
		- 2\big|\varphi\big|^{\frac{2d}{d-2}} + \frac{2d}{d+1} e^{-\frac{4}{d-1}\lambda} \big|\varphi\big|^{\frac{2d+2}{d-1}} \right) \; dx & \longrightarrow K_a^c(\varphi),\\
		H_a(\varphi^{\lambda}_{d-2,-2})= \int_{\R^d} \left( \frac{1}{2d} \big|\nabla
		\varphi\big|^2 + \frac{1}{2d}
		\big|\varphi\big|^{\frac{2d}{d-2}} \right)\; dx = & H_a(\varphi),
	\end{align*}
	as $\lambda \rightarrow +\infty$. This gives
	\eqref{minimization:Hc:smaller}, and completes the proof of the
	first equality.
	
	For the second equality, it is obvious that
	\begin{align*}
		& \inf \{ H_a(\varphi)\; |\; \varphi \in   H_a^1(\R^d),\; \varphi \not=0, \;
		K^c(\varphi) < 0 \} \\
		&\geq \inf \{ H_a(\varphi)\; |\;  \varphi
		\in H_a^1(\R^d),\; \varphi\not=0, \; K_a^c(\varphi) \leq  0 \},
	\end{align*}
	hence we only need to show that
	\begin{align}
		& \inf \{ H_a(\varphi)\; |\;  \varphi \in   H_a^1(\R^d),\; \varphi\not=0, \;
		K_a^c(\varphi) < 0 \} \nonumber \\
		& \leq \inf \{ H_a(\varphi)\; |\;  \varphi
		\in H_a^1(\R^d), \; \varphi\not=0, \; K_a^c(\varphi) \leq  0 \}.\label{minimization:Hc:smaller II}
	\end{align}
	To do this, we use the ($L^2$-invariant) scaling argument. For any $  \varphi \in H_a^1$, $\varphi\not=0$ with $K_a^c(\varphi)\leq
	0$, we have $  \varphi^{\lambda}_{d,-2}\in H_a^1$, $\varphi^{\lambda}_{d,-2}\not=0$. In addition, by
	\begin{align*}
		\mathcal{T} K_a^c (\varphi) = & \int_{\R^d} \left(8\big|\nabla \varphi
		\big|^2+8\frac{|u|^2}{|x|^2}-\frac{8d}{d-2} \big|\varphi \big|^{\frac{2d}{d-2}} \right)\; dx = 4 K^{c}(\varphi)
		-\frac{16}{d-2}\big\|\varphi\big\|^{\frac{2d}{d-2}}_{L^{\frac{2d}{d-2}}}<0,\\
		& H_a(\varphi^{\lambda}_{d,-2}) =  \int_{\R^d} \left(
		\frac{e^{4\lambda}}{2d}
		\big|\nabla \varphi \big|^2 + \frac{e^{\frac{4d}{d-2}\lambda}}{2d}
		\big|\varphi\big|^{\frac{2d}{d-2}}\right) \; dx,
	\end{align*}
	we have $K_a^c(\varphi^{\lambda}_{d,-2})<0$ for any $\lambda>0$, and
	\begin{align*}
		H_a(\varphi^{\lambda}_{d,-2}) \rightarrow H_a(\varphi),\;\;
		\text{as}\;\; \lambda \rightarrow 0.
	\end{align*}
	This implies \eqref{minimization:Hc:smaller II} and completes the
	proof.
\end{proof}

Now, we  are in position to compute the minimization $m_a$ of \eqref{minimization} by  making use of the sharp Sobolev inequality in $\eqref{Sharp-Sobolev}$.

\begin{lemma}\label{threshold} For the minimization $m_a$ in \eqref{minimization}, we
	have
	\begin{align*}
		m_a=E_a^c(W_a).
	\end{align*}
\end{lemma}
\begin{proof} By Lemma \ref{minimization:Hc}, we have
	\begin{align*}
		m_a= & \inf \left\{\frac{1}{2d} \int_{\R^d} \left(|\nabla \varphi|^2 +a\frac{|\varphi|^2}{|x|^2}+
		|\varphi|^{\frac{2d}{d-2}}\right)\; dx \; \Big| \;  \varphi \in H_a^1,\; \varphi\not=0,\;
		\big\| \varphi \big\|^2_{\dot{H}_a^1} \leq
		\big\|\varphi\big\|^{\frac{2d}{d-2}}_{L^{\frac{2d}{d-2}}}   \right\}
		\\
		\geq & \inf \bigg\{ \int_{\R^d} \frac{1}{2d} \left(|\nabla \varphi|^2 +a\frac{|\varphi|^2}{|x|^2}+
		|\varphi|^{\frac{2d}{d-2}}\right) + \frac{1}{2d} \left(|\nabla \varphi|^2+a\frac{|\varphi|^2}{|x|^2}-
		|\varphi|^{\frac{2d}{d-2}}\right)\; dx \; \Big|\;\\& \hspace{5ex} \varphi \in H_a^1,\; \varphi\not=0,\;\big\| \varphi \big\|^2_{\dot{H}_a^1} \leq
		\big\|\varphi\big\|^{\frac{2d}{d-2}}_{L^{\frac{2d}{d-2}}}  \bigg\}
	\end{align*}
	where the equality holds if and only if the minimization is taken
	by some $\varphi$ with $\big\| \varphi \big\|^2_{\dot{H}_a^1} =
	\big\|\varphi\big\|^{\frac{2d}{d-2}}_{L^{\frac{2d}{d-2}}}$. While
	\begin{align*}
		& \inf\left\{ \int_{\R^d} \frac{1}{d}  |\nabla \varphi|^2+a\frac{|\varphi|^2}{|x|^2}  \; dx \;
		\big| \;  \varphi \in H_a^1,\;\varphi\not=0,\; \big\| \varphi
		\big\|^2_{\dot{H}_a^1} \leq
		\big\|\varphi\big\|^{\frac{2d}{d-2}}_{L^{\frac{2d}{d-2}}}\right\} \\
		&  = \inf\left\{ \frac{1}{d} \big\|
		\varphi\big\|^2_{\dot{H}_a^1} \left(
		\frac{\big\|\varphi\big\|^2_{\dot{H}_a^1}}{\big\|\varphi\big\|^{\frac{2d}{d-2}}_{L^{\frac{2d}{d-2}}}}\right)^{\frac{d-2}{2}}\;
		\Big| \;   \varphi \in H_a^1, \; \varphi\not=0 \right\}\\
		& =\inf\left\{ \frac{1}{d}   \left(
		\frac{\big\|\varphi\big\|_{\dot{H}_a^1}}{\big\|\varphi\big\|_{L^{\frac{2d}{d-2}}}}\right)^{d}\;
		\Big| \;  \varphi \in H_a^1, \; \varphi\not =0 \right\}\\
		& = \inf\left\{ \frac{1}{d}   \left(
		\frac{\big\|\varphi\big\|_{\dot{H}_a^1}}{\big\|\varphi\big\|_{L^{\frac{2d}{d-2}}}}\right)^{d}\;
		\Big| \;  \varphi \in \dot H_a^1,\; \varphi\not=0\right\}= \frac{1}{d}
		\big(C_{GN}\big)^{-d},
	\end{align*}
	where we use the density property $H_a^1\hookrightarrow \dot H_a^1$ in the last second equality
	and that $C_{GN}$ is the sharp Sobolev constant in $\R^d$, that is,
	\begin{align*}
		\big\|\varphi\big\|_{L^{\frac{2d}{d-2}}}\leq C_{GN} \big\| \varphi
		\big\|_{\dot{H}_a^1},  \;\;\forall \; \varphi\in \dot H_a^1(\R^d),
	\end{align*}
	and the equality can be attained by the ground state $W_a$ of the following
	elliptic equation \begin{align*} -\Delta W_a+a\frac{|W_a|^2}{|x|^2} = |W_a|^{\frac{4}{d-2}}W_a.	
	\end{align*}
	This implies that $\frac{1}{d} \big(C_{GN}\big)^{-d}= E_a^c(W_a)$. The proof
	is finished.
\end{proof}
Combining Lemma \ref{minimization:Hc} and \ref{threshold}, we can deduce the characterization of $K_a^+/K_a^{-}$.
\begin{theorem}
	We define
	\begin{align*}
		\bar{\K_a}^+&=\{f \in H_a^1| E_a(f)<E_a^c(W_a),\ \| f\|_{\dot{H}_a^1}^2\leq \| W_a\|_{\dot{H}_a^1}^2\},\\
		\bar{\K_a}^-&=\{f \in H_a^1| E_a(f)<E_a^c(W_a),\ \| f\|_{\dot{H}_a^1}^2 > \| W_a\|_{\dot{H}_a^1}^2\},
	\end{align*}
	Then we have $\bar{\K_a}^+= \K_a^+$, $\bar{\K_a}^-= \K_a^-$.
\end{theorem}

\begin{proof}
	First we recall that
	\begin{align*}
		\K_a^+&=\{f \in H_a^1| E_a(f)<E_a^c(W_a),\ K_a(f)\geq0 \},\\
		\K_a^-&=\{f \in H_a^1|  E_a(f)<E_a^c(W_a),\ K_a(f)<0 \}.
	\end{align*}
	
	By the variational results of  Killip-Miao-Visan-Zhang-Zheng \cite{KillipMiaVisanZhangZheng}, we know that If
	\begin{align*}
		\| f\|_{\dot{H}_a^1}^2< \| W_a\|_{\dot{H}_a^1}^2,\ \ E_a^c(f)<E_a^c(W_a).
	\end{align*}
	Then
	\begin{align*}
		\| f\|_{\dot{H}_a^1}^2-\|f\|_{\frac{2d}{d-2}}^{\frac{2d}{d-2}}\geq 0,\ \ E_a^c(f)<E_a^c(W_a).
	\end{align*}
	This gives $\bar{\K_a}^+\subset \K_a^+$.
	
	On the other hand,  we clearly have $0\in \K_a^+,\bar{\K_a}^+$.
	
	Now let $f\in H_a^1, f\neq 0$, we have  $f_{d-2,-2}^\lambda\in H_a^1$ and $f_{d-2,-2}^\lambda \neq 0$ for any $\lambda>0$. In addition, we have
	\begin{align*}
		K_a(f_{d-2,-2}^\lambda)&=2(\| f\|_{\dot{H}_a^1}^2-\|f\|_{\frac{2d}{d-2}}^{\frac{2d}{d-2}})+\frac{2d}{d+1} e^{-\frac{\lambda}{d-1}}\|f\|_{\frac{2d+2}{d-1}}^{\frac{2d+2}{d-1}},\\
		H_a(f_{d-2,-2}^\lambda)&=H_a(f).
	\end{align*}
	By Lemma \ref{minimization:Hc} and direct calculation,
	\begin{align}
		&E_a^c(W_a)=\inf\{H_a(f)|f\in H_a^1,f\neq 0, K_a(f)\leq 0\},\\
		&E_a(f)=H_a(f)+\frac{1}{2d} K_a(f),  \label{EHK}   \\
		&E_a(f_{d-2,-2}^\lambda)\leq E_a(f)<E_a^c(W_a), \text{ ~for~ } \lambda >0, \label{opp}
	\end{align}
	and
	$$K_a(f_{d-2,-2}^\lambda)=K_a^c(f)+\frac{2d}{d+1} e^{-\frac{4\lambda}{d-2}}\|f\|_{\frac{2d}{d+1}}^{\frac{2d}{d+1}},$$
	we can deduce $\K_a^+\subset \bar{\K_a}^+$.
	In fact, if $K_a^c(f)< 0 $, there exists  $\lambda_0 \geq0$ such that $K_a(f^{\lambda_0}_{d-2,-2}))=0 $ (since $K_a(f_{d-2,-2}^{0})\geq 0$). Then, by the identity \eqref{EHK},  we have
	$E_a(f_{d-2,-2}^{\lambda_0})=H_a(f)\geq E_a^c(W_a),$
	which is opposite to \eqref{opp}.
	Then, the result $\K_a^+= \bar{\K_a}^+$ follows, which implies $\K_a^-= \bar{\K_a}^-$.
\end{proof}

After the computation of the minimization $m_a$ in
\eqref{minimization}, we next give some variational estimates.

\begin{lemma} \label{free-energ-equiva} For any $\varphi \in H_a^1$ with $K_a(\varphi)\geq 0$, we
	have
	\begin{align}\label{free energy}
		&\int_{\R^d} \left(\frac{1}{2d}\big|\nabla \varphi \big|^2  +\frac{a}{2d}\frac{|u|^2}{|x|^2} +\frac{1}{2d} \big| \varphi\big|^{\frac{2d}{d-2}}\right)  dx \notag\\
		&\leq
		E_a(\varphi) \leq \int_{\R^d} \left(\frac{1}{2}\big|\nabla \varphi \big|^2+\frac{a}{2}\frac{|u|^2}{|x|^2}
		+ \frac{d-1}{2d+2} \big| \varphi\big|^{\frac{2d+2}{d-1}}
		\right) dx.
	\end{align}
\end{lemma}
\begin{proof} On one hand, the right hand side of \eqref{free energy} is trivial.  On the other hand,
	by the definition of $E_a$ and $K_a$, we have
	\begin{align*}
		E_a(\varphi)= \int_{\R^d} \left(\frac{1}{2d}\big|\nabla \varphi \big|^2 +\frac{a}{2d}\frac{|u|^2}{|x|^2}+
		\frac{1}{2d} \big| \varphi\big|^{\frac{2d}{d-2}}\right)\;
		dx + \frac{1}{2d}K_a(\varphi),
	\end{align*}
	which implies the left hand side of \eqref{free energy}.
\end{proof}

At the last of this section, we give the uniform bounds on the
scaling derivative functional $K_a(\varphi)$ with the   energy
$E_a(\varphi)$ below the threshold $m_a$.

\begin{lemma}\label{uniform bound}
	For any $\varphi \in H_a^1$ with $E_a(\varphi)<m_a$.
	\begin{enumerate}
		\item If $K_a(\varphi)<0$, then
		\begin{align}\label{uniform:K:negative}
			K_a(\varphi)  \leq -\bar{\mu}\big(m_a-E_a(\varphi)\big)=\frac{2d}{d-2}\big(m_a-E_a(\varphi)\big).
		\end{align}
		\item If $K_a(\varphi)\geq 0$, then
		\begin{align}\label{uniform:K:positive}
			K_a(\varphi)&\geq \min\bigg(\bar{\mu}(m_a-E_a(\varphi)), \frac{2}{2d-3} \big\|\nabla
			\varphi \big\|^2_{L^2} +\frac{2a}{2d-3}\|\frac{\varphi}{x}\|^2_{L^2}\notag\\
			&\hspace{2ex} + \frac{2d}{(2d-3)(d+1)} \big\|\varphi\big\|^{\frac{2d+2}{d-1}}_{L^{\frac{2d+2}{d-1}}}
			\bigg)\notag\\ 
			&=\min\bigg(\frac{2d}{d-2}(m_a-E_a(\varphi)), \frac{2}{2d-3} \big\|\nabla
			\varphi \big\|^2_{L^2} +\frac{2a}{2d-3}\|\frac{\varphi}{x}\|^2_{L^2}\\&\hspace{2ex} + \frac{2d}{(2d-3)(d+1)} \big\|\varphi\big\|^{\frac{2d+2}{d-1}}_{L^{\frac{2d+2}{d-1}}}
			\bigg).\notag
		\end{align}
	\end{enumerate}
\end{lemma}
\begin{proof} By Lemma \ref{structure:J}, for any $\varphi \in H_a^1$, we have
	\begin{align*}
		\mathcal{T}^2 E(\varphi) = \bar{\mu} \mathcal{T} E_a(\varphi)- \frac{8}{d-1}\big\|\nabla \varphi\big\|^2_{L^2}-\frac{8a}{d-1}\big\|\frac{\varphi}{x} \big\|^2_{L^2} - \frac{8d}{(d-2)(d-1)}
		\big\|\varphi\big\|^{\frac{2d}{d-2}}_{L^{\frac{2d}{d-2}}}.
	\end{align*}
	Let $j(\lambda)=E_a(\varphi^{\lambda}_{d,-2})$, then we have
\begin{align}\label{diff J}
	j''(\lambda) =  \bar{\mu}j'(\lambda) - \frac{8e^{4\lambda}}{d-1}- \frac{8ae^{4\lambda}}{d-1}\big\|\frac{\varphi}{x} \big\|^2_{L^2} -\big\|\nabla \varphi\big\|^2_{L^2} -
	\frac{8de^{\frac{4d}{d-2}\lambda}}{(d-2)(d-1)}\big\|\varphi\big\|^{\frac{2d}{d-2}}_{L^{\frac{2d}{d-2}}}.
\end{align}

\noindent{\bf Case I:} If $K_a(\varphi)<0$, then by
\eqref{asymptotic:KQ}, Lemma \ref{Postivity:K} and the continuity of $K_a$ in
$\lambda$, there exists a negative number
$\lambda_0<0$ such that $K_a(\varphi^{\lambda_0}_{d,-2})=  0$, and
\begin{align*}
	K_a(\varphi^{\lambda}_{d,-2})<   0, \;\; \forall\; \; \lambda\in
	(\lambda_0, 0).
\end{align*}
By \eqref{minimization}, we obtain $j(\lambda_0)=E_a(\varphi^{\lambda_0}_{d,-2})
\geq m$. Now by integrating \eqref{diff J} over
$[\lambda_0, 0]$, we have
\begin{align*}
	\int^0_{\lambda_0} j''(\lambda)\; d\lambda \leq \bar{\mu}
	\int^0_{\lambda_0} j'(\lambda)\; d\lambda,
\end{align*}
which implies that
\begin{align*}
	K_a(\varphi)=j'(0)-j'(\lambda_0)\leq
	\bar{\mu}\left(j(0)-j(\lambda_0)\right)  \leq -\bar{\mu}
	\big(m_a-E_a(\varphi)\big),
\end{align*}
which implies \eqref{uniform:K:negative}.

\noindent{\bf Case II:} $K_a(\varphi) \geq 0$. We divide it
into two subcases:

When $2\bar{\mu} K_a(\varphi)\geq \frac{8d}{(d-2)(d-1)} \big\|\varphi\big\|^{\frac{2d}{d-2}}_{L^{\frac{2d}{d-2}}}$.
Since
\begin{align*}
	&\hspace{5ex}\frac{8d}{(d-2)(d-1)}\int_{\R^d}\big|\varphi\big|^{\frac{2d}{d-2}} \; dx\\& = -\frac{4d}{(d-2)(d-1)}K_a(\varphi) + \int_{\R^d}\bigg(\frac{8d}{(d-2)(d-1)} \big| \nabla \varphi \big|^2 \\
	&+\frac{8ad}{(d-2)(d-1)}\frac{|\varphi|^2}{|x|^2} + \frac{8d^2}{(d-2)(d-1)(d+1)}\big|  \varphi \big|^{\frac{2d+2}{d-1}}
	\bigg) \; dx,
\end{align*}
then we have
\begin{align*}
	2\bar{\mu} K_a(\varphi)&\geq -\frac{4d}{(d-2)(d-1)}K_a(\varphi) +\\
	& \int_{\R^d} \left(\frac{8d}{(d-2)(d-1)} \big|
	\nabla \varphi \big|^2 +\frac{8ad}{(d-2)(d-1)}\frac{|\varphi|^2}{|x|^2}+ \frac{8d^2}{(d-2)(d-1)(d+1)} \big|  \varphi \big|^{\frac{2d+2}{d-1}} \right) \; dx,
\end{align*}
which implies that
\begin{align*}
	K_a(\varphi)\geq \frac{2d}{d-3} \big\|\nabla
	\varphi \big\|^2_{L^2} +\frac{2d}{d-3} \big\|\frac{\varphi}{x}
	\big\|^2_{L^2}+  \frac{2d}{(2d-3)(d+1)} \big\|\varphi\big\|^{\frac{2d+2}{d-1}}_{L^{\frac{2d+2}{d-1}}}.
\end{align*}

When $2\bar{\mu}  K_a(\varphi) \leq \frac{8d}{(d-2)(d-1)} \big\|\varphi\big\|^{\frac{2d}{d-2}}_{L^{\frac{2d}{d-2}}}$.
By \eqref{diff J}, we have for $\lambda=0$
\begin{align}
	0< &  2 \bar{\mu}j'(\lambda) <
	\frac{8de^{\frac{4d}{d-2}\lambda}}{(d-2)(d-1)}\big\|\varphi\big\|^{\frac{2d}{d-2}}_{L^{\frac{2d}{d-2}}}, \nonumber\\
	j''(\lambda) =
	\bar{\mu}j'(\lambda) -& \frac{8e^{4\lambda}}{d-1} \big\|\nabla \varphi\big\|^2_{L^2}-\frac{8ae^{4\lambda}}{d-1}\big\|\frac{\varphi}{x} \varphi\big\|^2_{L^2}- \frac{8de^{\frac{4d}{d-2}\lambda}}{(d-2)(d-1)}\big\|\varphi\big\|^{\frac{2d}{d-2}}_{L^{\frac{2d}{d-2}}}
	\leq -\bar{\mu}j'(\lambda). \label{evolution j}
\end{align}
By the continuity of $j'$ and $j''$ in $\lambda$, we know that $j'$ is an accelerating decreasing function as $\lambda$ increases until $j'(\lambda_0)=0$ for some
finite number $\lambda_0>0$ and \eqref{evolution j} holds on $[0,
\lambda_0]$.

By
$
K_a(\varphi^{\lambda_0}_{d,-2})=j'(\lambda_0)=0,
$
we know that
\begin{align*}
	E_a(\varphi^{\lambda_0}_{d,-2})\geq m_a.
\end{align*}
Now integrating \eqref{evolution j} over $[0, \lambda_0]$, we obtain that
\begin{align*}
	-K_a(\varphi)=j'(\lambda_0)-j'(0) \leq -\bar{\mu} \big(j(\lambda_0)-j(0)\big)
	\leq -\bar{\mu} (m_a-E_a(\varphi)).
\end{align*}
This completes the proof.
\end{proof}

\section{Modulation analysis}\label{sec:Modulation}
In this section, we give the modulation analysis around the ground state $W_a(x)$, which is crucial in the proof of Theorem \ref{Threshold}.  We assume that $d\in\{3,4,5\}$, $0>a>-\left(\frac{d-2}{2}\right)^2+\left(\frac{2(d-2)}{d+2}\right)^2$ and $u:I\times\R^d\to\C$ is a solution to \eqref{NLS} with
\begin{equation}\label{22condition}
	E_a(u_{0})=E_a^{c}(W_a).
\end{equation}
We write $\delta(t):=\delta(u(t))$, where
\[
\delta(u):=\bigl|\| W_a\|_{\dot{H}_a^1}^{2}-\| u\|_{\dot{H}_a^1}^{2}\bigr|.
\]
Next,
we choose a small parameter $\delta_0$(to be determined below) and  define
\[
I_{0}=\{t\in I:\delta(u(t))<\delta_0\}.
\]

\begin{proposition}\label{Modilation11}
	There exist $\delta_{0}>0$ sufficiently small and functions $\theta: I_{0}\to \R$ and $\mu: I_{0}\to \R$ such that
	\begin{equation}\label{DecomU}
		u(t,x)=e^{i \theta(t)}[g(t)+\mu(t)^{\frac{d-2}{2}}W_a(\mu(t)x)]
	\end{equation}
	for $t\in I_0$, we have 
	\begin{align}
		&\|u(t)\|^{\frac{2d+2}{d-1}}_{L_{x}^{\frac{2d+2}{d-1}}} +\mu(t)^{-\frac{\beta(d-2)(d+1)}{d-1}}
		\lesssim\delta(t)^2\sim \|g(t)\|_{\dot{H}_a^{1}}^2,\label{Estimatemodu}\\
		&\left|\tfrac{\mu^{\prime}(t)}{\mu(t)}\right|\lesssim \mu^{2}(t)\delta(t).\label{EstimLaD}
	\end{align}
\end{proposition}

We begin with the following lemma. 

\begin{lemma}\label{ModulationLem}Assume that $u\in \dot{H}_{a}^{1}(\R^d)$ and $%
	E_a(u)=E_a^c(W_a)$. Then for any $\varepsilon>0$, there exists $\delta>0$ such that
	
	$$
	\delta(u(t))<\delta, \ \ \inf_{\theta\in\mathbb S^1, \mu>0}\|u-e^{i\theta}\mu^{-\frac{d-2}{2}}W_a(\tfrac x\mu)\|_{\dot{H}_a^1}<\varepsilon.
	$$
	
\end{lemma}

\begin{proof}
	We argue by contradiction. Suppose this lemma does not hold, then there must exist $\varepsilon _{0}>0$ and a sequence of $\dot H_a^1(\mathbb{R}^d)$ functions $\{u_n\}$ such that 
	\begin{equation}
		\text{ }E_a(u_{n})=E_a^c(W_a),\text{ }\delta(u_{n}(t))\to 0,  \label{vcW 1}
	\end{equation}%
	but
	\begin{equation}
		\inf_{\theta \in \mathbb{S}^1,\mu >0}\Vert u_n-e^{i\theta}\mu^{-\frac{d-2}{2}}W_a(\tfrac x\mu)\Vert _{%
			\dot{H}_{a}^{1}}>\varepsilon _{0}.  \label{vcW 2}
	\end{equation}
	Replacing $u_n$ by $u_n\cdot\frac{\|W_a\|_{\dot{H}_a^1}}{\|u_n\|_{\dot{H}_a^1}}$, we may assume
	\begin{gather}\label{guan}
		\|f_n\|_{\dot{H}_a^1}=\|W_a\|_{\dot{H}_a^1}, \quad \|f_n\|^{\frac{2d}{d-2}}_{\frac{2d}{d-2}}-\tfrac{(d-2)(d+1)}{(d-1)d}\|f_n\|^{\frac{2d+2}{d-1}}_{\frac{2d+2}{d-1}}\to \|W_a\|^{\frac{2d}{d-2}}_{\frac{2d}{d-2}}, \\
		\inf_{\theta\in\mathbb S^1,\mu>0}\|u_n-e^{i\theta}\mu^{-\frac{d-2}{2}}W_a(\tfrac x\mu)\|_{\dot{H}_a^1}>\eps_0.
	\end{gather}
	Applying Lemma \ref{LinearProfi} to $\{u_n\}$, we obtain 
	$$
	u_n=\sum_{j=1}^J\phi_n^j +r_n^J, 
	$$
	for each $J\in \{1, \cdots, J^*\}$ with the stated properties. 
		From the $\dot H_a^1$ decoupling in Lemma \ref{LinearProfi} and \eqref{guan} we have
		\begin{align}\label{933}
			\|W_a\|_{\dot{H}_a^1}^2= \lim_{n\to \infty}\biggl(\sum_{j=1}^J\|\phi_n^j\|_{\dot{H}_a^1}^2+\|r_n^J\|_{\dot{H}_a^1}^2\biggr)=\sum_{j=1}^J\|\phi^j\|_{X^j}^2+ \lim_{n\to \infty}\|r_n^J\|_{\dot{H}_a^1}^2.
		\end{align}
		Here $\|\cdot\|_{X^j}=\|\cdot\|_{\dot{H}_a^1}$ if $x_n^j\equiv 0$ and $\|\cdot\|_{X^j}=\|\cdot\|_{\dot H^1}$ if $\frac{|x_n^j|}{\lambda_n^j}\to \infty$.  As \eqref{933} holds for any $J$, we take a limit and get
		\begin{align}\label{lbd}
			\sum_{j=1}^{J^*} \|\phi^j\|_{X^j}^2
			\le \|W_a\|_{\dot{H}_a^1}^2. 
		\end{align}
		
		On the other hand, using the decoupling in $L_x^{\frac{2d}{d-2}}(\mathbb{R}^d)$, \eqref{guan} and the sharp Sobolev embedding, we have 
		\begin{align*}
			\|W_a\|_{\frac{2d}{d-2}}^{\frac{2d}{d-2}}=\lim_{n\to \infty}\|u_n\|_{\frac{2d}{d-2}}^{\frac{2d}{d-2}}-\tfrac{(d-2)(d+1)}{(d-1)d}\|u_n\|^{\frac{2d+2}{d-1}}_{\frac{2d+2}{d-1}}&\leq\sum_{j=1}^{J^*}\|\phi^j\|_{\frac{2d}{d-2}}^{\frac{2d}{d-2}}
			\leq \sum_{j=1}^{J^*}\|\phi^j\|_{\dot{H}_a^1}^{\frac{2d}{d-2}}\cdot \frac{\|W_a\|_{\frac{2d}{d-2}}^{\frac{2d}{d-2}}}{\|W_a\|_{\dot{H}_a^1}^{\frac{2d}{d-2}}},
		\end{align*}
		which implies
		\begin{align}\label{second}
			\|W_a\|_{\dot{H}_a^1}^{\frac{2d}{d-2}}\leq \sum_{j=1}^{J^*}\|\phi^j\|_{\dot{H}_a^1}^{\frac{2d}{d-2}}.
		\end{align}
		This together with \eqref{lbd} gives 
		\begin{align*}
			\biggl(\sum_{j=1}^{J^*} \|\phi^j\|_{X^j}^2
			\biggr)^{\frac{d}{d-2}}\leq \sum_{j=1}^{J^*}\|\phi^j \|_{\dot{H}_a^1}^{\frac{2d}{d-2}}. 
		\end{align*}
		Note also for $a<0$, $\|\phi\|_{\dot{H}_a^1}<\|\phi\|_{\dot H^1}$, this obviously implies that 
		\[
		J^*=1, \quad x_n^1 \equiv 0, \text{ and } \limsup_{n\to \infty} \|r_n^1\|_{\frac{2d}{d-2}} =0. 
		\]
		Therefore, \eqref{lbd} and \eqref{second} imply $\|\phi^1\|_{\dot{H}_a^1}=\|W_a\|_{\dot{H}_a^1}, \ \|\phi^1\|_{\frac{2d}{d-2}}=\|W_a\|_{\frac{2d}{d-2}}$. Moreover 
		$$
		u_n=(\lambda_n)^{-\frac {d-2}{2}}\phi^1\bigl(\frac x{\lambda_n}\bigr)+r_n^1, \mbox{  \it{and} } \|r_n^1\|_{\dot{H}_a^1}\to 0 
		$$
		follow from \eqref{933}. Hence $\phi^1=e^{i\theta_0}\mu_0^{-\frac{d-2}{2}}W_a(\frac{x}{\mu_0})$ for some $\theta_0, \mu_0$. This contradicts to the last inequality in \eqref{guan}. 
	\end{proof}

	
	\begin{proposition}\label{Yinfinity}
		We can choose  sufficiently small $R>0$ so that if $\delta_{0}$ is sufficiently small, then
		\begin{equation}\label{Nobound}
			\mu_{0}(t)\geq R
		\end{equation}
		for $t$ in the lifespan of $u$.
	\end{proposition}
	
	\begin{proof}
		Suppose not, then there exists $\left\{t_{n}\right\}$ so that
		\begin{equation}\label{Bounddelta}
			\delta(t_{n})\to0\qtq{and}\mu_{0}(t_{n})\to 0 \qtq{as} n\to \N.
		\end{equation}
		From \eqref{ModulationLem} we then deduce that
		\begin{align}\label{Cv112}
			f_{n}(x):=	e^{-i \theta_{0}(t_{n})}\mu_{0}(t_{n})^{-\frac{d-2}{2}}u(t_{n}, \mu_{0}(t_{n})^{-1}x)\to W_a(x) \qtq{in} \dot{H}_a^{1}(\R^{d}).
		\end{align}
		By mass conservation and \eqref{Bounddelta}, we also get 
		\[
		\|f_{n}\|_{L^{2}}\leq \mu_{0}(t_{n})M(u_{0})^{\frac{1}{2}}\to 1
		, n \to \infty.
		\]
		But then by \eqref{ModulationLem} we obtain  $f_{n}\rightharpoonup W_a$ in $L^{2}(B_1(0))$ as $n\to \infty$.  By the proposition of weak convergence, we have
		\[
		\|W_a\|_{L^{2}(B_1(0))}\leq \lim_{n\to \infty}{\|f_{n}\|_{L^{2}(B_1(0))}}\to 0
		\]
		which is a contradiction.
	\end{proof}
	Arguing as in \cite[Lemma 4.2]{KYang-SIAM}, we may also obtain the following lemma:
	\begin{lemma}\label{ExistenceF}
		If  $\delta_{0}>0$ is sufficiently small, then there exist $C^{1}$ functions $\theta: I_{0}\to \R$ and $\mu: I_{0}\to [0, \infty)$  so that
		\begin{equation}\label{Taylor}
			\| u(t)-e^{i\theta(t)}\mu(t)^{\frac{d-2}{2}}W_a(\mu(t)\cdot)\|_{\dot{H}_a^{1}} \ll 1.
		\end{equation}
		Writing 
		\[
		g(t):=g_{1}(t)+i g_{2}(t)=e^{-i\theta(t)}[u(t)-e^{i\theta(t)}\mu(t)^{\frac{d-2}{2}}W_a(\mu(t)\cdot)],
		\]
		and $W_{1}^a:=\tfrac{d-2}{2}W_a+x\cdot\nabla W_a\in \dot{H}_a^{1}$, we have
		\begin{equation}\label{Ortogonality}
			\<\sqrt{\mathcal{L}_a} g_{2}(t), \sqrt{\mathcal{L}_a}[\mu(t)^{\frac{d-2}{2}} W_a(\mu(t)\cdot)] \>
			=\<\sqrt{\mathcal{L}_a} g_{1}(t), \sqrt{\mathcal{L}_a}[\mu(t)^{\frac{d-2}{2}} W_{1}^a(\mu(t)\cdot)]\>\equiv 0.
		\end{equation}
	\end{lemma}
	
	To analysis the situation when $\delta(t)$ is small,we  need to expand the energy around the ground state. In fact, we will prove that 
	\begin{equation}\label{Tay22}
		E_a^{c}(u(t))-E_a^{c}(W_a)
		=\F(\mu^{-\frac{d-2}{2}}(t)g(t, \mu^{-1}(t)x))+o(\|g\|^{2}_{\dot{H}_a^{1}}),
	\end{equation}
	where $\F$ is the quadratic form on $\dot{H}_a^{1}$ defined by
	\begin{align*}
	\F(h):=\frac{1}{2}\<(E_a^{c})^{\prime\prime}(W_a)[h], h\>&=
	\frac{1}{2}\int_{\R^{d}}|\nabla h|^{2}dx+\frac{a}{|x|^2}|h|^2dx\\
	&\hspace{2ex}-\frac{1}{2}\int_{\R^{d}}W_a^{\frac{4}{d-2}}[\frac{d+2}{d-2}|h_{1}|^{2}+|h_{2}|^{2}]dx,
	\end{align*}
	with  $h=h_{1}+ih_{2}\in \dot{H}_a^{1}$.
	
	To prove \eqref{Tay22}, first observe that using \eqref{Taylor}, we may write 
	\[
	\begin{split}
		&E_a^{c}(u(t))-E_a^{c}(W_a)=E_a^{c}(e^{-i\theta(t)}u(t)) -E_a^{c}(\mu^{\frac{d-2}{2}}(t)W_a(\mu(t)\cdot))\\
		&=\< (E_a^{c})^{\prime}(\mu^{\frac{d-2}{2}}(t)W_a(\mu(t)\cdot), g(t) \>
		+\frac{d-2}{2}\< (E_a^{c})^{\prime\prime}(\mu^{\frac{d-2}{2}}(t)W_a(\mu(t)\cdot)[g(t)], g(t)\>+
		o(\|g\|^{2}_{\dot{H}_a^{1}}).
	\end{split}
	\]
	As $(E_a^{c})^{\prime}(\mu^{\frac{d-2}{2}}(t)W_a(\mu(t)\cdot)=0$, we obtain \eqref{Tay22}.

	On the other hand, The quadratic form  $\F$ can be written
	\[
	\F(h)=\tfrac{1}{2}\<L_{1}h_{1},h_{1}\>+\tfrac{1}{2}\<L_{2}h_{2}, h_{2} \>, \qtq{where} h=h_{1}+ih_{2}\in\dot{H}_a^{1}(\R^{d})
	\]
	and $L_{1}$ and $L_{2}$ are two bounded linear operators defined on $\dot{H}_a^{1}(\R^{d})$ by
	\begin{align*}
		L_{1}u:=\mathcal{L}_a u- \frac{d+2}{d-2}W_a^{\frac{4}{d-2}}u,\quad L_{2}v&:=\mathcal{L}_a v-W_a^{\frac{4}{d-2}}v.	
	\end{align*}
	
Setting $H^a:=\text{span}\left\{W_a, iW_a,W_a+x\cdot \nabla W_a\right\}$ (viewed as a subspace of $\dot {H}_a^1$), we have the following (see \cite[Proposition~3.1]{KYang-SIAM}):
	\begin{lemma}\label{CoerS}
		There exists $C>0$ such that for radial $h=h_{1}+ih_{2}\in (H^a)^{\bot}$,
		\[
		\F(h)\geq C \|h\|^{2}_{\dot{H}_a^{1}}.
		\]
	\end{lemma}
	
	 begin proving the estimates appearing in Proposition~\ref{Modilation11}.
	
	\begin{lemma}\label{BoundI} Let $(\theta(t), \mu(t))$ and $g(t)$ be as in Lemma \ref{ExistenceF}.  Then
		\begin{equation}\label{DeltaBound}
			\int_{\R^{d}}|u(t,x)|^{\frac{2d+2}{d-1}}dx\lesssim \delta^{2}(t)\sim \|g(t)\|^{2}_{\dot{H}_a^{1}}.
		\end{equation}
	\end{lemma}
	\begin{proof}
		By \eqref{Tay22} and $E_a(u(t))=E_a^{c}(W_a)$,we see that
		\begin{equation}\label{Aprox}
			0=\F(\mu^{-\frac{1}{2}}(t)g(t, \mu^{-1}(t)\cdot))
			+\tfrac{d-1}{2d+2}\|u(t)\|^{\frac{2d+2}{d-1}}_{L^{\frac{2d+2}{d-1}}}+
			o(\|g\|^{2}_{\dot{H}_a^{1}}).
		\end{equation}
		We now decompose $g(t)$ as follows
		\begin{gather}\label{Defg}
			g(t)=\alpha(t)\mu(t)^{\frac{d-2}{2}} W_a(\mu(t)\cdot)+h(t),\\
			\alpha(t):=\frac{(\mu^{-\frac{d-2}{2}}(t)g(t, \mu^{-1}(t)\cdot), W_a)_{\dot{H}_a^{1}}}{(W_a,W_a)_{\dot{H}_a^{1}}}.
		\end{gather}
		Note that $\alpha\in \R$ is chosen to guarantee that
		\begin{align}\label{Alfor}
			(\mu(t)^{\frac{d-2}{2}} W_a(\mu(t)\cdot), h(t))_{\dot{H}_a^{1}}=0.
		\end{align}
		
		Using \eqref{Taylor} and \eqref{Defg} we first observe that 
		\begin{equation}\label{BoundR}
			|\alpha(t)|\lesssim \|g\|_{\dot{H}_a^{1}}\ll 1.
		\end{equation}
		Moreover, by definition of $h$ (cf. \eqref{Alfor}) and  \eqref{Ortogonality} we have $\mu^{-\frac{d-2}{2}}(t)h(t, \mu^{-1}(t)\cdot)\in (H^a)^{\bot}$.
		Thus, Lemma \ref{CoerS} implies
		\[
		\F(\mu^{-\frac{d-2}{2}}(t)h(t, \mu^{-1}(t)\cdot)) \gtrsim \|h\|^{2}_{\dot{H}_a^{1}}.
		\]
		Combining this with \eqref{Aprox} we deduce that
		\[
		\|h\|^{2}_{\dot{H}_a^{1}}+\|u\|^{\frac{2(d+1)}{d-1}}_{L^{\frac{2(d+1)}{d-1}}}\lesssim \alpha^{2}+|\alpha\langle L_{1}W_a, \mu^{-\frac{d-2}{2}}(t)h(t, \mu^{-1}(t)\cdot) \rangle|+
		o(\|g\|^{2}_{\dot{H}_a^{1}}).
		\]
		Notice that $L_{2}(W_a)=0$, so that $L_{1}W_a=\frac{4}{d-2}\mathcal{L}_a W_a$. Since (recalling \eqref{Alfor})
		\[
		\langle L_{1}W, \mu^{-\frac{d-2}{2}}(t)h(t, \mu^{-1}(t)\cdot)\rangle=-\tfrac{4}{d-2}(W_a, \mu^{-\frac{d-2}{2}}(t)h(t, \mu^{-1}(t)\cdot))_{\dot{H}_a^{1}}=0,
		\]
		the inequality above shows
		\[
		\|h\|^{2}_{\dot{H}^{1}}+\|u\|^{\frac{2(d+1)}{d-1}}_{L^{\frac{2(d+1)}{d-1}}}\lesssim \alpha^{2}+o(\|h\|^{2}_{\dot{H}_a^{1}}).
		\]
		In particular,
		\begin{equation}\label{BoundV}
			\|h(t)\|^{2}_{\dot{H}_a^{1}}\lesssim \alpha^{2}(t)
			\quad \text{and}\quad 
			\|u(t)\|^{\frac{2(d+1)}{d-1}}_{L^{\frac{2(d+1)}{d-1}}}\lesssim \alpha^{2}(t).
		\end{equation}
		
		On the other hand, by the orthogonality condition \eqref{Alfor} we obtain
		\[
		\|g(t)\|^{2}_{\dot{H}_a^{1}}=\alpha^{2}(t)\| W_a\|^{2}_{\dot{H}_a^{1}}+\|h(t)\|^{2}_{\dot{H}_a^{1}},
		\]
		which implies by \eqref{BoundR} and \eqref{BoundV} that $\|g\|_{\dot{H}_a^{1}}\sim |\alpha|$.
		
		Finally,  combining \eqref{Alfor}) and \eqref{BoundV} we obtain
		\begin{align*}
			\delta(t)&=\bigl|\|W_a\|^{2}_{\dot{H}_a^{1}}-\|[(1+\alpha)\mu(t)^{\frac{d-2}{2}} W_a(\mu(t)\cdot)+h]\|^{2}_{\dot{H}_a^{1}}\bigr|\notag\\
			&=2|\alpha| \|W_a\|^{2}_{\dot{H}_a^{1}}+\O(\alpha^{2}),
		\end{align*}
		which shows $\delta\sim |\alpha|$. Combining the estimates above, we have
		\[
		\|u(t)\|^{\frac{2(d+1)}{d-1}}_{L^{\frac{2(d+1)}{d-1}}}\lesssim \alpha^{2}(t)\sim \delta^{2}(t)\sim \|g(t)\|^{2}_{\dot{H}_a^{1}}.
		\]
	\end{proof}
	\begin{lemma}\label{BoundEx}
		Under the conditions of Lemma~\ref{BoundI}, if $\delta_{0}$ is sufficiently small, then
		\begin{equation}\label{EquaEx22}
			\frac{1}{\mu(t)}\lesssim \delta(t)^{\frac{2d-2}{\beta(d-2)(d+1)}} \quad \text{for all $t\in I_{0}$}.
		\end{equation}
	\end{lemma}
	\begin{proof}
		By Proposition~\ref{Yinfinity}, we have that if $\delta_{0}$ is sufficiently small, then
		\[
		\delta(t)<\delta_{0}\Rightarrow \mu(t)\geq R.
		\]
		Recall that 
		\begin{align}
			W_a(x)=[d(d-2)\beta^2]^{\frac{d-2}{4}}\biggl(\frac{|x|^{\beta-1}}{1+|x|^{2\beta}}\biggr)^{\frac{d-2}{2}},\; \beta=\sqrt{1+(\tfrac 2{d-2})^2a}. 
		\end{align} 
		We see that 
		\begin{align}\label{EstWl}
			\mu(t)^{\frac{d-2}{2}}W_a(\mu(t)x)\gtrsim (\mu(t))^{-\frac{\beta(d-2)}{2}} \mbox { for all $|x|\leq 1$ and $t\in I_{0}$.}
		\end{align}
		On the other hand, by H\"older and Sobolev, we get
		\begin{align*}
			&\|e^{-i\theta(t)}[u(t)-e^{i\theta(t)}\mu(t)^{\frac{d-2}{2}}W_a(\mu(t)\cdot)]\|_{L^{\frac{2d+2}{d-1}}(B(0,1))}\\
			&=\|g(t)\|_{L^{\frac{2d+2}{d-1}}
				(B(0,1))}\lesssim
			\|g(t)\|_{L^{\frac{2d}{d-2}}}
			\lesssim \|g(t)\|_{\dot{H}^{1}}\lesssim \|g(t)\|_{\dot{H}_a^{1}}\lesssim \delta(t).
		\end{align*}
		Thus, by \eqref{DeltaBound} we deduce
		\[
		\|\mu(t)^{\frac{d-2}{2}}W_a(\mu(t)\cdot)]\|_{L^{\frac{2d+2}{d-1}}(B(0,1))}\lesssim \delta(t)+\delta(t)^{\frac{d-1}{d+1}}
		\lesssim \delta(t)^{\frac{d-1}{d+1}}.
		\]
		Now, combining \eqref{EstWl} and the inequality above we get
		\[
		\mu(t)^{-\frac{\beta(d-2)(d+1)}{d-1}}\lesssim \|\mu(t)^{\frac{d-2}{2}}W_a(\mu(t)\cdot)]\|^{\frac{2d+2}{d-1}}
		_{L^{\frac{2d+2}{d-1}}(B(0,1))}
		\lesssim \delta(t)^{2}.
		\]
	\end{proof}
	\begin{lemma}\label{BoundEx22}
		Under the conditions of Lemma \ref{BoundI}, we have the following estimate
		\begin{equation}\label{EquaEx}
			|\tfrac{\mu^{\prime}(t)}{\mu(t)}|\lesssim \mu^{2}(t)\delta(t) \quad \text{for all $t\in I_{0}$}.
		\end{equation}
	\end{lemma}
	\begin{proof}
		We define following notation
		\[
		W^a_{[\mu(t)]}(x)=\mu(t)^{\frac{d-2}{2}} W_a(\mu(t)x),
		\quad
		W^a_{1, [\mu(t)]}(x)=\mu(t)^{\frac{d-2}{2}} W^a_{1}(\mu(t)x),
		\]
		where $W_1^a=\frac{d-2}{2}W_a+x\cdot \nabla W_a$. In this notation, we have \[
		g(t)=e^{-i\theta(t)}[u(t)-e^{i\theta(t)}W^a_{[\mu(t)]}].
		\] 
		
By equation \eqref{NLS}, we can write
		\begin{equation}\label{EquationG}
			\begin{split}
				i\partial_{t}g+\Delta g-\frac{a}{|x|^2}g-\theta^{\prime}g
				-\theta^{\prime}W^a_{[\mu(t)]}+i\tfrac{\mu^{\prime}(t)}{\mu(t)}W^a_{1, [\mu(t)]}\\
				-|u|^{\frac{4}{d-1}}e^{-i\theta}u+[f(e^{-i\theta} u)-f(W^a_{[\mu(t)]})]=0,
			\end{split}
		\end{equation}
		where $f(z)=|z|^{\frac{4}{d-2}}z$. Next we observe that for $\varphi$ with $\mathcal{L}_a \varphi \in L^{\frac{2d}{d-2}}$, we have 
		\begin{align}
			|(f(e^{-i\theta}u)&-f(W^a_{[\mu(t)]}), \varphi)_{\dot{H}_a^{1}}|\nonumber \\
			&=|(f(e^{-i\theta } u)-f(W^a_{[\mu(t)]}), \mathcal{L}_a \varphi)_{L^{2}}| \label{FIne}\\
			& \leq \|e^{-i\theta} u-W^a_{[\mu(t)]}\|_{L^{{\frac{2d}{d-2}}}}[\|e^{-i\theta} u\|^{\frac{2d}{d-2}}_{L^{{\frac{2d}{d-2}}}}+\|W_a\|^{\frac{2d}{d-2}}_{L^{{\frac{2d}{d-2}}}}]\|\mathcal{L}_a \varphi\|_{L^{{\frac{2d}{d-2}}}} \nonumber \\
			& \lesssim \delta(t)\|\mathcal{L}_a \varphi\|_{L^{{\frac{2d}{d-2}}}}.\nonumber
		\end{align}
		Moreover, by H\"older's inequality, $L^p$ interpolation, conservation of mass, and \eqref{DeltaBound},
		\begin{align}
			|(|u|^{\frac{4}{d-1}}u, \varphi)_{\dot{H}_a^{1}}| & =|(|u|^{\frac{4}{d-1}}u, \mathcal{L}_a\varphi)_{L^{2}}| \nonumber \\
			& \leq \||u|^{\frac{4}{d-1}}u\|_{L^{\frac{2d}{d-2}}}\|\mathcal{L}_a \varphi\|_{L^{{\frac{2d}{d-2}}}} \nonumber \\
			& \leq  \|u\|_{L^{\frac{2d+2}{d-1}}} \|u\|^{\frac{4}{d-1}}_{L^{\frac{4d^2+4d}{2d^2-d-1}}}\|\mathcal{L}_a \varphi\|_{L^{{\frac{2d}{d-2}}}}\nonumber \\
			& \leq {\|u\|}_{L^{\frac{2d+2}{d-1}}}^{\frac{(d+1)^2}{d^2-d}} {\|u\|}_{L^{2}}^{\frac{1}{d}}\nonumber\\
			& \lesssim\delta(t)^{\frac{d+1}{d-1}}\|\mathcal{L}_a \varphi\|_{L^{{\frac{2d}{d-2}}}} \nonumber\\
			& \lesssim\delta(t)\|\mathcal{L}_a \varphi\|_{L^{{\frac{2d}{d-2}}}}. \label{MassS}
		\end{align}
		Note also that
		\begin{align}\label{Dg1}
			&|(\mathcal{L}_a g, W^a_{[\mu(t)]})_{\dot{H}_a^{1}}|
			=|( g,\mathcal{L}_a W^a_{[\mu(t)]})_{\dot{H}_a^{1}}|\leq \mu^{2}(t)\|g\|_{\dot{H}^{1}}\|\mathcal{L}_a W_a\|_{\dot{H}_a^{1}}
			\lesssim \mu^{2}(t)\delta(t),\\\label{Dg2}
			&	|(\mathcal{L}_a g, W^a_{1, [\mu(t)]})_{\dot{H}_a^{1}}|
			\leq \mu^{2}(t)\|g\|_{\dot{H}_a^{1}}\|\mathcal{L}_a W_{1}\|_{\dot{H}_a^{1}}
			\lesssim \mu^{2}(t)\delta(t).
		\end{align}
		
		We will now prove that 
		\begin{align}\label{Ang}
			|\theta^{\prime}|\lesssim \mu^{2}(t)\delta(t)+|\tfrac{\mu^{\prime}(t)}{\mu(t)}|\delta(t).
		\end{align}
	In fact, by the orthogonality condition \eqref{Ortogonality} we deduce that
		\begin{align*}
			|(i\partial_{t} g, W^a_{[\mu(t)]})_{\dot{H}_a^{1}}|&=|\IM( g, \partial_{t} W^a_{[\mu(t)]})_{\dot{H}_a^{1}}| \\
			&=\left|\tfrac{\mu^{\prime}(t)}{\mu(t)}\right||\IM( g, W^a_{[1, \mu(t)]})_{\dot{H}_a^{1}}|\lesssim\left|\tfrac{\mu^{\prime}(t)}{\mu(t)}\right|\delta(t).
		\end{align*}
		Thus, taking the $\dot{H}_a^{1}$ inner product of \eqref{EquationG} with $W^a_{[\mu(t)]}$ and using \eqref{Dg1}, \eqref{MassS}, and \eqref{FIne} (with $\varphi=W^a_{[\mu(t)]}$), we obtain
		\[
		|\theta^{\prime}|\lesssim  |\theta^{\prime}|\delta(t)+ \mu^{2}(t)\delta(t)+|\tfrac{\mu^{\prime}(t)}{\mu(t)}|\delta(t),
		\]
		where we have used that $\|\mathcal{L}_a W^a_{[\mu(t)]}\|_{\dot{H}_a^{1}}=\mu^{2}(t)\|\mathcal{L}_a W_a\|_{\dot{H}_a^{1}}$.  As $\delta(t)\ll 1$, the estimate above yields \eqref{Ang}.
		
		Next we claim that
		\begin{align}\label{Dlam}
			\left|\tfrac{\mu^{\prime}(t)}{\mu(t)}\right|\lesssim
			|\theta^{\prime}|\delta(t) + \mu^{2}(t)\delta(t)+|\tfrac{\mu^{\prime}(t)}{\mu(t)}|\delta(t).
		\end{align}
		Indeed, taking the $\dot{H}_a^{1}$ inner product of \eqref{EquationG} with $i\,W^a_{1, [ \mu(t)]}$, the estimates
		\eqref{MassS} and \eqref{FIne} (with $\varphi=i\, W^a_{1, [ \mu(t)]}$) and \eqref{Dg2} yield
		\[
		\left|\tfrac{\mu^{\prime}(t)}{\mu(t)}\right|\lesssim
		|(i\partial_{t} g, i W^a_{1, [\mu(t)]})_{\dot{H}_a^{1}}|+ \mu^{2}(t)\delta(t)
		+|\theta^{\prime}|\delta(t).
		\]
		In addition, by the orthogonality condition \eqref{Ortogonality} we have
		\begin{align*}
			|(i&\partial_{t} g, i W^a_{1, [\mu(t)]})_{\dot{H}_a^{1}}|=|( g, \partial_{t} W^a_{1, [\mu(t)]})_{\dot{H}_a^{1}}|\\
			&\lesssim|\tfrac{\mu^{\prime}(t)}{\mu(t)}|\|g\|_{\dot{H}_a^{1}}\|\tfrac{d-2}{2}W_1^a+x\cdot\nabla W_1^a\|_{\dot{H}_a^{1}}\lesssim |\tfrac{\mu^{\prime}(t)}{\mu(t)}|\delta(t),
		\end{align*}
		which implies claim \eqref{Dlam}. Putting together \eqref{Ang} and \eqref{Dlam} yields \eqref{EquaEx}.
	\end{proof}
	
	Proposition~\ref{Modilation11} now follows from Lemmas~\ref{ExistenceF}, \ref{BoundI}, \ref{BoundEx}, and \ref{BoundEx22}.

	\section{Blow-up for sub-threhold solution:the proof of Theorem \ref{sub-threshold}(ii)}\label{sec:blowup-sub}

	In this section, we prove the
	blow-up result of Theorem \ref{sub-threshold}(ii). 
	
	By Lemma \ref{VirialIden},
	we have
	\begin{align*}
		\partial^2_t V_R(t)
		= &\; 4  \int_{\R^d} w_R''(r) \big|\nabla u(t,x)\big|^2+\frac{ax}{|x|^4}\cdot\nabla w_R|u|^2 \;dx -\int_{\R^d} (\Delta^2 w_R)(x) |u(t,x)|^2 \; dx  \\
		&  - \frac4d \int_{\R^d} (\Delta w_R)  |u(t,x)|^\frac{2d}{d-2}\; dx +\frac{4}{d+1}
		\int_{\R^d} (\Delta w_R)  |u(t,x)|^\frac{2d+2}{d-1}\; dx\\
		\leq &\; 4 \int_{\R^d} \left( 2 |\nabla u (t)|^2+2\frac{a}{|x|^2}|u(t)|^2 -2|u (t)|^\frac{2d}{d-2} +\tfrac{2d}{d+1}
		|u (t)|^\frac{2d+2}{d-1} \right)\; dx \\
		& +  \frac{c}{R^2}\int_{R\leq |x|\leq 3R} \big| u (t) \big|^2 \; dx + c \int_{R\leq |x|\leq 3R} \left( \big| u (t) \big|^\frac{2d+2}{d-1}
		+ \big| u(t) \big|^\frac{2d}{d-2} \right) \; dx.
	\end{align*}
	
	By the  radial Sobolev inequality, we have
	\begin{align*}
		\big\|f\big\|_{L^\infty(|x|\geq R)} \leq &c R^{-\frac{d-1}{2}} \big\|f\big\|^{1/2}_{L^2(|x|\geq R)}\big\|\nabla f\big\|^{1/2}_{L^2(|x|\geq R)}.
	\end{align*}
	Therefore, by mass conservation and Young's inequality, we know that for any $\epsilon>0$ there exist sufficiently large $R$ such that
	\begin{align}
		\partial^2_t V_R(t) \leq & 4 K_a(u(t))
		+ \epsilon \big\|u(t,x)\big\|_{\dot{H}_a^1}^2 + \epsilon^2 \nonumber\\
		= & \frac{16d}{d-2} E_a(u) - \big(\frac{16}{d-2}-\epsilon\big)\big\| u(t)\big\|_{\dot{H}_a^1}^2 - \frac{8d}{(d+1)(d-2)}\big\|u(t)\big\|^\frac{2d+2}{d-1}_{L^\frac{2d+2}{d-1}} + \epsilon^2.\label{videntity:second der}
	\end{align}
	By $K_a(u)<0$, mass and energy conservations, Lemma \ref{uniform bound} and
	the continuity argument, we know that for any $t\in I$, we have
	\begin{align*}
		K_a(u(t)) \leq -\frac{2d}{d-2}\left(m_a-E_a(u(t))\right)<0.
	\end{align*}
	By Lemma \ref{minimization:H}, we have
	\begin{align*}
		m_a \leq H_a(u(t))< \frac1d \big\|u(t)\big\|^\frac{2d}{d-2}_{L^\frac{2d}{d-2}},
	\end{align*}
	where we have used the fact that $K_a(u(t))<0$ in the second inequality. By the fact $m=\frac1d \left(C_{GN}\right)^{-d}$ and the sharp Sobolev inequality, we have
	\begin{align*}
		\big\| u(t)\big\|^\frac{2d}{d-2}_{\dot{H}_a^1} \geq \left(C_{GN}\right)^{-\frac{2d}{d-2}} \big\|u(t)\big\|^\frac{2d}{d-2}_{L^\frac{2d}{d-2}} > \left(dC_{GN}\right)^{\frac{d}{d-2}},
	\end{align*}
	which implies that $\big\| u(t)\big\|^2_{\dot{H}_a^1} > dm_a$.
	
	In addition, by $E_a(u_0)<m_a$ and energy conservation, there exists $\delta_1>0$ such that
	$E_a(u(t))\leq (1-\delta_1)m_a$. Thus, if we choose $\epsilon$ sufficiently small, we have
	\begin{align*}
		\partial^2_t V_R(t) \leq \frac{16}{d-2} (1-\delta_1)m_a - d\big(\frac{16}{d-2}-\epsilon\big) m_a +   \epsilon^2 \leq -\frac{8d}{d-2} \delta_1 m_a,
	\end{align*}
	which implies that $u$ must blow up at finite time. \qed
		\section{ Blow-up for threshold solution:the proof of Theorem \ref{Threshold}(ii)}\label{sec:blowup-threshold}
	In this section we prove Theorem~\ref{Threshold}~(ii). Before giving the proof, we first give a notation which will be used in this section. Denote by $W_{[\lambda(t)]}^a:=\lambda(t)^{\frac{d-2}{2}}W_a(\lambda(t)\cdot)$. We proceed by proving several lemmas. We still use the local virial identity  in subsection \ref{virialargument}.
	
	\begin{lemma}\label{Lemma11}
		Suppose $u$ is a solution to \eqref{NLS} satisfying
		\begin{equation}\label{EnCon1}
			E_a(u)=E_a^{c}(W_a)\qtq{and}
			\| u_{0}\|_{\dot{H}_a^1}>\| W_a\|_{\dot{H}_a^1}.
		\end{equation}
		If $u$ is forward-global, i.e. $[0,+\infty)\subset I$, then
		\begin{align}\label{TBu}
			|I_{R}[u(t)]|\lesssim R^{2} \delta(t) \qtq{for all $t\geq 0$.}
		\end{align}
	\end{lemma}
	
	\begin{proof}  First, we fix a $\delta_1\in(0,\delta_0)$ sufficiently small  and prove the estimate  in the following two cases $\delta(u(t))\geq \delta_1$ or $\delta(u(t))<\delta_1$ separately.  In the first case, we  use the H\"older and Hardy inequality along with Sobolev embedding to estimate
		\begin{align*}
			|I_R[u(t)]| &\lesssim R^2\|\nabla u(t)\|_{L^2}^2 \\
			& \lesssim R^2\{\delta(t) + \| W_a\|_{\dot{H}_a^1}^2\} \lesssim R^2\{1+\delta_1^{-1}\| W_a\|_{\dot{H}_a^1}^2\}\delta(t). 
		\end{align*}
		In the second case, since $W_a$ is real, thus we can  deduce to
		\begin{align*}
			|I_{R}[u(t)]|&\leq \left|2\IM\int_{\R^{d}}\nabla w_{R}(\bar{u} \nabla u-e^{-i\theta(t)}W^a_{[\lambda(t)]} 
			\nabla [ e^{i\theta(t)}W^a_{[\lambda(t)]}])
			dx\right|\\
			&\lesssim
			R^{2}[\|u(t)\|_{\dot{H}^{1}_{x}}+\|W^a_{[\lambda(t)]}\|_{\dot{H}^{1}}]
			\|u(t)-e^{i\theta (t)}W^a_{[\lambda(t)]}\|_{\dot{H}^{1}}\\
			&\lesssim_{W_a}R^{2} \|g(t)\|_{\dot{H}_a^{1}}   \lesssim_{W_a}R^{2} \delta(t),
		\end{align*}
		where we use \eqref{Estimatemodu} with $\lambda(t):=\mu(t)$ for $t\in I_{0}$ in the last inequality.  \end{proof}
	
	In the following lemma, we use the same notation as in the proof of local virial identity, i.e. Section \ref{virialargument}.
	
	\begin{lemma}\label{LemmaB22}
		Under assumptions of Lemma~\ref{Lemma11}, if $u$ is global in forward direction and radial, then there exists $R_{1}\geq 1$ such that for $R\geq R_{1}$ and $t\geq 0$, the following estimate holds 
		\begin{align}\label{BNI}
			\frac{d}{dt}V_{R}[u(t)]&\leq -\frac{20-2d}{d-2} \delta(t).
		\end{align}
	\end{lemma}
	\begin{proof}
		We choose a parameter $R>1$ which will be determined below. For $\delta_{1}\in (0, \delta_{0})$ sufficiently small, we use the localized virial identity (see Section \ref{virialargument}) with $\chi(t)$ satisfying
		\[
		\chi(t)=
		\begin{cases}
			1& \quad \delta(t)<\delta_{1} \\
			0& \quad \delta(t)\geq \delta_{1}.
		\end{cases}
		\]
		Recall that $E_a^{c}(W_a)=\frac{1}{d}\| W_a\|^{2}_{\dot{H}_a^{1}}$. By the  kinetic energy condition \eqref{EnCon1},  we  have
		\begin{equation*}
			{{F^{c}_{\infty}}}[u(t)]=-\frac{16}{d-2}\delta(t)-\frac{8d^2-8d}{d^2-d-2}\|u(t)\|^{\frac{2d+2}{d-1}}_{L^{\frac{2d+2}{d-1}}}.
		\end{equation*}
		Then using Lemma~\ref{VirialModulate}, we get
		\begin{equation}\label{VirilaX}
			\tfrac{d}{dt}I_{R}[u(t)]=F^{c}_{\infty}[u(t)]+\EE(t)=-\frac{16}{d-2}\delta(t)-\frac{8d^2-8d}{d^2-d-2}\|u(t)\|^{\frac{2d+2}{d-1}}_{L^{\frac{2d+2}{d-1}}}+\EE(t)
		\end{equation}
		with
		\begin{equation}\label{Error11}
			\EE(t)=
			\begin{cases}
				F_{R}[u(t)]-F^{c}_{\infty}[u(t)]& \quad  \text{if $\delta(t)\geq \delta_{1}$}, \\
				F_{R}[u(t)]-F^{c}_{\infty}[u(t)]-\M[t]& \quad \text{if $\delta(t)< \delta_{1}$},
			\end{cases}
		\end{equation}
		where
		\begin{equation}\label{Error22}
			\M(t):=F^{c}_{R}[e^{i\theta(t)}\lambda(t)^{\frac{d-2}{2}}W_a(\lambda(t)x)]-F^{c}_{\infty}[e^{i\theta(t)}\lambda(t)^{\frac{d-2}{2}}W_a(\lambda(t)x)]
		\end{equation}
		and $\lambda(t):=\mu(t)$ for $t\in I_{0}$. 
		Now by \eqref{DecomU} we have ${u}_{[\theta(t), \lambda(t)]}=W_a+V_a$, where 
		\[
		{u}_{[\theta(t), \lambda(t)]}=e^{-i\theta(t)}\lambda(t)^{-\frac{d-2}{2}}u(t, \lambda(t)^{-1}\cdot)\qtq{and}\|V_a\|_{\dot{H}^{1}}\sim \|V_a\|_{\dot{H}_a^{1}}\sim \delta(t).
		\]
		
		We now claim that
		\begin{align}\label{InfP}
			\lambda_{\text{inf}}:=\inf\left\{{\lambda(t):\ t\geq 0,\ \delta(t)\leq \delta_{1}}\right\}>0
		\end{align}
		for $\delta_{1}\in (0, \delta_{0})$ sufficiently small. Indeed,  the mass conservation implies that
		\begin{align*}
			M({u}_{0})&\gtrsim \int_{|x|\leq \tfrac{1}{\lambda(t)}}|u(x,t)|^{2}dx
			=\tfrac{1}{\lambda(t)^{2}}\int_{|x|\leq 1}|u_{[\theta(t), \lambda(t)]}|^{2}dx\\
			&\gtrsim \tfrac{1}{\lambda(t)^{2}}\(\int_{|x|\leq 1}|W_a|^{2}dx-\int_{|x|\leq 1}|V_a|^{2}dx\).
		\end{align*}
		Since
		\[
		\|V_a(t)\|_{L^{2}(|x|\leq 1)}\lesssim \|V_a(t)\|_{L^{\frac{2d}{d-2}}(|x|\leq 1)}  
		\lesssim \|V_a\|_{\dot{H}_a^{1}}\lesssim\delta(t),
		\]
		it follows that
		\[
		M({u}_{0})\gtrsim  \tfrac{1}{\lambda(t)^{2}}\(\int_{|x|\leq 1}|W_a|^{2}dx-C\delta^{2}(t)\).
		\]
		Choosing $\delta_{1}$ sufficiently small yields \eqref{InfP}.
		
		Next we show that there exists $R_{*}$ so that for $R\geq R_{*}$,
		\begin{align}
			\label{EstimateE22}
			|\EE(t)|\leq 2\delta(t)\quad &\text{ for $t\geq 0$ with $\delta(t)< \delta_{1}$}.
		\end{align}
		To simplify the notation, we set  $\Tilde W_a(t)=e^{i\theta(t)}\lambda(t)^{\frac{d-2}{2}}W_a(\lambda(t)x)$. Then
		\begin{align}\label{Decomp11}
			\EE(t)&=-8\int_{|x|\geq R}[(|\nabla u|^{2}-(|\nabla \Tilde W_a(t)|^{2})]\,dx
			+8\int_{|x|\geq R}[|u|^{\frac{2d}{d-2}}-|\Tilde W_a(t)|^{\frac{2d}{d-2}}]\,dx\\\label{Decomp33}
			&\quad+\int_{R\leq |x| \leq 2R}(-\Delta \Delta w_{R})[|u|^{2}-|\Tilde W_a(t)|^{2}]\,dx \\\label{Decomp44}
			&\quad+4\int_{|x|>R}(%
			\frac{a}{|x|^{4}}x\nabla w _{R}|u(t)|^{2}-\frac{a}{|x|^{4}}x\nabla w _{R}|W(t)|^{2} %
			)dx.\\ \label{Decomp77}
			&\quad-4\int_{|x|>R}(%
			\frac{2a|u(t)|^{2}}{|x|^{2}}-\frac{2a|\Tilde W_a(t)|^{2}}{|x|^{2}}%
			)dx.\\ \label{Decomp88}
			&\quad-\tfrac{4}{d}\int_{|x|\geq R}\Delta[w_{R}(x)](|u|^{\frac{2d}{d-2}}-|\Tilde W_a(t)|^{\frac{2d}{d-2}})\,dx\\\label{Decomp55}
			&\quad+4\sum_{j,k}\RE\int_{|x|\geq R}[\overline{u_{j}} u_{k}-\overline{\partial_{j}\Tilde W_a(t)} \partial_{k}\Tilde W_a(t)]\partial_{jk}[w_{R}(x)]\,dx\\ \label{Decomp66}
			&\quad+\frac{4}{d+1}\int_{\R^{d}}\Delta[w_{R}(x)]|u|^{\frac{2d+2}{d-1}}\,dx
		\end{align}
		for all $t\geq 0$ such that $\delta(t)<\delta_{1}$.
		
		As $|\Delta \Delta w_{R}|\lesssim 1/|x|^{2}$,  $|\nabla w_R|\lesssim |x|$ , $|\partial_{jk}[w_{R}]|\lesssim 1$ and $|\Delta[w_{R}]|\lesssim 1$, we obtain that \eqref{Decomp11}--\eqref{Decomp55} can be estimated by terms of the form 
		\begin{align}
			&\bigg\{\|u(t)\|_{\dot{H}_{x}^{1}(|x|\geq R)}+\bigg\|\frac{u(t)}{x}\bigg\|_{L^2(|x|\geq R)}+\|\tilde{W}_a(t)\|_{\dot{H}_{x}^{1}(|x|\geq R)}+ \bigg\|\frac{\tilde{W}_a(t)}{x}\bigg\|_{L^2(|x|\geq R)}\notag\\
			&+
			\|u(t)\|^{\frac{d+2}{d-2}}_{L_{x}^{\frac{2d}{d-2}}(|x|\geq R)}+\big\|\tilde{W}_a(t)\big\|^{\frac{d+2}{d-2}}_{L_{x}^{\frac{2d}{d-2}}(|x|\geq R)}
			\bigg\}\|g(t)\|_{\dot{H}_{a}^{1}},
		\end{align}
		where $ g(t)=e^{-i\theta(t)}[{u}(t)-\Tilde W_a(t)]$. Moreover, since 
		\begin{gather*}
			\|W_a\|_{\dot{H}_{x}^{1}(|x|\geq R)}=O\big(R^{-\frac{\beta(d-2)}{2}}\big),\quad\|W_a\|_{L_{x}^{\frac{2d}{d-2}}(|x|\geq R)} = O\big(R^{-\frac{\beta(d-2)}{2}}\big),\\
			\bigg\|\frac{W_a}{x}\bigg\|_{L_{x}^{2}(|x|\geq R)} = O\big(R^{-\frac{\beta(d-2)}{2}}\big),
		\end{gather*}
		we deduce that
		\begin{align*}
			&\|u(t)\|_{\dot{H}_{x}^{1}(|x|\geq R)}+\bigg\|\frac{u(t)}{x}\bigg\|_{L^2(|x|\geq R)}+\|\tilde{W}_a(t)\|_{\dot{H}_{x}^{1}(|x|\geq R)}+ \bigg\|\frac{\tilde{W}_a(t)}{x}\bigg\|_{L^2(|x|\geq R)}\notag\\
			&+
			\|u(t)\|^{\frac{d+2}{d-2}}_{L_{x}^{\frac{2d}{d-2}}(|x|\geq R)}+\big\|\tilde{W}_a(t)\big\|^{\frac{d+2}{d-2}}_{L_{x}^{\frac{2d}{d-2}}(|x|\geq R)}\\
			&\leq \delta(t)+\delta(t)^{\frac{\beta(d+2)}{2}}+O\big((\lambda(t)R)^{-\frac{\beta(d+2)}{2}}\big)+O\big((\lambda(t)R)^{-\frac{\beta(d+2)}{2}}\big).
		\end{align*}
		Thus, for $\delta_{1}\in (0, \delta_{0})$ sufficiently small and $R$ sufficiently large,
		\[
		|\eqref{Decomp11}|+|\eqref{Decomp33}|+|\eqref{Decomp44}|+|\eqref{Decomp77}|+|\eqref{Decomp88}|+|\eqref{Decomp55}|\leq \delta(t) \qtq{when $\delta(t)<\delta_{1}$.}
		\]
		By taking $\delta_{1}$ smaller if necessary, Proposition~\ref{Modilation11} implies
		\[
		\tfrac{4}{d+1}\left|\int_{\R^{d}}\Delta[w_{R}(x)]|u|^{\frac{2d+2}{d-1}}dx\right|\lesssim \delta(t)^{2}\lesssim {\delta}_{1} \delta(t)\leq \delta(t),
		\]
		Combining the estimates above, we have \eqref{EstimateE22}.
		
		Now, suppose that $\delta(t)\geq \delta_{1}$. We show that for  large $R$, we have
		\begin{align}\label{EstimateE11}
			\EE(t)\leq \delta(t) +\frac{8d}{d+1}\|u(t)\|^{\frac{2d+2}{d-2}}_{L^{\frac{2d+2}{d-2}}}.
		\end{align}
		First, we recall the radial Sobolev embedding estimate:
		\[
		\|f\|_{L^{\infty}(|x|\geq R)}\lesssim R^{-\frac{d-1}{2}}\|f\|^{1/2}_{L^{2}}\|\nabla f\|^{1/2}_{L^{2}}.
		\]
		For $\delta(t)\geq \delta_1$, we may write
		\begin{align}\label{FF11}
			\EE(t)=&\int_{|x|\geq R}(- \Delta \Delta w_{R})|u|^{2}
			+4\RE \overline{u_{j}} u_{k} \partial_{jk}[w_{R}]-8|\nabla u|^{2}dx\\ \label{FF22}
			&+\int_{|x|>R}(%
			\frac{4a}{|x|^{4}}x\nabla \phi _{R}|u(t)|^{2}-\frac{8a|u(t)|^{2}}{|x|^{2}}%
			)dx.\\ \label{FF33}
			&-\tfrac{4}{d}\int_{|x|\geq R}\Delta[w_{R}(x)]|u|^{\frac{2d}{d-2}}dx
			+8\int_{|x|\geq R}|u|^{\frac{2d}{d-2}}dx
			\\ \label{FF44}
			&+\tfrac{4}{d+1}\int_{\R^{d}}\Delta[w_{R}(x)]|u|^{\frac{2d+2}{d-1}}dx.
		\end{align}
		Choosing $w_{R}$  so that $4\partial^{2}_{r}w_{R}\leq 8$, we have
		\[
		\int_{|x|\geq R}(4\RE \overline{u_{j}} u_{k} \partial_{jk}[w_{R}]-8|\nabla u|^{2})dx
		\leq 0.
		\]
		Moreover, since  $|\nabla w_R|\lesssim |x|$,$|\Delta \Delta w_{R}|\lesssim 1/|x|^{2}$, by \eqref{FF11}--\eqref{FF44} and radial Sobolev inequality,  we deduce
		\[
		\EE(t)\leq CR^{-2}+CR^{-\frac{2d-2}{d-2}}\|u(t)\|^{\frac{2}{d-2}}_{\dot{H}_a^{1}}
		+\frac{4}{d+1}\int_{\R^{d}}\Delta[w_{R}(x)]|u|^{\frac{2d+2}{d-1}}dx.
		\]
		Furthermore, 
		\begin{align*}
			\frac{4}{d+1}\int_{\R^{d}}\Delta[w_{R}(x)]|u|^{\frac{2d+2}{d-1}}dx&=\frac{8d}{d+1}\int_{|x|\leq R}|u|^{\frac{2d+2}{d-1}}dx+\int_{|x|\geq R}\Delta[w_{R}(x)]|u|^{\frac{2d+2}{d-1}}dx\\
			&\leq \frac{8d}{d+1}\int_{\R^{d}}|u|^{\frac{2d+2}{d-1}}dx+CR^{-2}\|u(t)\|^{\frac{2}{d-1}}_{\dot{H}_a^{1}}.
		\end{align*}
		Combining estimates above we deduce 
		\begin{align*}
			\EE(t)&\leq \frac{8d}{d+1}\int_{\R^{d}}|u|^{\frac{2d+2}{d-1}}dx+ CR^{-2}+CR^{-\frac{2d-2}{d-2}}\|u(t)\|^{\frac{2}{d-2}}_{\dot{H}_a^{1}}
			+CR^{-2}\|u(t)\|^{\frac{1}{2}}_{\dot{H}_a^{1}}\\
			&\leq \frac{8d}{d+1}\int_{\R^{d}}|u|^{\frac{2d+2}{d-1}}dx+C(\delta_{1}, \|\Tilde W_a\|_{\dot{H}_a^{1}})R^{-\frac{2}{d-1}}\delta(t)
		\end{align*}
		for $\delta(t)\geq \delta_{1}$.
		Thus \eqref{EstimateE11} holds for $R$ large.  Combining \eqref{EstimateE22} and \eqref{EstimateE11} now yields \eqref{BNI}.\end{proof}

	\begin{lemma}\label{NeD}
		Given ${u}$ as in Lemma~\ref{LemmaB22} , there exists $C>1$ and $c>0$ such that
		\begin{align}\label{DoQ11}
			\int^{\infty}_{t}\delta(s)\,ds \leq Ce^{-ct}\qtq{for all $t\geq 0$.}
		\end{align}
	\end{lemma}
	\begin{proof}
		Fix $R\geq R_{1}$ (where $R_1$ is as in Lemma~\ref{LemmaB22}). Writing
		\begin{align}\label{VrV}
			V_{R}(t)=\int_{\R^{d}}w_{R}(x)|u(t,x)|^{2}\,dx,
		\end{align}
		we have that $\tfrac{d}{dt}V_{R}(t)=I_{R}[{u}]$. By Lemma~\ref{LemmaB22}, we have 
		\[
		\tfrac{d^{2}}{dt^{2}}V_{R}(t)=\tfrac{d}{dt}I_{R}[u(t)]\leq -\tfrac{20-2d}{d-2}\delta(t).
		\] 
		Thus, since $\tfrac{d^{2}}{dt^{2}}V_{R}(t)<0$ and $V_{R}(t)>0$ for all $t\geq0$, it follows that  $I_{R}[u(t)]=\tfrac{d}{dt}V_{R}(t)>0$ for all $t\geq 0$. Thus \eqref{TBu} implies
		\[
		\tfrac{20-2d}{d-2}\int^{T}_{t}\delta(s)\,ds\leq -\int^{T}_{t}\tfrac{d}{ds}I_{R}[u(s)]\,ds
		=I_{R}[u(t)]-I_{R}[u(T)]\leq I_{R}[u(t)]\leq CR^{2}\delta(t).
		\]
		Sending $T\to \infty$, \eqref{DoQ11} now follows from Gronwall's inequality.
	\end{proof}
	
	As a direct consequence of Lemma~\ref{NeD}, we have the following result.
	
	\begin{corollary}\label{delace} Under the assumptions of Lemma~\ref{Lemma11}, there exists a sequence $t_{n}\to \infty$ such that $\delta(t_{n})\to 0$.
	\end{corollary}

	\begin{proof}[{Proof of Theorem~\ref{Threshold}}~(ii)]
		Suppose that $u$ is a solution to \eqref{NLS} and satisfy the condition in  Theorem~\ref{Threshold}~(ii). Without loss of generality, we assume that $u$ is forward global.  By Corollary~\ref{delace}, there exists increasing $t_{n}\to \infty$ such that $\delta(t_{n})\to 0.$  Moreover, by Proposition~\ref{Modilation11} we deduce that (with $\lambda(t)=\mu(t)$ for $t\in I_{0}$)
		\begin{align}\label{CWQ}
			e^{-i\theta(t_{n})}\lambda(t_{n})^{-\frac{d-2}{2}}u(t_{n}, \lambda(t_{n})^{-1}x)
			\to W_a(x) \qtq{in $\dot{H}_a^{1}$.}
		\end{align}
		Passing to a subsequence, we may assume that
		\[
		\lim_{n\to +\infty}\lambda(t_{n})=\lambda_{*} \in [0, \infty].
		\]
		We show that $\lambda_{*}<\infty$.  By contradiction, we assume that $\lambda_{*}=\infty$. From \eqref{CWQ} and a change of variables, we get that for any $C>0$,
		\[
		\|u(t_{n})\|_{{L}^{\frac{2d}{d-2}}(|x|\geq C)}\to 0\qtq{as $n\to +\infty$.}
		\]
		Fix $R\geq R_{1}$ (cf. Lemma~\ref{LemmaB22}) and $\epsilon>0$ and recall \eqref{VrV}. Using H\"older's inequality, we deduce 
		\begin{align*}
			V_{R}(t_{n})&=\int_{|x|\leq \epsilon}w_{R}(x)|u(t_n,x)|^{2}dx+\int_{\epsilon \leq |x|\leq 2R}w_{R}(x)|u(t_n,x)|^{2}dx\\
			&\lesssim \epsilon^2 \|u(t_n)\|_{\dot H_a^1}^2 + R^2\|u(t_n)\|_{L^{\frac{2d}{d-2}}(|x|\geq \epsilon)}^2.
		\end{align*}
		Sending $n\to \infty$ and then  $\epsilon \to 0$, we get that
		\[
		\lim_{n\to \infty}V_{R}(t_{n})=0.\ 
		\]
		However, since $\tfrac{d}{dt}V_{R}>0$ for all $t\geq 0$, it follows that $V_{R}(t)<0$ for $t\geq 0$, which is a contradiction.  Thus $\lambda_{*}<\infty$. 
		
		As $\delta(t_{n})\to 0$ as $n\to \infty$, there exists $N\in\N$ so that $\delta(t_{n})<\delta_{0}$ for $n\geq N$. Now observe that Proposition~\ref{Modilation11} yields
		\[
		\frac{1}{\lambda(t_{n})}\leq C\delta(t_{n})^{\frac{2d-2}{\beta(d-2)(d+1)}} \qtq{for all $n\geq N$.}
		\]
		Taking $n\to \infty$, we obtain a contradiction, so that $u$ blows up forward in time.  
		
		We can also obtain that $u$ blows up backward in time by reversing the time direction. Hence, we complete the proof. \end{proof}

	\section{Construction of minimal blowup solutions}\label{sec:minimal-blowup}
	
	The goal of this section is to show that if Theorem~\ref{sub-threshold}(i) fails, then we may construct a blowup solution with mass-energy in the region $\mathcal{K}_a$ that  obeys certain compactness properties.  In the next section, we will utilize a localized virial argument to preclude the possibility of such a solution, thus finishing the proof of  Theorem~\ref{sub-threshold}(i). 
	

	\subsection{Embedding nonlinear profiles}\label{Sec:embedding}
	
	In this subsection,we devote to construct non-scattering solutions to \eqref{NLS} with profiles $\phi_n$ living either at tiny length scales (i.e. in the regime $\lambda_n\to 0$) or far from the origin relative to their length scales (i.e. in the regime $|\tfrac{x_{n}}{\lambda_{n}}|\to \infty$\footnote[1]{As a consequence of radiality in dimension $3$, $x_n\equiv0$. Thus, this regime cannot occur. For the sake of uniformity in our proof, we will still discuss this regime in $d=3$. }), or both. Compared to   \eqref{NLS-critical}, the situation in \eqref{NLS} will be more complicate because  the translation symmetry  is broken by the potential and  scaling symmetry is broken by the double-power nonlinearity.So we must consider several limiting regimes and use approximation by a suitable underlying model in each case. The essential idea is that if $\lambda_n\to 0$, the disturbance term $|u|^{\frac{4}{d-1}}u$  will become negligible, while if $|\tfrac{x_n}{\lambda_n}|\to\infty$, the inverse-potential term will become negligible. To be more procisely: 
	\begin{itemize}
		\item If $\lambda_n\to 0$ and $x_n\equiv 0$, we will approximate by using solutions to the focusing energy-critical NLS with inverse-square potential, i.e. $\eqref{NLS-Potential}$. For the scattering result, see Theorem \ref{thm-NLS-Potential}.  
		\item If $\lambda_n\equiv 1$ and $|x_n|\to\infty$, we will  approximate using solutions to the  combined NLS $\eqref{Combine-NLS}$. For the scattering result of \eqref{Combine-NLS}, see Theorem \ref{thm-Combine-NLS}.
		\item If $\lambda_n\to 0$ and $|\tfrac{x_n}{\lambda_n}|\to \infty$, we will approximate by using solutions to the focusing energy-critical NLS without potential \eqref{NLS-critical}. For the scattering result of \eqref{NLS-critical}, see Theorem \ref{thm-NLS-critical}.
	\end{itemize}
	
	Now we state the key result  in this section:

	\begin{proposition}[Embedding nonlinear profiles]\label{embedding-nonlinear} Fix $a>-\frac{(d-2)^2}{4}+\left(\frac{d-2}{d+2}\right)^2$.

		Let $\L_a^n$  be as in \eqref{DefOperator} corresponding to sequence 
		$\{\tfrac{x_{n}}{\lambda_{n}}\}$, and let $\{t_n\}$ satisfy $t_n\equiv 0$ or $t_n\to\pm\infty$.
		
		\begin{itemize} 
			\item If $\lambda_n\equiv 1$, $|x_n|\to \infty$, then let $\phi\in H_a^1$ satisfy  $E_0(\phi)< E_0(W_0)$, $\|\phi\|_{\dot{H}^1}< \|  W_0\|_{\dot{H}_a^1} $and define
			\[
			\phi_{n}(x):=[e^{-it_{n}\L^{n}_{a}}\phi](x-x_{n}).
			\]
			\item If $\lambda_n\to 0$, $|\tfrac{x_n}{\lambda_n}|\to \infty$, then let $\phi\in \dot{H}^1$ satisfy  $E_0^c(\phi)< E_0^c(W_0)$, $\|\phi\|_{\dot{H}^1}< \|  W_0\|_{\dot{H}^1}$  
			\[
			\phi_{n}(x):=\lambda_{n}^{-\frac{d-2}{2}}[e^{-it_{n}\L^{n}_{a}} P^{a}_{\geq \lambda_{n}^{\theta}}\phi]( \tfrac{x-x_{n}}{\lambda_{n}}).
			\]
			
			\item If $\lambda_n\to 0$, $x_n\equiv0$, then let $\phi\in \dot{H}_a^1$ satisfy  $E_a^c(\phi)< E_a^c(W_a)$, $\|\phi\|_{\dot{H}_a^1}< \|  W_a\|_{\dot{H}_a^1}  $
			\[
			\phi_{n}(x):=\lambda_{n}^{-\frac{d-2}{2}}[e^{-it_{n}\L^{n}_{a}} P^{a}_{\geq \lambda_{n}^{\theta}}\phi]( \tfrac{x-x_{n}}{\lambda_{n}}).
			\]
		\end{itemize}
		
		Then for $n$ sufficiently large, there exists a global solution $v_n$ to \eqref{NLS} with
		\[
		v_n(0)=\phi_n \qtq{and} \|v_n\|_{L_{t,x}^{\frac{2(d+2)}{d-2}}(\R\times\R^d)}\lesssim 1,
		\]
		with the implicit constant depending on $\|\phi\|_{H^1}$ if $\lambda_n\equiv 1$ or $\|\phi\|_{\dot H^1}$ if $\lambda_n\to 0$. 
		
		Moreover, for any $\epsilon>0$ there exist $N=N(\epsilon)\in \N$ and a smooth compactly supported function
		$\chi_{\epsilon}\in C^{\infty}_{c}(\R\times \R^{d})$  such that for  $n\geq N$,
		\begin{align}\label{Comspt11}
			\Big\| v_{n}(t,x)-\lambda^{-\frac{d-2}{2}}_{n}\chi_{\epsilon}(\tfrac{t}{\lambda^{2}_{n}}+t_{n}, \tfrac{x-x_{n}}{\lambda_{n}})  \Big\|_{X(\R\times \R^{d})}
			&<\epsilon,
		\end{align}
		where 
		\[
		X\in\{L_{t,x}^{\frac{2(d+2)}{d-2}},L_{t,x}^{\frac{2(d+2)}{d}}, L^{\frac{2(d+2)}{d}}_{t}\dot{H}_{x}^{1,\frac{2(d+2)}{d}} \}.
		\]
	\end{proposition} 
	
	\begin{remark} In the scenario in which $\lambda_n\equiv 1$, the approximation in \eqref{Comspt11} may also be taken to hold in Strichartz spaces of $L^2$ regularity. 
	\end{remark}
	
	\begin{proof} We will label the three cases in the proposition \eqref{embedding-nonlinear}   as
		\begin{itemize}
			\item Scenario \uppercase\expandafter{\romannumeral1}: $\lambda_n\to 0$ and $x_n\equiv 0$. In this case, $\phi\in \dot H^1$;
			\item Scenario \uppercase\expandafter{\romannumeral2}: $\lambda_n\equiv 1$ and $|x_n|\to\infty$. In this case, $\phi\in H^1$;
			\item Scenario \uppercase\expandafter{\romannumeral3}: $\lambda_n\to 0$ and $|\tfrac{x_n}{\lambda_n}|\to\infty$. In this case, $\phi\in \dot H^1$.
		\end{itemize}
		for any  $\mu,\theta\in(0,1)$, we initially define
		\[
		\psi_n = \begin{cases} P_{>\lambda_n^\theta}^a \phi &\text{in Scenario \uppercase\expandafter{\romannumeral1},} \\ P_{\leq |x_n|^\mu}^a \phi & \text{in Scenario \uppercase\expandafter{\romannumeral2},} \\ P_{\lambda_n^\theta\leq \cdot<|\frac{x_n}{\lambda_n}|^\mu}^a \phi & \text{in Scenario \uppercase\expandafter{\romannumeral3}.}\end{cases}
		\]
		We also set
		\[
		H= \begin{cases} -\Delta & \text{in Scenarios \uppercase\expandafter{\romannumeral2} and \uppercase\expandafter{\romannumeral3}}, \\ \L_a & \text{in Scenario \uppercase\expandafter{\romannumeral1}}.
		\end{cases} 
		\]
		In the rest of the paper, we denote $H^\frac{1}{2}$ by $\sqrt{H}$. Here, $H^\frac{1}{2}$ is defined by standard functional calculus of $H$.

		\textbf{Establishment of approximate solutions, part 1.} We first define functions $w_n$ and $w$ as follows:
		
		If $t_n\equiv 0$, then we define $w_n$ and $w$ as the global solutions to an appropriate NLS model with initial data $\psi_n$ and $\phi$, respectively.  In particular, in Scenario \uppercase\expandafter{\romannumeral1}, we use the model \eqref{NLS-Potential} (focusing energy-critical NLS with inverse-square potential).  In Scenario \uppercase\expandafter{\romannumeral2}, we use the model \eqref{Combine-NLS} (Combined  NLS without potential). Finally, in Scenario \uppercase\expandafter{\romannumeral3}, we use the model \eqref{NLS-critical} (energy-critical NLS without potential). 
		
		If instead $t_n\to\pm\infty$, then we define $w_n$ and $w$ to be the solutions to the appropriate model (determined according to the three scenarios as above) satisfying
		\begin{equation}\label{wnscattering}
			\|w_n-e^{-itH}\psi_n\|_{\dot H^1} \to 0 \qtq{and} \|w-e^{-itH}\phi\|_{\dot H^1} \to 0
		\end{equation}
		as $t\to\pm\infty$.  Note that in either case (i.e. $t_n\equiv 0$ or $t_n\to\pm\infty$), $w$ has scattering states $w_\pm$ as $t\to\pm\infty$ in $\dot H^1$. 
		
		The solutions $w_n$ and $w$ we just established above  obey 
		\begin{equation}\label{wnscattering-bds}
			\|\sqrt{H} w_n\|_{S^0(\R)}+\|\sqrt{H} w\|_{S^0(\R)} \leq C(\|\phi\|_{\dot H^1})
		\end{equation}
		uniformly in $n$.  It also have the persistence of regularity at $L^2$ level,that is,
		\begin{equation}\label{Persistence-Regularity}
			\begin{cases}
				\|w_n\|_{S_a^0(\R)}\lesssim C(\|\phi\|_{\dot H_a^1})\lambda_n^{-\theta} & \text{in Scenario \uppercase\expandafter{\romannumeral1}}, \\
				\|w_n\|_{S^0(\R)} \lesssim C(\|\phi\|_{H^1}) & \text{in Scenario \uppercase\expandafter{\romannumeral2}}, \\
				\|w_n\|_{S^0(\R)}\lesssim C(\|\phi\|_{\dot H^1}) \lambda_n^{-\theta} & \text{in Scenario \uppercase\expandafter{\romannumeral3}},
			\end{cases}
		\end{equation}
		uniformly in $n$.  In Scenarios \uppercase\expandafter{\romannumeral2} and \uppercase\expandafter{\romannumeral3}, we may futher use persistence of regularity to derive the following bounds
		\begin{equation}\label{Persistence-Regularity-smalls}
			\| |\nabla|^s w_n\|_{\dot S^0(\R)} \lesssim \bigl|\tfrac{x_n}{\lambda_n}\bigr|^{s\mu} 
		\end{equation}
		for any $0<s<1$. 
		
Finally,by stability theory of equations \eqref{NLS-Potential}, \eqref{Combine-NLS} and \eqref{NLS-critical}, we may also derive that in each case
		\begin{equation}\label{stablity-lemma}
			\lim_{n\to\infty} \|\sqrt{H}[w_n-w]\|_{L_t^q L_x^r(\R\times\R^d)}=0\qtq{for all admissible}(q,r). 
		\end{equation}
		
		\textbf{Estabilishment of approximate solutions, part 2.} Next we define approximate solutions to \eqref{NLS} on $\R\times\R^d$:

		For each $n$, let $\chi_n$ be a $C^{\infty}$ function satisfying
		\[
		\chi_n(x)=\begin{cases} 0 & |x_n+\lambda_n x|<\tfrac14|x_n| \\ 1 & |x_n + \lambda_n x|>\tfrac12 |x_n|,\end{cases}\qtq{with} |\partial^k \chi_n(x)|\lesssim \bigl(\tfrac{\lambda_n}{|x_n|}\bigr)^{|k|}
		\]
		uniformly in $x$.Especially, $\chi_n(x)\to 1$ as $n\to\infty$ for each $x\in\R^d$. It's worth  noticing that in Scenario~\uppercase\expandafter{\romannumeral1}, we have $x_n\equiv 0$, so that $\chi_n(x)\equiv 1$ and the derivatives of $\chi_n$ vanish identically. 
		
	For $T\geq 1$, In Scenario \uppercase\expandafter{\romannumeral1},we define
	\[
	\tilde v_{n,T}(t,x)=	\lambda^{-\frac{d-2}{2}}_{n}[\chi_{n}w_{n}](\lambda^{-2}_{n}t, \lambda^{-1}_{n}(x-x_{n}))\]

	In Scenario \uppercase\expandafter{\romannumeral2} and \uppercase\expandafter{\romannumeral3},we will futher define
		\[
		\tilde v_{n,T}(t,x)=\begin{cases} 
			\lambda^{-\frac{d-2}{2}}_{n}[\chi_{n}w_{n}](\lambda^{-2}_{n}t, \lambda^{-1}_{n}(x-x_{n})), & |t|\leq \lambda^{2}_{n}T,\\
			e^{-i(t-\lambda^{2}_{n}T)\L_{a}}\tilde{v}_{n,T}(\lambda^{2}_{n}T,x),& t> \lambda^{2}_{n}T,\\
			e^{-i(t+\lambda^{2}_{n}T)\L_{a}}\tilde{v}_{n,T}(-\lambda^{2}_{n}T,x),& t<- \lambda^{2}_{n}T.
		\end{cases} 
		\]

		We emphasize that $\tilde v_{n,T}$ are intended to be approximate solutions to \eqref{NLS}, So we have  the following error terms:
		\[
		e_{n,T}:=(i\partial_{t}-\L_{a})\tilde{v}_{n,T}-|\tilde{v}_{n,T}|^{\frac{4}{d-2}}\tilde{v}_{n,T}+|\tilde{v}_{n,T}|^{\frac{4}{d-1}}\tilde{v}_{n,T}.
		\]

		\textbf{Conditions for stability lemma.} In order to apply the stability lemma \eqref{StabilityNLS},We need to establish the following estimates: 
		\begin{align}\label{11estimate}
			&\limsup_{T\to\infty}\limsup_{n\to\infty}\bigl\{\|\tilde{v}_{n,T}\|_{L^{\infty}_{t}H^{1}_{x}(\R\times\R^d)}
			+\|\tilde{v}_{n,T}\|_{L^{\frac{2(d+2)}{d-2}}_{t,x}(\R\times\R^d)}\bigr\}\lesssim 1,\\	\label{22estimate}
			&\limsup_{T\to\infty}\limsup_{n\to\infty}\|\tilde{v}_{n,T}\|_{L_{t,x}^{\frac{2(d+2)}{d-1}}(\R\times\R^d)}\lesssim 1,\\\label{11estimate1}
			&	\limsup_{T\to\infty}\limsup_{n\to\infty}\|\tilde{v}_{n,T}(\lambda^{2}_{n}t_{n})-\phi_{n}\|_{\dot H_a^1}=0,
			\\\label{33estimate}
			&\limsup_{T\to\infty}\limsup_{n\to\infty}\|\nabla
			e_{n,T}\|_{N(\R)}=0,\footnotemark[1]
		\end{align}
		\footnotetext[1]{Here we already use the fact that $\|\nabla
			e_{n,T}\|_{N(\R)}\sim\|\mathcal{L}_a^\frac{1}{2}e_{n,T}\|_{N(\R)}$}
		where space-time norms are over $\R\times\R^{d}$.

		\textbf{Proof of \eqref{11estimate} and \eqref{22estimate} (space-time uniform bounds).}  First, by definition of $\tilde v_{n,T}$, Strichartz estimate and \eqref{Persistence-Regularity}
		\begin{align*}
			\|\tilde v_{n,T}\|_{L_t^\infty L_x^2} \lesssim \lambda_n \|\chi_n\|_{L_x^\infty} \|w_n\|_{L_t^\infty L_x^2} \lesssim \lambda_n^{1-\theta}.
		\end{align*}
		Similarly, noting that $\chi_n\equiv 1$ in Scenario \uppercase\expandafter{\romannumeral1} and
		\[
		\|\nabla \chi_n\|_{L^\infty(\Bbb{R}^d)} \lesssim \tfrac{\lambda_n}{|x_n|}\to 0\qtq{as}n\to\infty
		\]
		and
		\[
		\|\nabla \chi_n\|_{L^d(\Bbb{R}^d)} \lesssim \tfrac{|x_n|}{\lambda_n}\tfrac{\lambda_n}{|x_n|}\lesssim1.
		\]
		in the remaining scenarios and using equivalence of Sobolev spaces, we may estimate 
		\begin{align*}
			\|\nabla& \tilde v_{n,T}\|_{L_t^{\frac{2(d+2)}{d-2}}L_x^{\frac{2d(d+2)}{d^2+4}}\cap L_t^\infty L_x^2} \\
			& \lesssim \| \nabla[\chi_n w_n]\|_{L_t^{\frac{2(d+2)}{d-2}} L_x^{\frac{2d(d+2)}{d^2+4}}\cap L_t^\infty L_x^2} + \|[\chi_n w_n](\pm \lambda_n^2 T)\|_{\dot H^1} \\
			&\lesssim  \|w_n\|_{L_t^{\frac{2(d+2)}{d-2}}\dot{H}^{1,\frac{2d(d+2)}{d^2+4}}\cap L_t^\infty \dot{H}^{1}}+\Vert\nabla\chi_n\Vert_{L^d(\Bbb{R}^d)}\Vert w_n\Vert_{L_{t,x}^{\frac{2(d+2)}{d-2}}} + \|\nabla \chi_n\|_{L^d}\|w_n\|_{L_t^\infty L_x^\frac{2d}{d-2}}\\
			&\hspace{2ex}+\|\chi_n\|_{L^\infty}\|\nabla w_n\|_{L_t^\infty L_x^2} \lesssim 1. 
		\end{align*}
		Thus, using Sobolev embedding as well, we derive \eqref{11estimate}. Finally,  one can deduce $\eqref{22estimate}$ from $\eqref{11estimate}$, persistence of the regularity(Lemma \eqref{PRegularity}) and $L^p$ interpolation.

		\textbf{Proof of \eqref{11estimate1} (agreement of initial data).}  
		First, if $t_n\equiv 0$, We change variables to obtain that
		\[
		\|\nabla[\tilde v_{n,T}(0)-\phi_n]\|_{L^2} = 
		\begin{cases}
			0 & \text{in Scenario \uppercase\expandafter{\romannumeral1}} \\
			\bigl\|\nabla[\chi_nP^a_{\leq |x_n|^\mu}\phi - \phi]\bigr\|_{L^2} & \text{in Scenario \uppercase\expandafter{\romannumeral2}} \\
			\bigl  \|\nabla[\chi_n P^a_{\lambda_n^\theta\leq\cdot<|\tfrac{x_n}{\lambda_n}|^\mu}\phi-P_{>\lambda_n^\theta}^a\phi]\bigr\|_{L^2} & \text{in Scenario \uppercase\expandafter{\romannumeral3}}.
		\end{cases}
		\]
		We will only treat the most complex Scenario \uppercase\expandafter{\romannumeral3} in detail and omit details for the simpler Scenario \uppercase\expandafter{\romannumeral2}. In Scenario \uppercase\expandafter{\romannumeral3} we  
		rewrite
		\begin{align}
			\chi_n  P_{\lambda_n^\theta\leq\cdot<|\frac{x_n}{\lambda_n}|^\mu}^a \phi - P^{a}_{>\lambda_n^\theta}\phi 
			& = (\chi_n-1)P_{>\lambda_n^\theta}^a \phi \label{1jda}\\
			& \quad - \chi_n P_{>|\frac{x_n}{\lambda_n}|^\mu}^a\phi \label{2jda}
		\end{align}
		
		For \eqref{1jda}, we apply the Leibniz rule and write
		\[
		\nabla\eqref{1jda} = \nabla\chi_n\cdot\phi+ (1-\chi_n)\nabla \phi  - \nabla\chi_n\cdot P_{\leq \lambda_n^\theta}^a\phi  - (1-\chi_n)\nabla P_{\leq \lambda_n^\theta}^a\phi. 
		\]
		For the first two terms, we have the following estimate:
		\begin{align*}
			\|\nabla \chi_n \phi + (1-\chi_n)\nabla\phi\|_{L^2} & \lesssim \|\nabla\chi_n\|_{L^d}\|\phi\|_{L^\frac{2d}{d-2}(\text{supp}(\nabla \chi_n))} + \|\nabla \phi\|_{L^2(\text{supp}(1-\chi_n))} \\
			& \lesssim \|\phi\|_{L^\frac{2d}{d-2}(\text{supp}(\nabla \chi_n))} + \|\nabla \phi\|_{L^2(\text{supp}(1-\chi_n))} = o(1)
		\end{align*}
		as $n\to\infty$ by Lebesgue's dominated convergence theorem. For the last two terms, we have
		\begin{align*}
			\|\nabla&\chi_n\cdot P_{\leq \lambda_n^\theta}^a\phi  + (1-\chi_n)\nabla P_{\leq \lambda_n^\theta}^a\phi\|_{L^2} \\
			& \lesssim \|\nabla \chi_n\|_{L^d}\|P_{\leq\lambda_n^\theta}^a\phi\|_{L^\frac{2d}{d-2}} + \|\nabla P_{\leq \lambda_n^\theta}^a\phi\|_{L^2} \\ 
			& \lesssim \|P_{\leq\lambda_n^\theta}^a\phi\|_{L^\frac{2d}{d-2}} + \|\sqrt{\L_a}P_{\leq\lambda_n^\theta}^a\phi\|_{L^2} = o(1)
		\end{align*}
		as $n\to\infty$ by a standard density argument and using the fact that $\lambda_n\to 0$. Finally we also apply the Leibniz rule  to \eqref{2jda} and then estimating as we just did for the last two terms shows that
		\[
		\|\nabla[\chi_n P_{>|\frac{x_n}{\lambda_n}|^\mu}^a \phi]\|_{L^2} \to 0 \qtq{as}n\to\infty,
		\]
		as well.  Thus we have estabilished 
		\[
		\lim_{n\to\infty} \|\nabla[\tilde v_{n,T}(0)-\phi_n]\|_{L^2} =0.
		\]
		 in the case $t_n\equiv 0$.
		 
		We next establish $\dot H^1$ convergence in the case $t_n\to+\infty$ (the case $t_n\to-\infty$ is handled similarly,so we omit it). As we did before, we first change variables to see that
		\begin{align*}
			\|\tilde v_{n,T}&(\lambda_n^2 t_n)-\phi_n\|_{\dot H_a^1} \\
			&=\begin{cases} \|\sqrt{\L_a}\bigl[w_n(t_n)-e^{-it_n\L_a}P^{a}_{>\lambda_n^\theta}\phi\bigr]\|_{L^2} & \text{in Scenario \uppercase\expandafter{\romannumeral1},} \\
				\|\sqrt{\L_a^n}\bigl[(\chi_n w_n)(T)-e^{-iT\L_a^n}P_{> \lambda_n^\theta}^a\phi\bigr]\|_{L^2} & \text{in Scenario \uppercase\expandafter{\romannumeral3}}, \\
				\|\sqrt{\L_a^n}\bigl[(\chi_n w_n)(T)-e^{-iT\L_a^n}\phi\bigr]\|_{L^2} & \text{in Scenario \uppercase\expandafter{\romannumeral2}}.
			\end{cases}
		\end{align*}
		
		In Scenario \uppercase\expandafter{\romannumeral1}, we have $P_{>\lambda_n^\theta}^a\phi = \psi_n$, thus by \eqref{wnscattering} we have
		\[
		\lim_{n\to\infty} \|\tilde v_{n,T}(\lambda_n^2 t_n)-\phi_n\|_{\dot H_a^1} = 0.
		\]

Again,we will only treat the most complex Scenario \uppercase\expandafter{\romannumeral3} in detail and omit details for the simpler Scenario \uppercase\expandafter{\romannumeral2}. Notice that by using the equivalence of Sobolev spaces and directly calculation, we have the following estimate:
		\begin{align}
			\|\tilde  v_{n,T}(\lambda_n^2 t_n)-\phi_n\|_{\dot H_a^1} 
			& \lesssim \|\nabla[\chi_n(w_n(T)-w(T))]\|_{L^2} \label{1jdaa} \\
			& \quad + \|\nabla[w(T)(\chi_n-1)]\|_{L^2} \label{2jdaa} \\
			& \quad + \|\sqrt{\L_a^n}[w(T)-e^{-iT\L_a^n}\phi]\|_{L^2}\label{3jdaa} \\
			& \quad + \| P_{\leq \lambda_n^\theta}^a\phi\|_{\dot H_a^1} \label{4jdaa}.
		\end{align}
We will treat \eqref{1jdaa}-\eqref{4jdaa} seperately.
		For \eqref{1jdaa}, we use H\"older's inequality and \eqref{stablity-lemma} to obtain
		\begin{align*}
			\eqref{1jdaa} & \lesssim \|\nabla \chi_n\|_{L^d}\|w_n(T)-w(T)\|_{L^\frac{2d}{d-2}} + \|\chi_n\|_{L^\infty}\|\nabla[w_n(T)-w(T)]\|_{L^2} \\
			& \to 0 \qtq{as}n\to\infty.
		\end{align*}
		 Using the same argument as above we can obtain the following estimare For \eqref{2jdaa}:
		\begin{align*}
			\|\nabla[w(T)(\chi_n-1)]\|_{L^2} & \lesssim \|\nabla w(T)\|_{L^2(\text{supp}(\chi_n-1))} + \|w(T)\|_{L^\frac{2d}{d-2}(\text{supp}(\nabla \chi_n))} \\
			& \to 0 \qtq{as}n\to\infty.
		\end{align*}
		To estimate \eqref{3jdaa}, we ned to decompose \eqref{3jdaa} further and first write
		\begin{align}
			\eqref{3jdaa} & \lesssim \|(\sqrt{\L_a^n}-\sqrt{H})w(T)\|_{L^2} + \|[\sqrt{\L_a^n}-\sqrt{H}]\phi\|_{L^2} \label{3jdaa1} \\
			& \quad + \|(e^{-iT\L_a^n}-e^{-iTH})\sqrt{H}\phi\|_{L^2} \label{3jdaa2} \\
			& \quad + \|\sqrt{H}(w(T)-e^{-iTH}\phi)\|_{L^2}\label{33jdaa}.
		\end{align}
		For \eqref{3jdaa1},it's easy to check that these terms converge to zero as $n\to\infty$ by using \eqref{Conver33}.  The term in \eqref{3jdaa2} converges to zero as $n\to\infty$ due to \eqref{Conver44}, while the term in \eqref{33jdaa} converges to zero as $T\to\infty$ due to \eqref{wnscattering}.  
		
		Finally for \eqref{4jdaa},one can use a stardard density argument and the fact that $\lambda_n\to 0$  to check that    \eqref{4jdaa} converges to zero as as $T\to\infty$. So we completes the proof of \eqref{22estimate}.

		\textbf{Proof of \eqref{33estimate} (control of error terms).} We will also prove each scenario separately.
		
		\underline{\emph{Proof of \eqref{33estimate}} in Scenario \uppercase\expandafter{\romannumeral1}}. In Scenario \uppercase\expandafter{\romannumeral1}, we have 
		\[
		e_n = \lambda_n^{-\frac{(d-2)(d+3)}{2(d-1)}} (|w_n|^\frac{4}{d-1} w_n)(\lambda_n^{-2}t,\lambda_n^{-1}x),
		\]
By a change of variables, \eqref{Persistence-Regularity}, and \eqref{wnscat-bds},we can deduce that
		\begin{align*}
			\|\nabla e_n\|_{L_t^2L_x^\frac{2d}{d+2}} & \lesssim \lambda_n^\frac{2}{d-1} \|w_n\|_{L_{t,x}^{\frac{2(d+1)}{d-1}}}^\frac{4}{d-1}\|\nabla w_n\|_{L_{t}^\frac{2(d+2)}{d-2}L_x^{\frac{2d(d+2)}{d^2+4}}} \\
			&\lesssim \lambda_{n}^\frac{2}{d-1}\|w_n\|_{L_{t,x}^{\frac{2(d+2)}{d-2}}}^\frac{2}{d-1}\| w_n\|_{L_{t,x}^{\frac{2(d+2)}{d}}}^\frac{2}{d-1}\|\nabla w_n\|_{L_{t}^\frac{2(d+2)}{d-2}L_x^{\frac{2d(d+2)}{d^2+4}}} \\
			& \lesssim \lambda_n^{\frac{2}{d-1}(1-\theta)} \to 0 \qtq{as}n\to\infty.
		\end{align*}		
		Thus \eqref{33estimate} holds in Scenario \uppercase\expandafter{\romannumeral1}.
		
		\underline{\emph{Proof of \eqref{33estimate} in Scenario \uppercase\expandafter{\romannumeral2}.}}  By the definition of $\tilde v_{n,T}$ and use the fact that $\lambda_n\equiv 1$, we will $\tilde v_{n,T}$ treat  in two  regions $|t|\leq T$ and $|t|>T$ separately. 
		
		First we recall that $w_n$ is a solution to \eqref{Combine-NLS}, so in the region $|t|\leq T$ we deduce the following estimate:
		\begin{align}
			e_{n,T}(t,x)
			&=[(\chi_{n}-\chi^{\frac{d+3}{d-1}}_{n})|w_{n}|^{\frac{4}{d-1}}w_{n}](t, x-x_{n})\label{11new}\\
			&\quad-[(\chi_{n}-\chi^{\frac{d+2}{d-2}}_{n})|w_{n}|^{\frac{4}{d-2}}w_{n}](t, x-x_{n})\label{22new}\\
			&\quad+2[\nabla \chi_{n}\cdot \nabla w_{n}](t, x-x_{n})
			+[\Delta \chi_{n} w_{n}](t, x-x_{n})\label{33new}\\
			&\quad-\tfrac{a}{|x|^{2}}[\chi_{n}w_{n}](t, x-x_{n}).\label{55new}
		\end{align}
		
		In the region $|t|>T$, we instead have
		\begin{align}\label{pertubation1}
			e_{n,T} = |\tilde v_{n,T}|^{\frac{4}{d-2}} \tilde v_{n,T} - |\tilde v_{n,T}|^\frac{4}{d-1} \tilde v_{n,T}.
		\end{align}
		
As before,we will estimate \eqref{11new}--\eqref{55new} on $[-T,T]\times\R^d$ seperately. 
		
		For \eqref{11new}, we apply a change of variables and H\"older's inequality to deduce that
		\begin{align*}
			\|\nabla& \eqref{11new}\|_{L_{t,x}^{\frac{2(d+2)}{d+4}}}
			\\ &\lesssim \|(\chi_{n}-\chi^{\frac{d+3}{d-1}}_{n})|w_{n}|^{\frac{4}{d-1}}\nabla w_{n}\|_{L_{t,x}^{\frac{2(d+2)}{d+4}}}+
			\|\nabla \chi_{n}(1-\tfrac{d+3}{d-1}\cdot\chi^{\frac{4}{d-1}}_{n})w_{n}^{\frac{d+3}{d-1}}\|_{L_{t,x}^{\frac{2(d+2)}{d+4}}}\\
			&\lesssim \bigl[ \|  \nabla w_{n} \|_{L_{t,x}^{\frac{2(d+2)}{d-2}}} \| w_{n} \|_{L_{t,x}^{\frac{2(d+2)}{d}}}^{\frac{2}{d-1}} 
			+\|\nabla\chi_{n}\|_{L^{\infty}_{x}} \| w_{n} \|_{L_{t,x}^{\frac{2(d+2)}{d}}} \| w_{n} \|_{L_{t,x}^{\frac{2(d+2)}{d-2}}}^{\frac{2}{d-1}} \bigr]\\
			&\quad \times\bigl[ \| w_{n}-w \|_{L_{t,x}^{\frac{2(d+2)}{d-2}}}^{\frac{2}{d-1}}+\|w \|_{L_{t,x}^{\frac{2(d+2)}{d-2}}([-T,T]\times\{|x+x_{n}|\leq \frac{|x_{n}|}{4}\} )}^{\frac{2}{d-1}}\bigr]
			\\ & \quad \to 0 \qtq{as}n\to\infty,
		\end{align*}
		where we have used the Lebesgue's dominated convergence theorem and \eqref{stablity-lemma}.  
	The terms \eqref{22new}, \eqref{33new}, and \eqref{55new} may be handled exactly as the same argument in \cite{Ardila-Murphy}, so we omit the details.
		
		For the term \eqref{pertubation1}, first by Sobolev embedding and Strichartz estimate, it's easy to show that
		\begin{equation}\label{esstential-LimitAx}
			\limsup_{T\to\infty}\limsup_{n\to\infty} \|e^{-it\L_a^n}[\chi_n w_n(T)]\|_{L_{t,x}^{\frac{2(d+2)}{d-2}}((0,\infty)\times\R^d)}\to 0.
		\end{equation}
	So we can deduce the estimates of terms in \eqref{pertubation1} as follows:
		\begin{align*}
			\|\nabla &  (| \tilde{v}_{n,T}|^{\frac{4}{d-1}} \tilde{v}_{n,T})\|_{L_{t,x}^{\frac{2(d+2)}{d+4}}
				(\left\{t>\lambda^{2}_{n}T\right\}\times \R^{d})}\\
			&\lesssim 
			\|\nabla  \tilde{v}_{n,T}\|_{L_{t,x}^{\frac{2(d+2)}{d}}(\left\{t>\lambda^{2}_{n}T\right\}\times \R^{d})}
			\|  \tilde{v}_{n,T} \|_{L^{{\frac{2(d+2)}{d-2}}}_{t, x}(\left\{t>\lambda^{2}_{n}T\right\}\times \R^{d})}^{\frac{2}{d-1}}\\
			&\quad \times \|  \tilde{v}_{n,T} \|_{L_{t,x}^{\frac{2(d+2)}{d}}(\left\{t>\lambda^{2}_{n}T\right\}\times \R^{d})}^{\frac{2}{d-1}}\\
			&\lesssim
			\|e^{it \L^{n}_{a}}[\chi_{n}w_{n}(T)]\|_{L^{\frac{2(d+2)}{d-2}}_{t,x}((0, \infty)\times \R^{d})}^\frac{2}{d-1}\to 0 \mbox{ (by \eqref{esstential-LimitAx})}
		\end{align*}
		as $n\to\infty$ and $T\to \infty$. Similarly, 
	\begin{align*}
			&\hspace{3ex}\| \nabla (\tilde{v}_{n,T}|^{\frac{4}{d-2}} \tilde{v}_{n,T})\|_{L_{t,x}^{\frac{2(d+2)}{d+4}}
				(\left\{t>\lambda^{2}_{n}T\right\}\times \R^{d})}\\
			&\lesssim \|\nabla  \tilde{v}_{n,T}\|_{L_{t,x}^{\frac{2(d+2)}{d}}(\left\{t>\lambda^{2}_{n}T\right\}\times \R^{d})}\|  \tilde{v}_{n,T} \|^{\frac{4}{d-2}}_{L^{{\frac{2(d+2)}{d-2}}}_{t, x}(\left\{t>\lambda^{2}_{n}T\right\}\times \R^{d})}\\
			&\lesssim \| e^{-it\L^{n}_{a}}(\chi_{n}\omega_{n}(T))  \|^{\frac{4}{d-2}}_{L^{{\frac{2(d+2)}{d-2}}}_{t, x}((0, \infty)\times \R^{d})}\to 0  \mbox{ (by \eqref{esstential-LimitAx})}
\end{align*}
		as $n\to\infty$ and $T\to \infty$.  So \eqref{33estimate} holds in Scenario \uppercase\expandafter{\romannumeral2}.

		
		\underline{\emph{Proof of \eqref{33estimate} in Scenario \uppercase\expandafter{\romannumeral3}.}} Once again, we will treat $e_{n,T}$ in the regions $|t|\leq\lambda_n^2T$ and $|t|>\lambda_n^2 T$ separately. 
		
		Recalling that $w_n$ is a solution to \eqref{NLS-critical}, we find that in the region $|t|\leq\lambda_n^2 T$, we have
		\begin{align}
			e_{n,T}(t,x)&=-\lambda^{-\frac{(d-2)(d+3)}{2(d-1)}}_{n}[\chi^{\frac{d+3}{d-1}}_{n}|w_{n}|^{\frac{4}{d-1}}w_{n}](\lambda^{-2}_{n}t, \lambda^{-1}_{n}(x-x_{n})) 
			\label{e11}
			\\
			&\quad-\lambda^{-\frac{d+2}{2}}_{n}[(\chi_{n}-\chi^{\frac{d+2}{d-2}}_{n})|w_{n}|^{\frac{4}{d-2}}w_{n}](\lambda^{-2}_{n}t, \lambda^{-1}_{n}(x-x_{n}))\label{e22}\\
			&\quad+2\lambda^{-\frac{d+2}{2}}_{n}[\nabla \chi_{n}\cdot \nabla w_{n}](\lambda^{-2}_{n}t, \lambda^{-1}_{n}(x-x_{n}))\label{e33}
			\\
			&\quad+\lambda^{-\frac{d+2}{2}}_{n}[\Delta \chi_{n} w_{n}](\lambda^{-2}_{n}t, \lambda^{-1}_{n}(x-x_{n}))\label{e44}
			\\
			&\quad-\lambda^{-\frac{d-2}{2}}_{n}\tfrac{a}{|x|^{2}}[\chi_{n}w_{n}](\lambda^{-2}_{n}t, \lambda^{-1}_{n}(x-x_{n})).\label{e55}
		\end{align}
		
		In the region $t>\lambda_n^2 T$, we again have
		\begin{equation}\label{e66}
			e_{n,T} = |\tilde v_{n,T}|^{\frac{4}{d-2}} \tilde v_{n,T} - |\tilde v_{n,T}|^{\frac{4}{d-1}} \tilde v_{n,T}.
		\end{equation}
		
		By the almost same argument in \cite{Ardila-Murphy}, We can deduce that \eqref{e22}--\eqref{e55} tend to 0 as $n\to\infty$ and $T\to \infty$.So we only treat $\eqref{11e}$ in details. 
		
In fact,by changing variables and  H\"older's inequality, \eqref{Persistence-Regularity-smalls}, and \eqref{Persistence-Regularity}, we obtain
		\begin{align*}
			\|\nabla\eqref{e11}\|_{L_{t,x}^{\frac{2(d+2)}{d+4}}}
			&\lesssim \lambda_{n}^{\frac{2}{d-1}}\|\chi_{n}\|_{L^{\infty}_{x}}
			\|\nabla w_{n}\|_{L_{t,x}^{\frac{2(d+2)}{d}}}
			\|  w_{n} \|^{\frac{2}{d-1}}_{L_{t,x}^{\frac{2(d+2)}{d}}}\| w_{n}\|^{\frac{2}{d-1}}_{L_{t,x}^{\frac{2(d+2)}{d-2}}}\\
			&\quad+\lambda^{\frac{d+1}{d-1}}_{n}\|\chi_{n}\|^{\frac{4}{d-1}}_{L^{\infty}_{x}}
			\|\nabla \chi_{n}\|_{L^{\infty}_{x}}
			\| w_{n}\|^{\frac{2}{d-1}}_{L_{t,x}^{\frac{2(d+2)}{d-2}}}\| w_{n}\|^{\frac{d+1}{d-1}}_{L_{t,x}^{\frac{2(d+2)}{d}}}\\
			& \lesssim\lambda^{\tfrac{(d+1)(1-\theta)}{d-1}}_{n}\to 0 \quad \text{as $n\to \infty$.}
		\end{align*}

	Finally we treat \eqref{e66}.By \eqref{esstential-LimitAx}, we can handle the estimate of the energy-critical term with exactly the same argument in  Scenario~\uppercase\expandafter{\romannumeral2}.For the energy-subcritical pertubation term, We just use Strichartz estimate and the persistence of $L^2$ regularity \eqref{Persistence-Regularity} to deduce that
		\begin{align*}
			&\|\nabla | \tilde{v}_{n,T}|^{\frac{4}{d-1}} \tilde{v}_{n,T}\|_{L_{t,x}^{\frac{2(d+2)}{d+4}}
				(\left\{t>\lambda^{2}_{n}T\right\}\times \R^{d})}\\
			&\lesssim 
			\|\nabla \tilde{v}_{n,T}\|_{L_{t,x}^{\frac{2(d+2)}{d}}(\left\{t>\lambda^{2}_{n}T\right\}\times \R^{d})}
			\|  \tilde{v}_{n,T} \|^{\frac{2}{d-1}}_{L_{t,x}^{\frac{2(d+2)}{d}}(\left\{t>\lambda^{2}_{n}T\right\}\times \R^{d})}\|  \tilde{v}_{n,T} \|^{\frac{2}{d-1}}_{L_{t,x}^{\frac{2(d+2)}{d-2}}(\left\{t>\lambda^{2}_{n}T\right\}\times \R^{d})}\\
			&\lesssim
			\|  \tilde{v}_{n,T} \|^{\frac{2}{d-1}}_{L_{t,x}^{\frac{2(d+2)}{d}}(\left\{t>\lambda^{2}_{n}T\right\}\times \R^{d})}
			\lesssim \lambda^{\frac{2}{d-1}}_{n}\|w_{n}\|^{\frac{2}{d-1}}_{{L}^{\infty}_{t}{L}^{2}_{x}}\lesssim \lambda^{\frac{2(1-\theta)}{d-1}}_{n}\to 0 \quad \text{as $n\to\infty$.}
		\end{align*}

		Thus \eqref{33estimate} holds in Scenario~\uppercase\expandafter{\romannumeral3}. 
		
		\textbf{Estabilishment of true solutions.} With \eqref{11estimate}--\eqref{33estimate} in place, we may apply the stability lemma (Lemma~\ref{StabilityNLS}) to deduce the existence of a global solution $v_n$ to \eqref{NLS} with $v_n(0)=\phi_n$,
		\[
		\|v_n\|_{L_{t,x}^{\frac{2(d+2)}{d-2}}(\R\times\R^d)}\lesssim 1 \qtq{uniformly in}n,
		\]
		and
		\[
		\limsup_{T\to\infty}\limsup_{n\to\infty}\|v_n(t-\lambda_n^2 t_n)-\tilde v_{n,T}(t)\|_{\dot S_a^1(\R\times\R^d)} =0 .
		\]
		
		\textbf{Approximation by compactly supported functions.}  Once again, One can verify \eqref{Comspt11} by almost the same argument in \cite{Ardila-Murphy}. Thus we omit the proof. \end{proof}

	
	\subsection{Existence of minimal blow-up solutions}\label{MinimalSolutions}
	For each $E>0$ and $C>0$,  we define
	\begin{align*}
		&L_C(E):=\sup\bigg\{\|  u \|_{L^{\frac{2(d+2)}{d-2}}_{t,x}(I\times\R^{d})}:\mbox{\emph{u} solves \eqref{NLS} on } I\times \R^d \mbox{ and }\\&M(u)\leqslant C,\quad E_a(u(t))\leq E,\quad\| u(t_0)\|_{\dot{H}_a^1(\Bbb{R}^d)}< \|W_a\|_{\dot{H}_a^1(\Bbb{R}^d)} \mbox{ for some }t_0\in I\bigg\}.
	\end{align*}
	
	By the local theory, we suppose that Theorem \ref{Th1} fails, then there exists $C_*>0$ such that $E_{C_*}< E_a(W_a)$. Using Lemma~\ref{StabilityNLS}, this implies that $L_{C_*}(E_{C_*})=\infty$. Thus there exists a sequence of solutions $u_{n}$ such that  $E_a(u_{n})\rightarrow E_{C_*}$ and $\|  u_{n} \|_{L^{\frac{2(d+2)}{d-2}}_{t,x}(\R\times\R^{d})}\rightarrow \infty$ as $n\rightarrow\infty$. We will prove the existence of a solution $u_{C_0}\in H^{1}(\R^{d})$ such that $E_a(u_{C_*})=E_{C_*},\text{sup}_{t\in I}\| u_{C_*}\|_{\dot{H}_a^1(\Bbb{R}^d)}<\|W_a\|_{\dot{H}_a^1},0\in I=(T_{min},T_{max})$ with
	\begin{equation}\label{minimalblowsolution}
		\|  u_{C_*} \|_{L^{\frac{2(d+2)}{d-2}}_{t,x}([0,T_{max})\times\R^{d})}=\|  u_{C_*} \|_{L^{\frac{2(d+2)}{d-2}}_{t,x}((T_{min},0]\times\R^{d})}=\infty,
	\end{equation}
	and such that
	\begin{equation}\label{orbit-is-precompact}
		\left\{u_{C_*}(t):t\in I\right\}\qtq{is precompact in }\dot{H}_a^{1}(\R^{d}).
	\end{equation}
	
	\begin{theorem}[Existence of minimal blow-up solutions]\label{CompacSolution} Suppose Theorem~\ref{sub-threshold} fails.  Then there exists a solution   $u_{C_*}$  to \eqref{NLS}  with maximal lifespan $I=(T_{min},T_{max})$ such that $0\in I$, 
		$M(u_{C_*})\leq C_*,E_a(u_{C_*})=E_{C_*},\text{sup}_{t\in I}\| u_{C_*}\|_{\dot{H}_a^1}<\|W_a\|_{\dot{H}_a^1}$.Morever,\eqref{minimalblowsolution} and \eqref{orbit-is-precompact} hold.
	\end{theorem}
	
	Arguing as in \cite[Theorem~9.6]{KillipOhPoVi2017}, to establish Theorem~\ref{CompacSolution}, it will suffice to establish the following Palais--Smale condition. 
	
	\begin{proposition}\label{PScondition}
		Let $\left\{u_{n}\right\}_{n\in \N}\subset H_a^{1}(\R^{d})$ be a sequence of solutions to \eqref{NLS} such that
		$M(u_n)\leq C_*, \lim\limits_{n\rightarrow\infty}E_a(u_{n})=E_{C_*}$, and suppose $t_{n}\in \R$ satisfy
		\begin{equation}\label{Blowp2}
			\lim_{n\rightarrow\infty}\|  u_{n} \|_{L^{\frac{2(d+2)}{d-2}}_{t,x}([t_{n},\infty)\times\R^{d})}=\|  u_{n} \|_{L^{\frac{2(d+2)}{d-2}}_{t,x}((-\infty,t_{n}]\times\R^{d})}=\infty.
		\end{equation}
		Then we have that $\left\{u_{n}\right\}_{n\in \N}$ converges along a subsequence in  $\dot{H}_{x}^{1}(\R^{d})$.
	\end{proposition}
	\begin{proof}
		By time-translation invariance, we may assume that $t_{n}\equiv 0$. Using Lemma \ref{GlobalS} and writing $u_{n,0}=u_n(0)$, we have 
		\[
		\|u_{n,0}\|^{2}_{\dot{H}_a^{1}}  < \|W_a\|_{\dot{H}_a^1}^2
		\]
		and
		\[
		\|u_{n,0}\|^{2}_{L^2}  \lesssim M(u_{n})  \lesssim C_*.
		\]
		Applying Theorem~\ref{LinearProfi}, we may write
		\begin{equation}\label{profile dcp}
			u_{n}(0)=\sum^{J}_{j=1}\phi^{j}_{n}+W^{J}_{n}
		\end{equation}
		for each $J\leq J^{\ast}$, with the various sequences satisfying \eqref{Reminder}--\eqref{EnergyEx}.  We have $M(u_{n})\leq C_*$, $E_{a}(u_{n})\rightarrow E_{C_*}$ and 
		By \eqref{MassEx} and \eqref{EnergyEx}, we also have 
		\begin{align}\label{profilemassdcp}
			&\limsup_{n\rightarrow\infty}\sum^{J}_{j=1}M(\phi_{n}^{j})+M(W^{J}_{n})\leq C_*,\\\label{Pv22}
			&\limsup_{n\rightarrow\infty}\sum^{J}_{j=1}E_{a}(\phi_{n}^{j})+E_{a}(W^{J}_{n})\leq E_{C_*},
		\end{align}
		for each finite $J\leq J^{\ast}$, with all energies in \eqref{Pv22} nonnegative. Moreover, by sharp Sobolev embedding and the nontriviality of $\phi_{n}^{j}$ we have that $\liminf\limits_{n\to\infty}E_{a}(\phi_{n}^{j})>0$.  
		
		We shall prove that there at most exists one nonzero $\phi_n^j$.
		
		\textbf{Scenario \uppercase\expandafter{\romannumeral1}.}  If
		\begin{equation}\label{Scena1}
			\sup_{j}\limsup_{n\rightarrow\infty}E_{a}(\phi_{n}^{j})=E_{C_*}.
		\end{equation}
		By \eqref{Pv22}, positivity of energy yields $J^{\ast}=1$.  In this case, we have that 
		\begin{equation}\label{Csc}
			u_{n}(0)=\phi^{1}_{n}+W_{n}^{1}, \quad\text{with}\quad \lim_{n\rightarrow\infty}\|W_{n}^{1}\|^{2}_{\dot{H}_{a}^{1}}=0.
		\end{equation}
		
		Now suppose that $\frac{|x^{1}_{n}|}{\lambda^{1}_{n}}\to \infty$. Then Proposition~\ref{embedding-nonlinear} yields a global solution $v_{n}$ with $v_{n}(0)=\phi^{1}_{n}$ such that
		\begin{equation*}
			\|v_{n}\|_{L^{\frac{2(d+2)}{d-2}}_{t,x}(\R\times\R^{d})}\lesssim 1,\hspace{2ex}\|v_{n}\|_{L^{\frac{2(d+2)}{d-1}}_{t,x}(\R\times\R^{d})}\lesssim1.
		\end{equation*}
		As $W_{n}^{1}=u_{n}(0)-v_{n}(0)$, it follows that $\lim\limits_{n\to\infty}\|u_{n}(0)-v_{n}(0)\|_{\dot{H}_{a}^{1}}=0$. Thus, Lemma~\ref{StabilityNLS} yields that for $n$ large $u_{n}$ is a global solution with finite scattering norm, contradicting \eqref{Blowp2}. It follows that $x^{1}_{n}\equiv 0$.
		
		Next, suppose that $\lambda^{1}_{n}\to 0$ as $n\to \infty$.  In this case, Proposition~\ref{embedding-nonlinear} yields a global solution $v_{n}$
		with  $v_{n}(0)=\phi^{1}_{n}$ and $\|v_{n}\|_{L^{\frac{2(d+2)}{d-2}}_{t,x}(\R\times\R^{d})}\lesssim 1$. Then Lemma~\ref{StabilityNLS} implies that $\|u_{n}\|_{L^{\frac{2(d+2)}{d-2}}_{t,x}(\R\times\R^{d})},\|v_{n}\|_{L^{\frac{2(d+2)}{d-1}}_{t,x}(\R\times\R^{d})}\lesssim 1$ for $n$ large enough, again contradicting \eqref{Blowp2}. It follows that $\lambda_n^1\equiv 1$.
		
		Finally, suppose that $t^{1}_{n}\to \infty$ as $n\to \infty$.  By Sobolev embedding, Strichartz estimate and \eqref{Csc}, we obtain that
		\begin{equation}\label{Eq11}
			\begin{split}
				\| & e^{-it\L_{a}}u_{n}(0)  \|_{L^{\frac{2(d+2)}{d-2}}_{t,x}([0,\infty)\times \R^{d})} \\
				&\leq \|e^{-it\L_{a}} \phi^{1}_{n}  \|_{L^{\frac{2(d+2)}{d-2}}_{t,x}([0,\infty)\times \R^{d})}+
				\|e^{-it\L_{a}} W_{n}^{1} \|_{L^{\frac{2(d+2)}{d-2}}_{t,x}([0,\infty)\times \R^{d})}\\
				&\lesssim
				\|e^{-it\L_{a}} \phi^{1} \|_{L^{\frac{2(d+2)}{d-2}}_{t,x}([t^{1}_{n},\infty)\times \R^{d})}+\|W_{n}^{1}\|_{\dot{H}_{x}^{1}}\rightarrow 0,
			\end{split}
		\end{equation}
		as $n\rightarrow\infty$. Writing $\tilde{u}_{n}=e^{-it\L_{a}}u_{n}(0)$ and $e_{n}=|\tilde{u}_{n}|^{\frac{4}{d-2}}\tilde{u}_{n}-|\tilde{u}_{n}|^{\frac{4}{d-1}}\tilde{u}_{n}$, By \eqref{Eq11}, H\"older's inequality, and 
		Strichartz estimate,we obtain
		\[
		\|\nabla e_{n}  \|_{N(\R)}\to 0\quad \text{as $n\to \infty$}.
		\]
		Then by stability lemma (Lemma~\ref{StabilityNLS}), $\|u_n\|_{L_{t,x}^{\frac{2(d+2)}{d-2}}}<\infty$ uniformly with $n$, which contradicts  \eqref{Blowp2}. An analogous argument handles the case $t^{1}_{n}\rightarrow -\infty$ as $n\rightarrow\infty$.
		
		Thus, in Scenario \uppercase\expandafter{\romannumeral1}, we obtain that $x^{1}_{n}\equiv 0$, $t^{1}_{n}\equiv 0$ and $\lambda^{1}_{n}\equiv 1$.  This yields the desired conclusion of Proposition~\ref{PScondition}, and hence it remains to show that the only remaining scenario results in a contradiction. 
		
		\textbf{Scenario \uppercase\expandafter{\romannumeral2}.} If \eqref{Scena1} fails for all $j$, then there exists $\delta>0$ such that
		\begin{equation}\label{Scena2}
			\sup_{j}\limsup_{n\rightarrow\infty}M(\phi_{n}^{j})\leq C_*\quad
			\mbox {and} \quad \sup_{j}\limsup_{n\rightarrow\infty}E_{a}(\phi_{n}^{j})\leq E_{C_*}-\delta.
		\end{equation}
		
		We then define nonlinear profiles $\psi_{n}^{j}$ associated to each $\phi_{n}^{j}$ as follows:
		\begin{itemize}
			\item If $\frac{|x^{j}_{n}|}{\lambda^{j}_{n}}\to \infty$ for some $j$, then we are in position to apply Proposition~\ref{embedding-nonlinear}, and hence we have a global solution $\psi_{n}^{j}$ of \eqref{NLS} with data $\psi_{n}^{j}(0)=\phi_{n}^{j}$.

			\item If $x^{j}_{n}\equiv 0$ and $\lambda^{j}_{n}\rightarrow 0$, we define $\psi_{n}^{j}$ to be the global solution of \eqref{NLS} with the initial data $\psi_{n}^{j}(0)=\phi_{n}^{j}$ guaranteed by Proposition~\ref{embedding-nonlinear}.
			
			\item If $x^{j}_{n}\equiv 0$, $\lambda^{j}_{n}\equiv 1$ and $t^{n}_{j}\equiv 0$, we take  $\psi^{j}$ to be the global solution of 
			\eqref{NLS} with the initial data $\psi^{j}(0)=\phi^{j}$.
			
			\item If $x^{j}_{n}\equiv 0$, $\lambda^{j}_{n}\equiv 1$ and $t^{n}_{j}\rightarrow\pm\infty$, we take $\psi^{j}$  to be the global solution  of \eqref{NLS} that scatters to $e^{-it\L_{a}}\phi^{j}$ in $H^{1}_{x}(\R^{d})$  as $t\rightarrow\pm\infty$.  In either case, we define the global solution to \eqref{NLS},
			\[\psi^{j}_{n}(t,x):=\psi^{j}(t+t^{j}_{n}, x).\]
		\end{itemize}
		
		By construction, we have that for each $j$, 
		\begin{equation}\label{Aproxi11}
			\|\psi_{n}^{j}(0)- \phi_{n}^{j}\|_{H^{1}_{a}}\rightarrow0,\quad \text{as $n\rightarrow\infty$}.
		\end{equation}
		Moreover, notice that by \eqref{Scena2} and the definition of $E_{C_*}$, we may obtain 
		\begin{equation}\label{BoundProfile}
			\| \psi^{j}_{n}  \|_{L^{\frac{2(d+2)}{d-2}}_{t,x}}\lesssim_{\delta} 1, \quad \text{for $n$ large and $1\leq j\leq J$}.
		\end{equation}
		 
		\noindent\textbf{Case 1}: $d=4,5$.
		By \eqref{BoundProfile} and 
		Lemma \ref{PRegularity},  we have the following estimate:
		\begin{align}
			&\| \psi^{j}_{n}  \|_{L^{\frac{2(d+2)}{d-2}}_{t,x}}\lesssim_{\delta} [E_{a}(\psi^{j}_{n})]^{\frac{1}{2}},\quad
			\|\psi^{j}_{n}\|_{L^{\frac{2(d+2)}{d-2}}_{t}\dot{H}_{a}^{1,\frac{2d(d+2)}{d^2+4}}}\lesssim_{\delta} [E_{a}(\psi^{j}_{n})]^{\frac{1}{2}},\label{ImporBound}\\
			&\|\nabla\psi^{j}_{n}\|_{L_{t,x}^{\frac{2(d+2)}{d}}}\sim\|\sqrt{\mathcal{L}_a}\psi^{j}_{n}\|_{L_{t,x}^{\frac{2(d+2)}{d}}}\lesssim_{\delta} [E_a(\psi^{j}_{n})]^{\frac{1}{2}},
			\|\psi^{j}_{n}\|_{L_{t,x}^{\frac{2(d+2)}{d}}}\lesssim_{\delta} [M(\psi^{j}_{n})]^{\frac{1}{2}}.	\label{ImporBound22}
		\end{align}
		We define the approximate solutions
		\[
		u^{J}_{n}(t):=\sum^{J}_{j=1}\psi^{j}_{n}(t)+e^{-it\L_{a}}W^{J}_{n},
		\] 
		with the goal of applying Lemma~\ref{StabilityNLS} to contradict \eqref{Blowp2}. In particular, we define the errors $e_n^J$ via 
		\[
		(i\partial_{t}-\L_{a}) {u}^{J}_{n}=-|{u}^{J}_{n}|^{\frac{4}{d-2}}{u}^{J}_{n}+|{u}^{J}_{n}|^{\frac{4}{d-1}}{u}^{J}_{n}+e^{J}_{n}.
		\]
		
		From \eqref{Aproxi11} we see that
		\begin{equation}\label{DefiuJ}
			\begin{split}
				\lim_{n\to \infty}\|u_{n}^{J}(0)- u_{n}(0)\|_{\dot{H}^{1}_{a}}=0, \quad \text{for any $J$}.
			\end{split}
		\end{equation}
		As before, it's sufficient to establish the following estimates:
		\begin{align}\label{Bound11}
			&\sup_{J}\limsup_{n\rightarrow\infty}\| {u}^{J}_{n}  \|_{L_{t}^{\infty}H^{1}_{a}(\R\times \R^{d})}\lesssim_{\delta} 1,
			\\
			&\sup_{J}\limsup_{n\rightarrow\infty}\big[ \| {u}^{J}_{n}  \|_{L^{\frac{2(d+2)}{d-2}}_{t,x}}
			+\|u_{n}^{J}\|_{L_{t,x}^{\frac{2(d+2)}{d}}}+ \|u_{n}^{J}\|_{L^{\frac{2(d+2)}{d-2}}_{t}\dot{H}_{a}^{1,\frac{2d(d+2)}{d^2+4}}}\notag\\
			&\hspace{14ex}+\|u_{n}^{J}\|_{L^{\frac{2(d+2)}{d}}_{t}\dot{H}_{a}^{1,\frac{2(d+2)}{d}}}\big]
			\lesssim_{\delta} 1,\label{Bound22}
			\\
			\label{Bound33}
			&\lim_{J\to J^{\ast}}\limsup_{n\rightarrow\infty}\|\nabla e^{J}_{n}  \|_{N(\R)}=0,
		\end{align}
		where here and below all space-time norms are taken over $\R\times\R^d$. Indeed, using  \eqref{DefiuJ}, \eqref{Bound11}, \eqref{Bound22}, and \eqref{Bound33}, Lemma~\ref{StabilityNLS} implies that  $\| u_{n}  \|_{L^{\frac{2(d+2)}{d-2}}_{t,x}}\lesssim_{\epsilon,\delta}1$ for  $n$ large, contradicting \eqref{Blowp2}. 
		
		We therefore turn to establishing the estimates \eqref{Bound11}-\eqref{Bound33}. We will use the following lemma.  
		
		\begin{lemma}[Asymptotic decoupling]\label{AsymptoticDec}
			If $j\neq k$ we have
			\[
			\lim_{n\to \infty}[
			\|\psi^{j}_{n} \psi^{k}_{n}\|_{L^{\frac{d+2}{d-2}}_{t, x}}
			+\| \psi^{j}_{n} \nabla\psi^{k}_{n} \|_{L^{\frac{d+2}{d-1}}_{t, x}}
			+\|\nabla \psi^{j}_{n} \nabla\psi^{k}_{n} \|_{L^{\frac{d+2}{d}}_{t, x}}
			+\| \psi^{j}_{n}\psi^{k}_{n} \|_{L^{\frac{d+2}{d}}_{t, x}}
			]=0.
			\]
		\end{lemma}
		
		\begin{proof} The proof is based on Proposition \ref{embedding-nonlinear}, \eqref{ImporBound} and \eqref{ImporBound22}, which is similar to the proof in  \cite{YANG2020124006}. Hence, we omit the details. \end{proof}
		
		As \eqref{Bound11} readily follows from Strichartz \eqref{DefiuJ}, \eqref{Bound22}, and \eqref{Bound33}, it will suffice to establish \eqref{Bound22} and \eqref{Bound33}. 
		
		\underline{\emph{Proof of \eqref{Bound22}}} We only prove the estimate of the $L_{t,x}^{\frac{2(d+2)}{d-2}}$-norm, and the estimates of other terms can be checked by the same argument.  In fact, by \eqref{ImporBound}, equivalence of Sobolev spaces and Strichartz estimate, we deduce that 
			\[\begin{split}
				\| u^{J}_{n}  \|^{2}_{L^{\frac{2(d+2)}{d-2}}_{t,x}}&\lesssim \sum^{J}_{j=1}\| \psi^{j}_{n}  \|^{2}_{L^{\frac{2(d+2)}{d-2}}_{x}}+
				\sum_{j\neq k}\| \psi^{j}_{n} \psi^{k}_{n}  \|_{L^{\frac{(d+2)}{d-2}}_{t,x}}+\| e^{-it\L_{a}}W^{J}_{n}  \|^{2}_{L^{\frac{2(d+2)}{d-2}}_{t,x}}\\
				&\lesssim \sum^{J}_{j=1}E_{a}( \psi^{j}_{n} )+\sum^{J}_{j\neq k}o(1)+E_{a}(W^{J}_{n})
				\lesssim 1 +o(1)\cdot J^{2}
			\end{split}
			\]
			as $n\to\infty$. Thus \eqref{Bound22} holds.

			\underline{\emph{Proof of \eqref{Bound33}}} Since $\psi^{j}_{n}$ is a solution of \eqref{NLS}, we can write
					\begin{align}\label{Fdefi11}
						e^{J}_{n}&= \sum^{J}_{j=1}F( \psi^{j}_{n} )-F\big(\sum^{J}_{j=1}\psi^{j}_{n}\big)\\\label{Fdefi22}
						&+\sum^{J}_{j=1}F(u^{J}_{n}-e^{-it\L_{a}}W^{J}_{n})-F(u^{J}_{n}),
					\end{align}
					where $F(z)=F_{2}(z)-F_{1}(z)$ with $F_{1}(z):=|z|^{\frac{4}{d-2}}z$ and $F_{2}(z):=|z|^{\frac{4}{d-1}}z$.By H\"older's inequality we have
					\begin{align}\label{F1E}
						\|\nabla\big[\sum^{J}_{j=1}F_{1}(\psi^{j}_{n}) -F_{1}(\sum^{J}_{j=1}\psi^{j}_{n})\big]\|_{L^{\frac{2(d+2)}{d+4}}_{t,x}}
						&\lesssim \sum_{j\neq k}\| \psi^{j}_{n}  \|^{\frac{6-d}{d-2}}_{L^{\frac{2(d+2)}{d-2}}_{t,x}}\|  \psi^{j}_{n} \nabla\psi^{k}_{n}  \|_{L^{\frac{d+2}{d-1}}_{t,x}},\\
						\label{F2E}
						\|\nabla\big[\sum^{J}_{j=1}F_{2}(\psi^{j}_{n}) -F_{2}(\sum^{J}_{j=1}\psi^{j}_{n})\big]\|_{L^{\frac{2(d+2)}{d+4}}_{t,x}}
						&\lesssim \sum_{j\neq k} \| \psi^{j}_{n}  \|^{\frac{5-d}{d-2}}_{L^{\frac{2(d+2)}{d-2}}_{t,x}}\|  \psi^{j}_{n} \nabla\psi^{k}_{n}  \|^{\frac{1}{2}}_{L^{\frac{d+2}{d}}_{t,x}} \|  \psi^{j}_{n} \nabla\psi^{k}_{n}  \|^{\frac{1}{2}}_{L^{\frac{d+2}{d-1}}_{t,x}}.
					\end{align}
					Thus, by orthogonality, \eqref{ImporBound}, \eqref{ImporBound22}, \eqref{F1E} and  \eqref{F2E}  we get
					\begin{equation}\label{FirstF}
						\lim_{J\to J^{\ast}}\limsup_{n\rightarrow\infty}\| \nabla \eqref{Fdefi11}\|_{N(\R)}=0.
					\end{equation}
					
					We next estimate \eqref{Fdefi22}. First, by interpolation we get
					\begin{align*}
						&\|\nabla [F_{1}(u^{J}_{n}-e^{-it\L_{a}}W^{J}_{n})-F_{1}(u^{J}_{n})]\|_{L^{\frac{2(d+2)}{d+4}}_{t,x}}\\
						\lesssim& 
						\|  e^{-it\L_{a}}W^{J}_{n} \|^{\frac{4}{d-2}}_{L^{\frac{2(d+2)}{d-2}}_{t, x}}
						\| \nabla  e^{-it\L_{a}}W^{J}_{n}  \|_{L^{\frac{2(d+2)}{d}}_{t,x}}
						+	\|  e^{-it\L_{a}}W^{J}_{n} \|^{\frac{4}{d-2}}_{L^{\frac{2(d+2)}{d-2}}_{t, x}}
						\| \nabla u^{J}_{n}  \|_{L^{\frac{2(d+2)}{d}}_{t,x}}\\
						&+	\|  e^{-it\L_{a}}W^{J}_{n} \|_{L^{\frac{2(d+2)}{d-2}}_{t, x}}\| u^{J}_{n} \|^{\frac{6-d}{d-2}}_{L^{\frac{2(d+2)}{d-2}}_{t, x}}
						\| \nabla u^{J}_{n}  \|_{L^{\frac{2(d+2)}{d}}_{t,x}}+
						\| u^{J}_{n} \|^{\frac{6-d}{d-2}}_{L^{\frac{2(d+2)}{d-2}}_{t, x}}\|u^{J}_{n} \nabla  e^{-it\L_{a}}W^{J}_{n}  \|_{L^{\frac{d+2}{d-1}}_{t,x}}.
					\end{align*}
					Combining \eqref{Reminder}, \eqref{Pv22} and \eqref{ImporBound} we see that
					\[\begin{split}
						\lim_{J\to J^{\ast}}\limsup_{n\rightarrow\infty}\|\nabla [F_{1}(u^{J}_{n}-e^{-it\L_{a}}W^{J}_{n})-F_{1}(u^{J}_{n})]\|_{N(\R)}
						\\
						\lesssim 
						\lim_{J\to J^{\ast}}\limsup_{n\rightarrow\infty}\|u^{J}_{n} \nabla  e^{-it\L_{a}}W^{J}_{n}  \|_{L^{\frac{d+2}{d-1}}_{t,x}}.
					\end{split}
					\]
					Similarly, 
					\begin{align*}
						&\|\nabla [F_{2}(u^{J}_{n}-e^{-it\L_{a}}W^{J}_{n})-F_{2}(u^{J}_{n})]\|_{L^{\frac{2(d+2)}{d+4}}_{t,x}}\\
						\lesssim &
						\|  e^{-it\L_{a}}W^{J}_{n} \|^{\frac{5-d}{d-2}}_{L^{\frac{2(d+2)}{d-2}}_{t, x}}
						\|   e^{-it\L_{a}}W^{J}_{n}  \|^{\frac{1}{2}}_{L^{\frac{2(d+2)}{d-2}}_{t,x}}
						\|   e^{-it\L_{a}}W^{J}_{n}  \|^{\frac{1}{2}}_{L^{\frac{2(d+2)}{d}}_{t,x}}
						\| \nabla  e^{-it\L_{a}}W^{J}_{n}  \|_{L^{\frac{2(d+2)}{d}}_{t,x}}\\
						&\quad +	
						\|  e^{-it\L_{a}}W^{J}_{n} \|^{\frac{5-d}{d-2}}_{L^{\frac{2(d+2)}{d-2}}_{t, x}}
						\|   e^{-it\L_{a}}W^{J}_{n}  \|^{\frac{1}{2}}_{L^{\frac{2(d+2)}{d-2}}_{t,x}}
						\|   e^{-it\L_{a}}W^{J}_{n}  \|^{\frac{1}{2}}_{L^{\frac{2(d+2)}{d}}_{t,x}}
						\| \nabla  u^{J}_{n}  \|_{L^{\frac{2(d+2)}{d}}_{t,x}}\\
						&\quad+
						\|  u^{J}_{n} \|^{\frac{5-d}{d-2}}_{L^{\frac{2(d+2)}{d-2}}_{t, x}}
						\|   e^{-it\L_{a}}W^{J}_{n}  \|^{\frac{1}{2}}_{L^{\frac{2(d+2)}{d-2}}_{t,x}}
						\|   e^{-it\L_{a}}W^{J}_{n}  \|^{\frac{1}{2}}_{L^{\frac{2(d+2)}{d}}_{t,x}}
						\| \nabla  u^{J}_{n}  \|_{L^{\frac{2(d+2)}{d}}_{t,x}}\\
						& \quad +\|  u^{J}_{n} \|^{\frac{5-d}{d-2}}_{L^{\frac{2(d+2)}{d-2}}_{t, x}}
						\|  u^{J}_{n} \nabla  e^{-it\L_{a}}W^{J}_{n}   \|^{\frac{1}{2}}_{L^{\frac{d+2}{d}}_{t,x}} \|  u^{J}_{n} \nabla e^{-it\L_{a}}W^{J}_{n} \|^{\frac{1}{2}}_{L^{\frac{d+2}{d-1}}_{t,x}}.
					\end{align*}
					As Strichartz together with \eqref{Reminder}, \eqref{Pv22} and \eqref{ImporBound} implies
					\begin{equation}
						\begin{split}
							\lim_{J\to J^{\ast}}\limsup_{n\rightarrow\infty}\|\nabla [F_{2}(u^{J}_{n}-e^{-it\L_{a}}W^{J}_{n})-F_{2}(u^{J}_{n})]\|_{N^{1}(\R)}\\
							\lesssim 
							\lim_{J\to J^{\ast}}\limsup_{n\rightarrow\infty}\|u^{J}_{n} \nabla  e^{-it\L_{a}}W^{J}_{n}  \|_{L_{t,x}^{\frac{d+2}{d-1}}},
						\end{split}
					\end{equation}

					it remains to show
					\begin{equation}\label{FinalLimit}
						\lim_{J\to J^{\ast}}\limsup_{n\rightarrow\infty}\|u^{J}_{n} \nabla  e^{-it\L_{a}}W^{J}_{n}  \|_{L^{\frac{d+2}{d-1}}_{t,x}}
						=0.
					\end{equation}
					
					Applying H\"older we deduce
					\begin{equation}
						\begin{split}
							\|u^{J}_{n} \nabla  e^{-it\L_{a}}W^{J}_{n}  \|_{L^{\frac{d+2}{d-1}}_{t,x}}\leq
							\|\big(\sum^{J}_{j=1}\psi^{j}_{n}\big) \nabla  e^{-it\L_{a}}W^{J}_{n}  \|_{L^{\frac{d+2}{d-1}}_{t,x}}\\
							+\|  e^{-it\L_{a}}W^{J}_{n} \|_{L^{\frac{2(d+2)}{d-2}}_{t, x}}
							\| \nabla e^{-it\L_{a}}W^{J}_{n} \|_{L^{\frac{2(d+2)}{d}}_{t, x}}.
						\end{split}
					\end{equation}
					Thus, using   Strichartz inequality, \eqref{Reminder} and \eqref{Pv22} we see that
					$$
					\lim_{J\to J^{\ast}}\limsup_{n\rightarrow\infty}
					\|u^{J}_{n} \nabla  e^{-it\L_{a}}W^{J}_{n}  \|_{L^{\frac{d+2}{d-1}}_{t, x}}
					\leq 
					\lim_{J\to J^{\ast}}\limsup_{n\rightarrow\infty}
					\|\big(\sum^{J}_{j=1}\psi^{j}_{n}\big) \nabla  e^{-it\L_{a}}W^{J}_{n}  \|_{L^{\frac{d+2}{d-1}}_{t, x}}.
					$$
					On the other hand, it follows from \eqref{ImporBound} that 
					\begin{equation}
						\begin{split}
							\|\sum^{J}_{j=J'}\psi^{j}_{n}  \|^{2}_{L^{\frac{2(d+2)}{d-2}}_{t, x}}&\lesssim \sum^{J}_{j=J'}\| \psi^{j}_{n}  \|^{2}_{L^{\frac{2(d+2)}{d-2}}_{t, x}}+
							\sum_{j\neq k}\| \psi^{j}_{n} \psi^{k}_{n}  \|_{L^{\frac{d+2}{d-2}}_{t, x}}
							\lesssim \sum^{J}_{j=J'}E_{a}( \psi^{j}_{n} )+\sum^{J}_{j\neq k}o(1)\notag
						\end{split}
					\end{equation}
					as $n\to \infty$. Thus, applying \eqref{Pv22}, H\"older's inequality and Strichartz estimate, we deduce that there 
					exists $J'=J(\eta)$ such that
					\begin{equation}
						\begin{split}
							\limsup_{n\rightarrow\infty}
							\|\big(\sum^{J}_{j=J'}\psi^{j}_{n}\big) \nabla  e^{-it\L_{a}}W^{J}_{n}  \|_{L^{\frac{d+2}{d-1}}_{t,x}}
							&\lesssim 
							\limsup_{n\rightarrow\infty}\big\|\sum^{J}_{j=J'}\psi^{j}_{n} \big \|_{L^{\frac{2(d+2)}{d-2}}_{t, x}}
							\big\| \nabla e^{-it\L_{a}}W^{J}_{n} \big\|_{L^{\frac{2(d+2)}{d}}_{t, x}}\\
							&\lesssim \eta \quad \text{uniformly in $J\geq J'$}\notag
						\end{split}
					\end{equation}
					for any $\eta>0$. Especially, in order to establish \eqref{FinalLimit}, it suffices to show
					\begin{equation}\label{Newfinallimit}
						\limsup_{n\rightarrow\infty}
						\|\psi^{j}_{n}\nabla  e^{-it\L_{a}}W^{J}_{n}  \|_{L^{\frac{d+2}{d-1}}_{t,x}}=0
						\quad \text{for all $1\leq j \leq J'$.}
					\end{equation}
					To this end, we just observe that for any $\delta>0$ there exists $\varphi^{j}_{\delta}\in C^{\infty}_{\delta}$
					with support in $[-T,T]\times\left\{|x|\leq R\right\}$ such that (see \eqref{Comspt11})
					\[
					\|  \psi^{j}_{n}-(\lambda^{j}_{n})^{-\frac12} \varphi^{j}_{\delta}\big( \tfrac{t}{(\lambda_n^j)^2}+t^{j}_{n},
					\tfrac{x-x^{j}_{n}}{\lambda^{j}_{n}}\big)   \|_{L^{\frac{2(d+2)}{d-2}}_{t,x}}\leq \delta.
					\]
					Writing 
					\[
					\tilde{W}^{J}_{n}(t,x):=(\lambda^{j}_{n})^{\frac{d-2}{2}}[e^{-it\L_{a}}W^{J}_{n}]( (\lambda^{j}_{n})^{2}(t-t^{j}_{n}),
					\lambda^{j}_{n}x+ x^{j}_{n})
					\]
					and applying  the equivalence of Sobolev norms and H\"older's inequality, we  obtain
					\begin{align*}
						&\|\psi^{j}_{n}\nabla  e^{-it\L_{a}}W^{J}_{n}  \|_{L^{\frac{d+2}{d-1}}_{t,x}} \\
						\lesssim& \delta
						\| \nabla e^{-it\L_{a}}W^{J}_{n} \|_{L^{\frac{2(d+2)}{d}}_{t,x}}
						+	\|\varphi^{j}_{\delta}\|_{L^{\frac{2(d+1)(d+2)}{d^2-2d-2}}_{t,x}}
						\| \nabla \tilde{W}^{J}_{n} \|_{L^{\frac{2(d+1)}{d}}_{t,x}([-T,T]\times\left\{|x|\leq R\right\})}\\
						\lesssim&\delta+C(\delta,T,R)\|\nabla \tilde{W}^{J}_{n}\|^{\frac{d+2}{2(d+1)}}_{L^{\frac{2(d+2)}{d}}_{t,x}}\|\nabla \tilde{W}^{J}_{n}\|^{\frac{d}{2(d+1)}}_{L^{2}_{t,x}}
						.
					\end{align*}
					Thus \eqref{Newfinallimit} finally follows from Lemma \ref{Kato} and \eqref{Reminder}, which completes the proof of  \eqref{Bound33} in $d=4,5$.
			
				As described above, with \eqref{Bound22} and \eqref{Bound33}, we show that Scenario \uppercase\expandafter{\romannumeral2} cannot occur and hence finish the proof of Proposition \ref{PScondition} in dimensions $d=4,5$.
				
				\noindent\textbf{Case 2}: $d=3$.
				We note that the following Sobolev norm equivalence 
				\begin{align*}
					\Vert\nabla \psi_n^j\Vert_{L_{x}^\frac{2(d+2)}{d}}\sim \Vert\sqrt{\mathcal{L}_a} \psi_n^j\Vert_{L_{x}^\frac{2(d+2)}{d}},
				\end{align*}
				fails,
				so we cannot simply carry out the proof above to deal with the case of $d=3$. Hence we need to choose the admissible pairs carefully.  We refer to Ardila-Murphy\cite{Ardila-Murphy} for the detailed proof of $d=3$.\footnote{In Ardila-Murphy\cite{Ardila-Murphy}, they deal with the defocusig-focusing NLS:
					$i\partial_tu+\mathcal{L}_a=-|u|^4u+|u|^2u$. But the proof  can be paralleled to our case without any additional modification.}
			\end{proof}
			\section{Preclusion of compact solutions: The proof of Theorem \ref{sub-threshold}(i)}\label{sec:preclude-sub}
			
			In this section, we use the localized virial argument to preclude the possibility of a solution $u_c$ as in Theorem~\ref{CompacSolution}, thus completing the proof of Theorem~\ref{sub-threshold}. Without loss of generality, we only consider  $u_c(t)$ on $t\in[0,T_{max})$.
			
			We begin with the finite-time blow-up case. \begin{proposition}
				There are no solutions to $\eqref{NLS}$ of the form given in Theorem \ref{CompacSolution} with $T_{max}<\infty$.
			\end{proposition}
			\begin{proof}
				By contradiction, we suppose that the solution of $\eqref{NLS}$ exists in the sense of Proposition \ref{CompacSolution}. Then, we can choose $\{t_n\}$ such that $t_n\rightarrow T_{max}$. Using the compactness property, $u(t_n)\rightarrow h$ in $\dot{H}^1$. Note that $\Vert u(t_n)\Vert_{L^2(\Bbb{R}^d)}\lesssim C$. Moreover, by the mass conservation and the uniqueness of the weak limit, we have $\Vert u(t_n)\Vert_{L^2}^2\leqslant C$. From the local theory, $u$ can be extended to the neighbourhood of $T_{max}$, which is a contradiction.
			\end{proof}
			Next, we preclude the infinite time blow-up case. 
			\begin{proposition}\label{BuIo} Suppose $u_{C_*}$ is a solution as in Theorem \ref{CompacSolution}. Then for every $\epsilon>0$ there exists $R=R(\epsilon)>1$ such that
				\begin{equation}\label{Uniform}
					\sup_{t\in [0,\infty)}\int_{|x|>R}|\nabla u_{C_*}(t,x)|^{2}+\tfrac{a}{|x|^2}| u_{C_*}(t,x)|^{2}+| u_{C_*}(t,x)|^{\frac{2(d+1)}{d-1}}+| u_{C_*}(t,x)|^{\frac{2d}{d-2}}dx\leq \epsilon.
				\end{equation}
			\end{proposition}

			\begin{proof}[{Proof of Theorem~\ref{sub-threshold}}]  We suppose Theorem~\ref{sub-threshold} fails and take a solution $u_{C_*}$ as in Theorem~\ref{CompacSolution}. By Lemma \ref{lem.var}, $E_a(u_{C_*})>0$.  We recall the virial identity(Lemma \ref{LocalVirial}) with $u=u_{C_*}$
				\begin{align}\label{Vzero11}
					I_{R}[u_{C_*}]=&2\IM\int_{\R^{d}} \nabla w_{R} \cdot\nabla u_{C_*} \,\overline{u_{C_*}} \,dx\\
					\partial_{t}I_R[u_{C_*}]=& 8\int_{\R^{d}}|\nabla u_{C_*}|^2dx+\tfrac{a}{\left|x\right|^2}|u_{C_*}|^2dx-|u_{C_*}|^{\tfrac{2d}{d-2}}dx+\tfrac{d}{d+1}|u_{C_*}|^\frac{2(d+1)}{d-1}\\
					&+\mathcal{O}\left(\int_{|x|\sim R}|\nabla u_{C_*}|^2+\tfrac{a}{|x|^2}|u_{C_*}|^2+|u_{C_*}|^\frac{2d}{d-2}+|u_c|^{\frac{2(d+1)}{d-1}}dx\right)
				\end{align}
				By sharp Sobolev embedding(Lemma \ref{Sharp-Sobolev}) and Lemma \ref{GlobalS}, $E_a(u_{C_*})>0$
				\begin{align}
					\int_{\Bbb{R}^d}\left|\nabla u_{C_*}\right|^2+\frac{a}{\left|x\right|^2}\left|u_{C_*}\right|^2-\left|u_{C_*}\right|^{\frac{2d}{d-2}}dx\gtrsim\Vert u\Vert_{\dot{H}_a^1}^2\gtrsim \eta_0>0.
				\end{align}
				Combining with Proposition \ref{BuIo},there exists $R_0>0 $ so that
				\begin{align*}
					I_{R_0}[u_{C_*}]\gtrsim \eta_0.
				\end{align*}
				On the other hand,
				\[
				\left|I_{R_0}[u_{C_*}]\right|\lesssim R_0\|u_{C_*}\|^{2}_{L^{\infty}_{t}H^{1}_{x}}\lesssim R_0.
				\]
				Thus, the Fundamental Theorem of Calculus implies 
				\[
				\eta_0 T\lesssim\left|\int^{T}_{0}\partial_{t}I_{R_0}[u_{C_*}]dt\right|\lesssim R_0\qtq{for any} T>0,
				\]
				which yields a contradiction for $T$ sufficiently large. Thus we finish the proof.
				\end{proof}
			\section{Compactness for non-scattering threshold solutions}\label{sec:threshold-solution}
		In this section, we will prove that the failure of Theorem~\ref{Threshold}~(i) implies the existence of a non-scattering solution $u$ to \eqref{NLS} at the energy threshold such that the orbit of $u$ is pre-compact in $\dot{H}_a^{1}(\R^{d})$ modulo scaling symmetry.
		
		\begin{proposition}\label{Compacness11}
			Assume that Theorem~\ref{Threshold}~(i) fails. Then there exists a solution ${u}: [0, T^{*})\times \R^{d}\to \C$ of \eqref{NLS} with
			\begin{align}\label{EGN}
				&E_a(u_{0})=E_a^{c}(W_a) \quad \text{and} \quad \| u_{0}\|_{\dot{H}_a^1}<\| W_a\|_{\dot{H}_a^1}\\
				\label{SN}
				&\|{u}\|_{L^{\frac{2(d+2)}{d-2}}_{t, x}([0, T^{*})\times \R^{d})}=\infty.
			\end{align}
			Moreover, there exists $\lambda:[0, T^{*}) \to (0, \infty)$ such that
			\begin{equation}\label{CompactX}
				\left\{\lambda(t)^{-\frac{d-2}{2}}u(t, \tfrac{x}{\lambda(t)}): t\in [0, T^{*})\right\} \quad
				\text{is pre-compact in $\dot{H}_a^{1}(\R^{d})$}.
			\end{equation}
		\end{proposition}
		\begin{proof}
			The proof follows that of \cite{AJZ} and we will only give the sketch of the proof.  If Theorem~\ref{Threshold}~(i) fails, then  there exists a  solution ${u}: [0, T^{*})\times \R^{d}\to \Bbb{C}$ 
			satisfying \eqref{EGN} and \eqref{SN}. In particular, we have by Lemma~\ref{GlobalS},
			\begin{align}\label{boundGC}
				&\| u(t)\|_{\dot{H}_a^1}<\| W_a\|_{\dot{H}_a^1}\qtq{for all $t\in [0, T^{*})$.}
			\end{align}
			It is sufficient to prove that for any sequence $\left\{\tau_{n}\right\}_{n\in \N}\subset [0, T^{*})$, there exists $\lambda_{n}$ such that 
			\begin{align}\label{LamC}
				\lambda_{n}^{-\frac{d-2}{2}}u\(\tau_{n}, \tfrac{x}{\lambda_{n}}\)
				\mbox{ converges strongly in $\dot{H}_a^{1}$ (up to a subsequence).}	
			\end{align}
			By continuity of $u$ it suffices to consider $\tau_{n}\to T^{*}$. Using \eqref{boundGC}, we apply the profile decomposition (Theorem~\ref{LinearProfi}) to $\left\{u(\tau_{n})\right\}_{n\in \N}$ and write
			\begin{equation}\label{Dpe}
				u_{n}:=u(\tau_{n})=\sum^{J}_{j=1}\phi^{j}_{n}+R_n^J
			\end{equation}
			with $J\leq J^{*}$ and all of the properties stated in Theorem~\ref{LinearProfi}. 
			
			Similar to the previous section,	we shall prove that there  exists  at most one nonzero $\phi_n^j$.
			
			\textbf{Scenario \uppercase\expandafter{\romannumeral1}}.  $J^{*}\geq 2$. 
			
			By Theorem~\ref{LinearProfi} and \eqref{boundGC} we have 
			\begin{align}\label{MassC}
				&\lim_{n \to \infty} \big(\sum_{j=1}^{J} M(\phi_n^j) + M(R_n^J)\big) = \lim_{n \to \infty} M(u_{n}) =M(u_0),\\
				\label{DECE}
				&\lim_{n \to \infty} \big(\sum_{j=1}^{J} E_a(\phi_n^j) + E_a(R_n^J)\big) = \lim_{n \to \infty} E_a(u_{n})=E_a^{c}(W),\\
				\label{DEG}
				&\limsup_{n \to \infty} \big(\sum_{j=1}^{J} \| \phi_n^j\|_{\dot{H}_a^{1}}^2 + \| R_n^J\|_{\dot{H}_a^{1}}^2\big)
				= \limsup_{n \to \infty} \|u(\tau_n)\|_{\dot{H}_a^{1}}^2 \leq \|W_a\|_{\dot{H}_a^{1}}^2
			\end{align}
			for any $0\leq J \leq J^{*}$. By Lemma \ref{lem.var} and \eqref{DEG} we deduce that for $n$ large,
			\begin{align}\label{BounWd}
				\| \phi_n^j\|_{\dot{H}_a^{1}}^2\leq \frac{\|W_a\|_{\dot{H}_a^{1}}^2}{E_a^{c}(W_a)} E_a(\phi_n^j)
				\qtq{and}
				\| R_n^J\|_{\dot{H}_a^{1}}^2\leq \frac{\|W_a\|_{\dot{H}_a^{1}}^2}{E_a^{c}(W)} E_a(R_n^J).
			\end{align}
			In particular, $\liminf\limits_{n\to \infty}E_a(\phi_n^j)>0$. Thus there exists  $\delta>0$ so that
			\begin{align}\label{MEC}
				E_a(\phi_n^j)&\leq E_a^{c}(W_a)-\delta,
			\end{align}
			for sufficiently large $n$ and $1\leq j\leq J$, so that each $ \phi_n^j$ satisfies \eqref{condition1} in Theorem~\ref{sub-threshold} (i).

			We then define nonlinear profiles $\psi_{n}^{j}$ associated to each $\phi_{n}^{j}$ as follows:
			\begin{itemize}
				\item If $\frac{|x^{j}_{n}|}{\lambda^{j}_{n}}\to \infty$ for some $j$, then we are in position to apply Proposition~\ref{embedding-nonlinear}, and hence we have a global solution $\psi_{n}^{j}$ of \eqref{NLS} with data $\psi_{n}^{j}(0)=\phi_{n}^{j}$.

				\item If $x^{j}_{n}\equiv 0$ and $\lambda^{j}_{n}\rightarrow 0$, we define $\psi_{n}^{j}$ to be the global solution of \eqref{NLS} with the initial data $\psi_{n}^{j}(0)=\phi_{n}^{j}$ guaranteed by Proposition~\ref{embedding-nonlinear}.
				
				\item If $x^{j}_{n}\equiv 0$, $\lambda^{j}_{n}\equiv 1$ and $t^{j}_{n}\equiv 0$, we take  $\psi^{j}$ to be the global solution of 
				\eqref{NLS} with the initial data $\psi^{j}(0)=\phi^{j}$.
				
				\item If $x^{j}_{n}\equiv 0$, $\lambda^{j}_{n}\equiv 1$ and $t^{n}_{j}\rightarrow\pm\infty$, we take $\psi^{j}$  to be the global solution  of \eqref{NLS} that scatters to $e^{-it\L_{a}}\phi^{j}$ in $H^{1}_{x}(\R^{d})$  as $t\rightarrow\pm\infty$.  In either case, we define the global solution to \eqref{NLS},
				\[\psi^{j}_{n}(t,x):=\psi^{j}(t+t^{j}_{n}, x).\]
			\end{itemize}
			We define the approximate solutions
			\[
			u^{J}_{n}(t):=\sum^{J}_{j=1}\psi^{j}_{n}(t)+e^{-it\L_{a}}W^{J}_{n},
			\] 
			Following the same strategy in proving Proposition \ref{PScondition}, one can show that $u_n^J(t)$ is the approximate solution of $u_n$ for large $n$ and $J\to J^*$. Thus we deduce that $\|u_n\|_{L_{t,x}^{\frac{2(d+2)}{d-2}}(\R \times\R^d)}<\infty$ which is contradict to \eqref{SN}, so Scenario \uppercase\expandafter{\romannumeral1} cannot occur.

			\textbf{Scenario \uppercase\expandafter{\romannumeral2}.} $J^*=1$. \eqref{Dpe} simplifies to
			\begin{align}\label{1Dec}
				u_{n}=u(\tau_{n})=\phi_{n}+W^{1}_{n}.
			\end{align}
			We now observe that 
			\begin{align}\label{ZerI}
				\| W^{1}_{n}\|_{\dot{H}_a^1}\to 0 \qtq{as $n\to +\infty$.}
			\end{align}
			Indeed, otherwise, by \eqref{BounWd} we see that there exists $c>0$ so that  $E_a( W^{1}_{n})\geq c$. Thus, $\phi_{n}$ obeys the sub-threshold condition \eqref{MEC}, and the same arguments used above show $u\in L^{\frac{2(d+2)}{d-2}}_{t, x}(\R\times \R^{d})$, contradicting \eqref{SN}. 
			
			First, we preclude the possibility that $t_{n}\rightarrow \pm\infty$ as $n\rightarrow\infty$. Without loss of generality, suppose $\tau_{n}\rightarrow \infty$.  If $\lambda_{n}\equiv1$, then by monotone convergence we have
			\[
			\| e^{it\mathcal{L}_a}\phi_{n}  \|_{L^{\frac{2(d+2)}{d-2}}_{t,x}([0,\infty)\times \R^{d})}= \|e^{it\mathcal{L}_a} \phi\|_{L^{\frac{2(d+2)}{d-2}}_{t,x}([\tau_n,\infty)\times \R^{d})}\rightarrow0
			\]
			as $n\rightarrow\infty$. On the other hand, if $\lambda_{n}\to +\infty$ a change of variables, Bernstein's inequality, Strichartz estimate, and  monotone convergence yield
			\[
			\| e^{it\mathcal{L}_a}\phi_{n}  \|_{L^{\frac{2(d+2)}{d-2}}_{t,x}([0,\infty)\times \R^{d})}=\|e^{it\mathcal{L}_a} P_{\geq (\lambda_{n})^{-\theta}}\phi\|_{L^{\frac{2(d+2)}{d-2}}_{t,x}([\tau_n,\infty)\times \R^{d})}\rightarrow0,
			\]
			as $n\rightarrow\infty$. In either case, we may apply Lemma~\ref{StabilityNLS} with $\tilde{u}:=e^{it\mathcal{L}_a}\phi_{n}$ to obtain that $u\in L^{\frac{2(d+2)}{d-2}}_{t,x}([0,\infty)\times \R^{d})$, contradicting \eqref{SN}.  Thus $t_n\equiv 0$.
			
			Next we show that $x_n\equiv0$.
			In fact, by Lemma \ref{GlobalS} and Proposition \ref{embedding-nonlinear} one can show that $u$ has an approximate solution which has is $L_{t,x}^{\frac{2(d+2)}{d-2}}$-bounded for sufficiently large $n$. It contradicts to \eqref{SN}. Thus $x_n\equiv0$.
			Finally, by \eqref{1Dec} and \eqref{ZerI} we get
			\[
			\|\lambda^{-\frac{d-2}{2}}_{n}u(\tau_{n}, \tfrac{x}{\lambda_{n}})-\phi\|^{2}_{\dot{H}_a^{1}}\to 0 \qtq{as $n\to \infty$.}
			\]
			Note that in the case $\lambda_{n}\to +\infty$, we have used the fact that 
			$\|P^a_{\geq (\lambda_{n})^{-\theta}}\phi-\phi\|_{\dot{H}_a^{1}}\to 0$ as $n\to \infty$. This complete the proof of \eqref{LamC}.
		\end{proof}

		\section{Preclusion of compact solutions: The proof of Theorem \ref{Threshold}(i)}\label{sec:preclude-threshold}
		
		In this section we complete the proof of Theorem~\ref{Threshold}(i). It's enough to prove that there are no solutions of \eqref{NLS} as in Proposition~\ref{Compacness11}.
		
		\subsection{Finite-time blow-up}\label{S:FTB}
		In this subsection we preclude the possibility of finite-time blow-up.
		\begin{theorem}\label{Tfinito}
			There are no solutions of \eqref{NLS} as in Proposition~\ref{Compacness11} with $T^{*}<\infty$.
		\end{theorem}
		\begin{proof}
			
			Suppose that $u:[0, T^{*})\times \R^{d}\to \C$ is such a solution. Now we claim that
			\begin{align}\label{LamdaI}
				\lim_{t\to T^{*}}\lambda(t)=\infty.
			\end{align}
			By contradiction, we suppose  that $\lim_{t\to T^{*}}\lambda(t)<\infty$. Then there exists a sequence $\{t_{n}\}$ that converges to $T^{*}$ such that $\lambda(t_{n})\to \lambda_{0}\in \R$. Indeed, we have $\lambda_{0}>0$. Thus, by \eqref{CompactX}, we obtain that there exists nonzero $\psi\in \dot{H}_a^{1}$ such that
			\begin{align}\label{CVs}
				\|u(t_{n})-\psi\|_{\dot{H}_a^{1}}\to 0 \qtq{as $n\to\infty$.}
			\end{align}
			Moreover,  we have $\|\psi\|^{2}_{L^{2}}\leq M(u_{0})$ by using the mass conservation and the uniqueness of the weak limit. Following the argument in \cite{KenigMerle2006}, \eqref{CVs} and Lemma \ref{StabilityNLS}, we can show that 
			\[
			\| u  \|_{S^{0}([T^{*}-\epsilon, T^{*}+\epsilon]\times \R^{d})}<\infty
			\]
			for some $0<\epsilon\ll1$, which contradicts \eqref{SN} and hence proves \eqref{LamdaI}.
			
			For a fixed $R>0$ and $t\in [0, T^*)$, we define 
			\[
			M_{R}(t)=\int_{\R^{d}}|u(t,x)|^{2}\xi\(\tfrac{x}{R} \)\,dx.
			\qtq{}
			\]
			where $\xi(x)=1$ if $|x|\leq 1$ and $\xi(x)=0$ if $|x|\geq 2$. Since 
			\[
			M^{\prime}_{R}(t)=\tfrac{2}{R}\IM \int_{\R^{d}}\overline{u}\nabla u\cdot (\nabla\xi)\(  \tfrac{x}{R} \)dx,
			\]
			it follows from \eqref{PositiveP} that $M^{\prime}_{R}(t)\lesssim_{W_a} 1$. Thus
			\begin{align}\label{InteF}
				M_{R}(t_{1})=M_{R}(t_{2})+\int^{t_{2}}_{t_{1}}M^{\prime}_{R}(t)dt \lesssim_{W_a}M_{R}(t_{2}) +|t_{1}-t_{2}|.
			\end{align}
			On the other hand, for $\mu\in (0,1)$ and $t\in [0, T^{*})$,
			\begin{align*}
				\int_{|x|\leq R}|u(t,x)|^{2}dx
				&\leq  
				\int_{|x|\leq \mu R}|u(t,x)|^{2}dx+\int_{\mu R\leq |x|\leq R}|u(t,x)|^{2}dx\\
				&\lesssim
				\mu^{2} R^{2}\|u(t)\|^{2}_{L^{\frac{2d}{d-2}}}+R^{2}\(\int_{|x|\geq \eta R}|u(t,x)|^{\frac{2d}{d-2}}dx\)^{\frac{d-2}{d}}\\
				&\lesssim \mu^{2} R^{2}\|\nabla W_a\|^{2}_{L^{2}}+R^{2}\(\int_{|x|\geq \eta R}|u(t,x)|^{\frac{2d}{d-2}}dx\)^{\frac{d-2}{d}}.
			\end{align*}
			By \eqref{CompactX} and \eqref{LamdaI} we see that for all $R_{0}>0$,
			\[\int_{|x|>R_{0}}|u(t,x)|^{\frac{2d}{d-2}}\,dx\to 0 \qtq{as $t\to T^{*}$.}\]
			Thus we deduce that for any $R>0$,
			\[
			\limsup_{t\nearrow T^{*}}\int_{|x|\leq R}|u(t,x)|^{2}dx=0.
			\]
			In particular, $M_{R}(t_{2})\to 0$ as $t_{2}\to T^{*}$, so that \eqref{InteF} implies
			\[
			M_{R}(t_{1}) \lesssim_{W_a}|T^{{*}}-t_{1}|.
			\]
			Letting $R\to\infty$ and using conservation of mass, we get $M(u_{0})\lesssim_{W}|T^{{*}}-t_{1}|$. Letting $t_{1}\to T^{*}$ we obtain ${u}_{0}=0$, which contradicts $E_a({u}_{0})= E^{c}(W_a)>0$.\end{proof}
		
		Combining the above result   with Theorem~\ref{sub-threshold}, we can deduce the following corollary.
		\begin{corollary}\label{remark}
			If $u$ is a maximal-lifespan solution to \eqref{NLS} satisfying 
			\[
			E_a({u}_{0})\leq E_a^{c}(W_a) \mbox{ and } \| u_{0}\|_{\dot{H}_a^1}\leq \| W_a\|_{\dot{H}_a^1},
			\]
			then $u$ is global in time.
		\end{corollary}

		\subsection{Infinite time blow-up}\label{S:ITB}In this subsection, we  preclude the possibility of infinite time blow-up.

		\begin{theorem}\label{ITfinito22}
			There are no solutions to \eqref{NLS} as in Proposition~\ref{Compacness11} with $T^{*}=\infty$.
		\end{theorem}
		
		We prove Theorem~\ref{ITfinito22} by contradiction. We suppose that $u:[0, \infty)\times \R^{d}\to \C$ is a solution to \eqref{NLS} as in Proposition~\ref{Compacness11} with $T^*=\infty$. In particular,  $u$ obeys
		\begin{align}\label{Eequal}
			E_a^c(u_{0})=E_a^{c}(W_a), \quad \| u_{0}\|_{\dot{H}_a^1}<\| W_a\|_{\dot{H}_a^1},\quad \text{and}\quad  \|u \|_{L_{t,x}^{\frac{2(d+2)}{d-2}}([0, \infty)\times\R^{d})}=\infty.
		\end{align}
		Moreover,  there exists compactness parameter  $\lambda_{0}:[0, \infty) \to (0,\infty)$ such that the set
		\begin{equation}\label{New0Compact}
			\left\{\lambda_{0}(t)^{-\frac{d-2}{2}}u(t, \tfrac{x}{\lambda_{0}(t)}): t\in [0, +\infty)\right\} \quad
			\text{is pre-compact in $\dot{H}_a^{1}(\R^{d})$}.
		\end{equation}
		
		We first show that the  parameter $\lambda_0(t)$ is equivalent to the modulation parameter $\mu(t)$ when $\delta(t)$ is sufficently small.
		\begin{lemma}\label{Parametrization}
			If $\delta_{0}$ is sufficiently small, then there exists $C>0$ so that
			\begin{align}\label{limia}
				\tfrac{\lambda_{0}(t)}{\mu(t)}+\tfrac{\mu(t)}{\lambda_{0}(t)}\leq C \quad \text{for  $t\in I_{0}$},
			\end{align}
			where the parameter $\mu(t)$ is given in Proposition~\ref{Modilation11}.
		\end{lemma}
		
		\begin{proof}
			By \eqref{DecomU} and \eqref{Estimatemodu},  the following holds true 
			\[
			\int_{\frac{1}{\mu(t)}\leq|x|\leq\frac{2}{\mu(t)}}|\nabla u(t,x)|^{2}dx
			\geq \tfrac{1}{2}\int_{1\leq|x|\leq 2}|\nabla W(x)|^2dx-C\delta^{2}(t)\geq c
			\]
			for  $\delta_{0}\ll 1$, $t\in I_{0}$ and some positive constant $c$. By a change of variable, we deduce
			\begin{align}\label{PoW}
				\int_{\frac{\lambda_{0}(t)}{\mu(t)}\leq|x|\leq\frac{2\lambda_{0}(t)}{\mu(t)}}\tfrac{1}{\lambda_{0}(t)^{d}}
				\left|\nabla u\(t,\tfrac{x}{\lambda_{0}(t)}\)\right|^{2}dx\geq c\qtq{for all}t\in I_0.
			\end{align}
			Thus \eqref{limia} follows from  \eqref{New0Compact}.
		\end{proof}
		
		As a consequence of the above lemma, we deduce that
		\begin{equation}\label{CompacNew}
			\left\{\lambda(t)^{-\frac{d-2}{2}}u(t, \tfrac{x}{\lambda(t)}): t\in [0, \infty)\right\} \quad
			\text{is pre-compact in $\dot{H}_a^{1}(\R^{d})$},
		\end{equation}
		where
		\[
		\lambda(t)=
		\begin{cases}
			\lambda_{0}(t)& \quad t\in [0, \infty)\setminus I_{0},\\
			\mu(t)&\quad t\in I_{0}.
		\end{cases}
		\]
		We note that $\lambda(t)\geqslant1$ for $t\geqslant0$.
		
		Next we will prove some essential lemmas which are the key to proving Theorem~\ref{ITfinito22}.

		\begin{lemma}\label{SequenceInf}
			For any sequence $\left\{t_{n}\right\}\subset [0, \infty)$,  the following holds
			\begin{equation}\label{ZeroPoten}
				\lambda(t_{n})\to \infty \quad \text{if and only if}\quad \int_{\R^{d}}|u(t_{n},x)|^{\frac{2d+2}{d-1}}dx\to 0.
			\end{equation}
		\end{lemma}
		\begin{proof}
			Suppose  $\lambda(t_{n})\to \infty$ and $\epsilon>0$. Using \eqref{CompacNew}, there exists $\rho_{\epsilon}>0$ so that
			\[
			\int_{{|x|\geq \tfrac{\rho_{\epsilon}}{\lambda(t_{n})}}}\bigl[|\nabla u(t_{n},x)|^{2}+|u(t_{n},x)|^{\frac{2d}{d-2}}\bigr]dx\ll\epsilon \qtq{for all $n\in \N$.}
			\]
			we obtain by using interpolation and the mass conservation $M(u(t_n))=M(u_0)$,
			\begin{align}\label{PNe}
				\int_{{|x|\geq \tfrac{\rho_{\epsilon}}{\lambda(t_{n})}}}| u(t_{n},x)|^{\frac{2d+2}{d-1}}dx <\epsilon \qtq{for all $n\in \N$.}
			\end{align}
			Next, since $\tfrac{1}{\lambda(t_{n})}\to 0$ as $n\to\infty$, we deduce from holder inequality that
			\begin{align}\label{Elar}
				\int_{{|x|\leq \tfrac{\rho_{\epsilon}}{\lambda(t_{n})}}}| u(t_{n},x)|^{\frac{2d+2}{d-1}}dx  \to 0 \qtq{as $n\to \infty$.}
			\end{align}
			Combining \eqref{PNe} and \eqref{Elar} we have  $$\int_{\R^{d}}|u(t_{n},x)|^{\frac{2d+2}{d-1}}dx\to0,\quad n\to+\infty.$$
			
			On the other hand, we assume that
			\begin{align}\label{Otra}
				\int_{\R^{d}}|u(t_{n},x)|^{\frac{2d+2}{d-1}}dx\to 0,
			\end{align}
			but   $\left\{\lambda(t_{n})\right\}_{n\in\N}$ converges.  From \eqref{CompacNew}, we see that there exists $\phi\in \dot{H}_a^{1}(\R^{d})$ so that
			\begin{equation}\label{contra11}
				u(t_{n})\to \phi\quad \text{in $\dot{H}_a^{1}(\R^{d})$}.
			\end{equation}
			Combining \eqref{Otra} and \eqref{contra11} we have $E_a^{c}(W_a)=E_a^{c}(\phi)$, so that $\phi\neq 0$. Moreover, since $M(u(t_{n}))=M(u_{0})$, we deduce that $\phi\in L^{2}(\R^{d})$. 
			
			However, by the sharp Sobolev inequality and \eqref{contra11} we get
			\[
			\|u(t_{n})-\phi\|_{L^{\frac{2d+2}{d-1}}}\lesssim \|u(t_{n})-\phi\|^{\frac{1}{d+1}}_{L^{2}}\|u(t_{n})-\phi\|^{\frac{d}{d+1}}_{\dot{H}_a^{1}}
			\to 0 \qtq{as $n\to \infty$.}
			\]
			Thus, by \eqref{Otra} we have that $\|\phi\|_{L^{\frac{2d+2}{d-1}}}=0$, which is contradict to $\phi\neq0$. \end{proof}

		\begin{lemma}\label{DeltaZero}
			If $t_{n}\to \infty$, then
			\begin{equation}\label{equivalence}
				\lambda(t_{n})\to \infty \quad \text{if and only if}\quad \delta(t_{n})\to 0.
			\end{equation}
		\end{lemma}
		\begin{proof} If $\delta(t_{n})\to 0$,  then Proposition \ref{Modilation11} and 
			Lemma \ref{SequenceInf} imply $\lambda(t_{n})\to \infty$.
			
			Next, by contradiction we assume that $\lambda(t_{n})\to \infty$ but 
			\begin{equation}\label{Zeroplus}
				\delta(u(t_{n}))\geq c>0.
			\end{equation}
			Notice that 
			\[
			\left\{\lambda(t)^{-\frac{d-2}{2}}u(t, \tfrac{x}{\lambda(t)}): t\in [0, \infty)\right\} \quad
			\text{is pre-compact in $\dot{H}_a^{1}(\R^{d})$},
			\] 
			there exists $v_{0}\in \dot{H}_a^{1}(\R^{d})$ such that 
			\begin{equation}\label{CompactV}
				\lambda(t_{n})^{-\frac{d-2}{2}}u(t_{n}, \tfrac{x}{\lambda(t_{n})}) \to v_{0} \quad \text{in}\quad \dot{H}_a^{1}(\R^{d}).
			\end{equation}
			In particular, from \eqref{Eequal}, \eqref{ZeroPoten} and \eqref{Zeroplus}  we have
			\begin{align}\label{ScR}
				E_a^{c}(v_{0})=E_a^{c}(W_a) \quad \text{and}\quad \| v_{0}\|_{\dot{H}_a^1}<\| W_a\|_{\dot{H}_a^1}.
			\end{align}
			
			Now let $v$ be the solution of the energy-critical NLS  \eqref{NLS-Potential} with$v|_{t=0}=v_{0}$. By Theorem \ref{th.classifa}, $v$ is global and its scattering size is finite 
			\[
			\|v\|_{L^{\frac{2(d+2)}{d-2}}_{t,x}([0, \infty)\times \R^{d})}<\infty\qtq{or}\|v\|_{L^{\frac{2(d+2)}{d-2}}_{t,x}((-\infty,0]\times \R^{d})}<\infty.
			\]
			In either case, we will use stability theory to deduce that 
			$u\in L^{\frac{2(d+2)}{d-2}}_{t,x}([0, \infty)\times \R^{d})$, contradicting \eqref{Eequal}.
			
			The argument is similar to the proof of Proposition~\ref{embedding-nonlinear}. For simpleness, we suppose that $\|v\|_{L^{\frac{2(d+2)}{d-2}}_{t,x}([0, \infty)\times \R^{d})}<+\infty$. For  the case of  $\|v\|_{L^{\frac{2(d+2)}{d-2}}_{t,x}((-\infty,0]\times \R^{d})}<+\infty$, it can be handled similarly. 
			
			Let $w_{n}$ be the solution to the energy-critical NLS with inverse-square potential \eqref{NLS-Potential} with initial data $w_{n}(0)=P^a_{\geq (\lambda_{n})^{-\theta}}v_{0}$ for some $0<\theta \ll1$, where
			$\lambda_{n}:=\lambda(t_{n})$.  Since $\lambda_{n}\to \infty$ as $n\to\infty$, 
			\begin{align}\label{cvv}
				\| P^a_{\geq (\lambda_{n})^{-\theta}}v_{0}-v_{0} \|_{\dot{H}^{1}_{x}}\rightarrow 0 \quad \text{as $n\rightarrow \infty$}.
			\end{align}
			By \eqref{cvv} and Lemma \ref{StabilityNLS}, we deduce that
			\begin{align*}
				\| w_{n}   \|_{S_a^{1}([0, \infty)\times \R^{d})}\lesssim 1\qtq{for $n$ large.}
			\end{align*}
			Using Benstein's estimate, we obtain
			\[
			\| P^a_{\geq (\lambda_{n})^{-\theta}}v_{0}\|_{{L}^{2}}\lesssim\lambda^{\theta}_{n}\|  v_{0}\|_{\dot{H}_a^1},
			\]
			which implies 
			\begin{align}\label{wpersi}
				\|w_{n}\|_{S_a^{0}([0, \infty)\times \R^{d})}\leq C(\|v_{0}\|_{\dot{H}_a^{1}})\lambda^{\theta}_{n}
			\end{align}
			by persistence of regularity.
			Arguing as in  Proposition~\ref{embedding-nonlinear} one can show that
			\[
			\tilde{u}_{n}(t,x)=\lambda^{\frac{d-2}{2}}_{n}w_{n}(\lambda^{2}_{n}t, \lambda_{n}x)
			\]
			are approximate solutions to equation \eqref{NLS} on $[0, \infty)$. Indeed, following the same argument as in \eqref{estimate33}, we obtain that
			\[
			\lim_{n\to \infty}\|\sqrt{\mathcal{L}_a}\bigl( |\tilde{u}_{n}|^{\frac{4}{d-1}}\tilde{u}_{n}\bigr) \|_{L^{\frac{2(d+2)}{d+4}}_{t, x}(\R\times \R^{d})}=0.
			\]
			Then combining \eqref{CompactV} and \eqref{cvv}  we get
			\[
			\| \tilde{u}_{n}(0)-u(t_{n}) \|_{\dot{H}_a^{1}}\rightarrow 0 \quad \text{as $n\rightarrow \infty$}.
			\]
			Applying Lemma \ref{StabilityNLS}, we deduce that for $n$ large,
			\[\|u(t_{n}+t)\|_{L^{\frac{2(d+2)}{d-2}}_{t,x}([0, +\infty)\times \R^{d})}
			=\|u\|_{L^{\frac{2(d+2)}{d-2}}_{t,x}([t_{n}, +\infty))}\lesssim 1,
			\]
			which is	contradict to \eqref{Eequal}.\end{proof}

		\begin{lemma}\label{Compa}
			There exists $c>0$ so that
			\begin{equation}\label{InequeN}
				{{F^{c}_{\infty}}}[u(t)]=8\|\nabla u(t)\|^{2}_{L^{2}}+8a\int_{\R^d}\frac{|u|^2}{|x|^2}dx-8\|u(t)\|^{\frac{2d}{d-2}}_{L^{\frac{2d}{d-2}}}\geq c\delta(t)
				\qtq{for all $t\geq 0$.}
			\end{equation}
		\end{lemma}
		\begin{proof}
			We recall that $E_a^{c}(W_a)=\frac{1}{d}\| W_a\|^{2}_{\dot{H}_a^1}$.  By using $E_a(u)=E_a^{c}(W_a)$ we have 
			\begin{equation}\label{Pohi22}
				{{F^{c}_{\infty}}}[u(t)]=\tfrac{16}{d-2}\delta(t)-\tfrac{8d^2-8d}{d^2-d-2}\|u(t)\|^{\frac{2d+2}{d-1}}_{L^{\frac{2d+2}{d-1}}}.
			\end{equation}
			We claim that ${{F^{c}_{\infty}}}[u(t)]>0$ for all $t\geq 0$. Indeed, by using sharp Sobolev inequality \eqref{Sharp-Sobolev} and \eqref{PositiveP} 
			we deduce 
			\begin{align}\nonumber
				\|u\|^{\frac{2d}{d-2}}_{L^{\frac{2d}{d-2}}}&\leq  C^{\frac{2d}{d-2}}_{GN}\| u\|_{\dot{H}_a^1}^{\frac{2d}{d-2}}
				= \frac{1}{\| W_a\|^{\frac{4}{d-2}}_{\dot{H}^1_a}}\|  u\|^{\frac{2d}{d-2}}_{\dot{H}^1_a}\\ \label{IGU}
				&=\( \frac{\| u\|_{\dot{H}_a^1}}{\| W_a\|_{\dot{H}_a^1}} \)^{\frac{4}{d-2}}\| u\|^{2}_{\dot{H}_a^1}
				<\| u\|^{2}_{\dot{H}_a^1}.
			\end{align}
			
			Suppose  that \eqref{InequeN} were false. Then there exists  $\left\{t_{n}\right\}$ such that
			\begin{equation}\label{Contra22}
				{{F^{c}_{\infty}}}[u(t_{n})]\leq \tfrac{1}{n}\delta(t_{n}).
			\end{equation}
			We first observe that $\left\{\delta(t_{n})\right\}$ is bounded (cf. Lemma~\ref{GlobalS}). Now, we show that
			\[
			\delta(t_{n})\to 0 \quad \text{as $n\to\infty$}.
			\]
			In fact ${{F^{c}_{\infty}}}[u(t_{n})]\to 0$ as $n\to \infty$ by \eqref{Contra22}, so that \eqref{IGU} implies
			\begin{align*}
				\| u(t_n)\|_{\dot{H}_a^1}
				\rightarrow \| W_a\|_{\dot{H}_a^1},
			\end{align*}
			i.e.  $\delta(t_{n})\to 0$ as $n\to\infty$.  By Proposition \ref{Modilation11} we have
			\[
			\| u(t_n)\|^{\frac{2d+2}{d-1}}_{L^\frac{2d+2}{d-1}}\lesssim \delta(t_{n})^{2}\leq \tfrac{d^2-d-2}{8d^2-8d}\delta(t_{n})\quad \text{for $n$ large}.
			\]
			Thus, by using \eqref{Pohi22} and \eqref{Contra22} we get
			\[
			\tfrac{18-d}{d-2}\delta(t_{n})\leq \tfrac{1}{n}\delta(t_{n})\quad \text{for $n$ large},
			\]
			which is a contradiction because $\delta(t)>0$ for all $t\geq 0$.
		\end{proof}
		
		
		\begin{lemma}\label{Lemma111}
			For $\delta_{1}\in (0, \delta_{0})$ sufficiently small, there exists a constant $C=C(\delta_{1})>0$ such that for any interval $[t_{1}, t_{2}]\subset [0, \infty)$,
			\begin{equation}\label{BoundT11}
				\int^{t_{2}}_{t_{1}}\delta(t)\,dt\leq C\sup_{t\in[t_{1},t_{2}]}\tfrac{1}{\lambda(t)^{2}}
				\left\{\delta(t_{1})+\delta(t_{2})\right\}.
			\end{equation}
		\end{lemma}
		\begin{proof}
			Let   $\delta_{1}\in (0, \delta_{0})$ and $R>1$  be two parameters to be determined below. We use the localized virial identities  with $\chi(t)$ satisfying
			\[
			\chi(t)=
			\begin{cases}
				1& \quad \delta(t)<\delta_{1} \\
				0& \quad \delta(t)\geq \delta_{1}.
			\end{cases}
			\]
			From  Lemmas~\ref{VirialModulate} and Lemma \ref{Compa} we  deduce
			\begin{equation}\label{VirilaX11}
				\tfrac{d}{dt}I_{R}[u(t)]=F^{c}_{\infty}[u(t)]+\EE(t)\geq c\delta(t)+\EE(t)
			\end{equation}
			with
			\begin{equation}\label{Error111}
				\EE(t)=
				\begin{cases}
					F_{R}[u(t)]-F^{c}_{\infty}[u(t)]& \quad  \text{if $\delta(t)\geq \delta_{1}$}, \\
					F_{R}[u(t)]-F^{c}_{\infty}[u(t)]-\M[t]& \quad \text{if $\delta(t)< \delta_{1}$},
				\end{cases}
			\end{equation}
			where
			\begin{equation}\label{Error222}
				\M(t)=F^{c}_{R}[e^{ i \theta(t)}\lambda(t)^{\frac{d-2}{2}}W_a(\lambda(t)x)]
				-F^{c}_{\infty}[e^{ i \theta(t)}\lambda(t)^{\frac{d-2}{2}}W_a(\lambda(t)x)].
			\end{equation}
			
			On the one hand, it has been proven that 
			\begin{align}\label{EstimateV111}
				|I_{R}[u(t)]|\lesssim R^{2} \delta(t).
			\end{align}
			for any $R>1$.
			
			On the other hand, we shall prove that for any $\epsilon>0$, there exists $\rho_{\epsilon}=\rho(\epsilon)>0$ so that if
			\begin{equation}
				R:=\rho_{\epsilon}\sup_{t\in [t_{1}, t_{2}]}\tfrac{1}{\lambda(t)},
			\end{equation}
			then
			\begin{align}\label{EstimateE111}
				\EE(t)\geq-\epsilon\geq- \tfrac{\epsilon}{\delta_{1}}\delta(t) \quad &\text{uniformly for $t\in [t_{1}, t_{2}]$ such that
					$\delta(t)\geq \delta_{1}$},\\
				\label{EstimateE221}
				|\EE(t)|\leq (\epsilon+\delta_{1}) \delta(t)\quad &\text{uniformly for $t\in [t_{1}, t_{2}]$ such that $\delta(t)< \delta_{1}$}.
			\end{align}
			In fact,for $\delta(t)<\delta_1$,
			\begin{align}\label{F11}
				\EE(t)=F_{R}[u(t)]-F^{c}_{\infty}[{u}(t)]&=\int_{|x|\geq R}(- \Delta \Delta w_{R})|u|^{2}
				+4\RE \overline{u_{j}} u_{k} \partial_{jk}[w_{R}]-8|\nabla u|^{2}dx\\ \label{F22}
				&+\int_{|x|>R}(%
				\frac{4a}{|x|^{4}}x\nabla \phi _{R}|u(t)|^{2}-\frac{8a|u(t)|^{2}}{|x|^{2}}%
				)dx.\\ \label{F33}
				&-\frac{4}{d}\int_{|x|\geq R}\Delta[w_{R}(x)]|u|^{\frac{2d}{d-2}}dx
				+8\int_{|x|\geq R}|u|^{\frac{2d}{d-2}}dx
				\\ \label{F44}
				&+\frac{4}{d+1}\int_{\R^{d}}\Delta[w_{R}(x)]|u|^{\frac{2d+2}{d-1}}dx.
			\end{align}
			By compactness  \eqref{CompacNew} and mass conservation, we show that
			\begin{align}\label{epeq}
				\sup_{t\in \R}\int_{|x|>\tfrac{\rho_{\epsilon}}{\lambda(t)}}\big[|\nabla u|^{2} +|u|^{\frac{2d}{d-2}}+|u|^{\frac{2d+2}{d-1}} +\tfrac{|u|^{2}}{|x|^{2}}\big](t,x)dx  \ll \epsilon
			\end{align}
			for large enough $\rho_\eps>0$.  Given $T>0$, we now choose
			\[
			R:=\sup_{t\in [t_1,t_2]}\frac{\rho_{\epsilon}}{\lambda(t)},\qtq{so that}
			\left\{|x|\geq R\right\}\subset \left\{|x|\geq \tfrac{\rho_{\epsilon}}{\lambda(t)}\right\}\qtq{for all}t\in[t_1,t_2].
			\]
			It follows that
			\begin{equation}\label{smallI}
				|\eqref{F11}|+|\eqref{F22}|+|\eqref{F33}|<\epsilon \qtq{for all $t\in [0, T]$.}
			\end{equation}
			Next,  we use \eqref{epeq} and  $|\Delta[w_{R}](x)|\lesssim 1$ to get
			\[
			\eqref{F44}
			=2d\int_{|x|\leq R}|u|^{\frac{2d+2}{d-1}}dx+\int_{|x|\geq R}\Delta[w_{R}(x)]|u|^{\frac{2d+2}{d-1}}dx\geq -\epsilon\]
			So \eqref{EstimateE111} holds.
			For $\delta<\delta_1$
			
			For  $\delta(t)< \delta_{1}$. We recall the notation $\tilde{W}_a(t)=e^{i\theta(t)}\lambda(t)^{\frac{d-2}{2}}W_a(\lambda(t)x)$. Then
			\begin{align}\label{Decomp111}
				\EE(t)&=-8\int_{|x|\geq R}[(|\nabla u|^{2}-(|\nabla \tilde{W}_a(t)|^{2})]\,dx
				+8\int_{|x|\geq R}[|u|^{\frac{2d}{d-2}}-|\tilde{W}_a(t)|^{\frac{2d}{d-2}}]\,dx\\\label{Decomp222}
				&\quad+\int_{R\leq |x| \leq 2R}(-\Delta \Delta w_{R})[|u|^{2}-|\tilde{W}_a(t)|^{2}]\,dx \\\label{Decomp333}
				&\quad+4\int_{|x|>R}(%
				\frac{a}{|x|^{4}}x\nabla w _{R}|u(t)|^{2}-\frac{a}{|x|^{4}}x\nabla w _{R}|\tilde{W}_a(t)|^{2} %
				)dx.\\ \label{Decomp444}
				&\quad-4\int_{|x|>R}(%
				\frac{2a|u(t)|^{2}}{|x|^{2}}-\frac{2a|\tilde{W}_a(t)|^{2}}{|x|^{2}}%
				)dx.\\ \label{Decomp555}
				&\quad-\tfrac{4}{d}\int_{|x|\geq R}\Delta[w_{R}(x)](|u|^{\frac{2d}{d-2}}-|\tilde{W}_a(t)|^{\frac{2d}{d-2}})\,dx\\\label{Decomp666}
				&\quad+4\sum_{j, k}\RE\int_{|x|\geq R}[\overline{u_{j}} u_{k}-\overline{\partial_{j}W(t)} \partial_{k}\tilde{W}_a(t)]\partial_{jk}[w_{R}(x)]\,dx\\ \label{Decomp777}
				&\quad+\frac{4}{d+1}\int_{\R^{d}}\Delta[w_{R}(x)]|u|^{\frac{2d+2}{d-1}}\,dx
			\end{align}
			for all $t\geq 0$ such that $\delta(t)<\delta_{1}$.
			
			As $|\Delta \Delta w_{R}|\lesssim 1/|x|^{2}$,  $|\nabla w_R|\lesssim |x|$ , $|\partial_{jk}[w_{R}]|\lesssim 1$ and $|\Delta[w_{R}]|\lesssim 1$, \eqref{Decomp111}--\eqref{Decomp666} can be bounded by the following terms
			\begin{align}
				&\bigg\{\|u(t)\|_{\dot{H}_{x}^{1}(|x|\geq R)}+\bigg\|\frac{u(t)}{x}\bigg\|_{L^2(|x|\geq R)}+\|\tilde{W}_a(t)\|_{\dot{H}_{x}^{1}(|x|\geq R)}+ \bigg\|\frac{\tilde{W}_a(t)}{x}\bigg\|_{L^2(|x|\geq R)}\notag\\
				&+
				\|u(t)\|^{\frac{d+2}{d-2}}_{L_{x}^{\frac{2d}{d-2}}(|x|\geq R)}+\big\|\tilde{W}_a(t)\big\|^{\frac{d+2}{d-2}}_{L_{x}^{\frac{2d}{d-2}}(|x|\geq R)}
				\bigg\}\|g(t)\|_{\dot{H}_{a}^{1}},
			\end{align}
			where $ g(t)=e^{-i\theta(t)}[{u}(t)-\tilde{W}_a(t)]$. Moreover, since 
			\begin{gather*}
				\|W_a\|_{\dot{H}_{x}^{1}(|x|\geq R)}=O(R^{-\frac{\beta(d-2)}{2}}),\quad\|W_a\|_{L_{x}^{\frac{2d}{d-2}}(|x|\geq R)} = O(R^{-\frac{\beta(d-2)}{2}}),\\
				\bigg\|\frac{W_a}{x}\bigg\|_{L_{x}^{2}(|x|\geq R)} = O(R^{-\frac{\beta(d-2)}{2}}),
			\end{gather*} 
			we have that
			\begin{align*}
				&\|u(t)\|_{\dot{H}_{x}^{1}(|x|\geq R)}+\bigg\|\frac{u(t)}{x}\bigg\|_{L^2(|x|\geq R)}+\|\tilde{W}_a(t)\|_{\dot{H}_{x}^{1}(|x|\geq R)}+ \bigg\|\frac{\tilde{W}_a(t)}{x}\bigg\|_{L^2(|x|\geq R)}\notag\\
				&+
				\|u(t)\|^{\frac{d+2}{d-2}}_{L_{x}^{\frac{2d}{d-2}}(|x|\geq R)}+\big\|\tilde{W}_a(t)\big\|^{\frac{d+2}{d-2}}_{L_{x}^{\frac{2d}{d-2}}(|x|\geq R)}\\
				&\leq \delta(t)+\delta(t)^{\frac{\beta(d+2)}{2}}+O((\lambda(t)R)^{-\frac{\beta(d+2)}{2}})+O\big(((\lambda(t)R)^{-\frac{\beta(d+2)}{2}}\big).
			\end{align*}
			Thus, for $\delta_{1}\in (0, \delta_{0})$ sufficiently small and $R$ sufficiently large,
			\[
			|\eqref{Decomp111}|+|\eqref{Decomp222}|+|\eqref{Decomp333}|+|\eqref{Decomp444}|+|\eqref{Decomp555}|+|\eqref{Decomp666}|\leq \delta(t) \qtq{when $\delta(t)<\delta_{1}$.}
			\]
			Taking $\delta_{1}$ small if necessary, Proposition~\ref{Modilation11} implies
			\[
			\frac{4}{d+1}\left|\int_{\R^{d}}\Delta[w_{R}(x)]|u|^{\frac{2d+2}{d-1}}dx\right|\lesssim \delta(t)^{2}\lesssim {\delta}_{1} \delta(t)\leq \delta(t),
			\]
			Putting together the estimates above yields \eqref{EstimateE221}.
			
			Integrating inequality \eqref{VirilaX11} on $[t_{1}, t_{2}]$ and using  the estimates \eqref{EstimateV111}, \eqref{EstimateE111} and \eqref{EstimateE221}, we deduce
			\[
			\int^{t_{2}}_{t_{1}}\delta(t)\,dt\lesssim 
			\rho^{2}_{\epsilon}\sup_{t\in [t_{1}, t_{2}]}\tfrac{1}{\lambda(t)^{2}}(\delta(t_{1})+\delta(t_{2}))
			+(\tfrac{\epsilon}{\delta_{1}}+\epsilon+\delta_{1})\int^{t_{2}}_{t_{1}}\delta(t)\,dt.
			\]
			For $\delta_{1}\in (0, \delta_{0})$ small, and  choosing $\epsilon=\epsilon(\delta_{1})$ sufficiently small yields \eqref{BoundT11}. 
		\end{proof}

		As a direct consequence of the lemma, we can show the existence of a sequence $t_n\to \infty$ so that $\delta(t_n)\to 0$:
		
		\begin{proposition}\label{BoundedUN}
			There exists a sequence $t_{n}\to\infty$ such that
			\[
			\lim_{n\to \infty}\delta(t_{n})=0.
			\]
		\end{proposition}
		\begin{proof} First we recall that
			\begin{align}\label{Bla}
				\lambda(t)\gtrsim 1\qtq{for all $t\geq 0$.}
			\end{align}
			By \eqref{BoundT11},we have 
			$$\int_0^{T}\delta(t) \lesssim 1 \qtq{for all $T> 0$.}$$
			So we can find a such sequence ${t_n}$.
		\end{proof}
		We next show that we may use $\delta$ to control the variation of the spatial scale.

		\begin{proposition}\label{Spatialcenter}
			Let $[t_{1}, t_{2}]\subset(0, \infty)$ with $t_{1}+\tfrac{1}{\lambda(t_{1})^{2}}< t_{2}$. Then
			\begin{equation}\label{BoundCenter1}
				\left|\tfrac{1}{\lambda(t_{2})^{2}}-\tfrac{1}{\lambda(t_{1})^{2}}\right|\leq \widetilde{C}\int^{t_{2}}_{t_{1}}\delta(t)\,dt.
			\end{equation}
		\end{proposition}
		
		\begin{proof}
			\textsl{Step 1.}  First we show that there exists a constant $C_{1}$ such that
			\begin{equation}\label{step11}
				\tfrac{\lambda(s)}{\lambda(t)}+\tfrac{\lambda(t)}{\lambda(s)}\leq C_{1} \quad \text{for all $t$, $s\geq 0$ such that $|t-s|\leq \tfrac{1}{\lambda(s)^{2}} $}.
			\end{equation}
			By contradiction, we	suppose  that there exist two sequences $s_{n}$ and  $t_{n}$  such that
			\begin{align}\label{CBound}
				|t_{{n}}-s_{n}|\leq \tfrac{1}{\lambda(s_{n})^{2}} \qtq{but}
				\tfrac{\lambda(s_{n})}{\lambda(t_{n})}+\tfrac{\lambda(t_{n})}{\lambda(s_{n})}\to \infty.
			\end{align}
			Passing to a subsequence, one can suppose that
			\[
			\lim_{n\to \infty}\lambda(s_{n})^{2}(t_{n}-s_{n})=\tau_{0}\in[-1, 1].
			\]

			Recall that $\lambda(t)\gtrsim 1$ for all $t\geq 0$, by \eqref{CBound} we observe that
			\[
			\lambda(s_{n})+\lambda(t_{n}) \to \infty \qtq{as $n\to \infty$.}
			\]
			
			\textbf{Case I.}   $\lambda(s_{n}) \to \infty$ as $n\to\infty$. From \eqref{CBound} we deduce $|t_{n}-s_{n}|\to 0$ as $n\to \infty$. Then by the local 
			theory for \eqref{NLS} and $H_a^1$-boundedness we may obtain that 
			\[
			\|u(t_{n})-u(s_{n})\|_{\dot{H}^{1}}\lesssim \|u(t_{n})-u(s_{n})\|_{\dot{H}_a^{1}}\to 0
			\qtq{as $n\to \infty$. }
			\]
			
			Using the change of variables, we deduce that			\[
			\|\lambda(s_{n})^{-{\frac{d-2}{2}}}u(t_{n}, \tfrac{\cdot}{\lambda(s_{n})})-\lambda(s_{n})^{-{\frac{d-2}{2}}}u(s_{n}, \tfrac{\cdot}{\lambda(s_{n})})\|_{\dot{H}_a^{1}}\to 0
			\qtq{as $n\to \infty$.}
			\]
			Now, by compactness, $\psi\neq 0$, thus we have
			\begin{align}\label{CvC}
				\lambda(s_{n})^{-{\frac{d-2}{2}}}u(t_{n}, \tfrac{\cdot}{\lambda(s_{n})}) \to \psi 
				\qtq{in $\dot{H}_a^{1}$ as $n\to \infty$}
			\end{align}
			along some subsequence of $\{s_n\}$. Writing $w_{n}(x):=\lambda(s_{n})^{-{\frac{d-2}{2}}}u(t_{n}, \tfrac{\cdot}{\lambda(s_{n})})$, we can use the compactness condition again to deduce 
			\begin{align}\label{2dce}
				\bigl(\tfrac{\lambda(s_{n})}{\lambda(t_{n})}\bigr)^{\frac{d-2}{2}} w_{n}(\tfrac{\lambda(s_{n})}{\lambda(t_{n})}x)=\lambda(t_{n})^{-{\frac{d-2}{2}}}u(t_{n}, \tfrac{\cdot}{\lambda(t_{n})})\to \zeta \qtq{in $\dot{H}_a^{1}$ as $n\to \infty$}
			\end{align}
			for some $\zeta$, which is necessarily nonzero due to conservation of energy. Finally, by \eqref{2dce} we see that there exists $C_0,C_1>0$ so that
			\begin{align}\label{contradict}
				\int_{C_0\,\tfrac{\lambda(s_{n})}{\lambda(t_{n})}\leq |x|\leq 2C_0\,\tfrac{\lambda(s_{n})}{\lambda(t_{n})}}
				|\nabla w_{n}(x)|^{2}dx\geq C_1\mbox{ for all large }n.
			\end{align}
			However, using \eqref{CvC}, \eqref{CBound} and Lebesgue's dominated convergence  theorem to deduce that
			\[
			\int_{C_0\,\tfrac{\lambda(s_{n})}{\lambda(t_{n})}\leq |x|\leq 2C_0\,\tfrac{\lambda(s_{n})}{\lambda(t_{n})}}
			|\nabla w_{n}(x)|^{2}dx\to 0
			\]
			as $n\to \infty$, which is  contradict to \eqref{contradict}.
			
			\textbf{Case II.}  $\lambda(t_{n})\to \infty$. We will show that $\lambda(s_{n})\to \infty$.
			
			By contradiction, we assume that $\lambda(s_{n})$  is bounded.  By \eqref{CBound} we suppose that there exists $\tau\in \R$ such that
			$t_{n}-s_{n}\to \tau$.  Since  $\lambda(s_{n})$  is bounded, by compactness \eqref{CompacNew}, it is not difficult to verify that up to subsequence,  $u(s_{n})\to v_{0}$ in $\dot{H}_a^{1}$ as $n\to +\infty$ for some $v_{0}\in H_a^{1}$. By conservation of mass, we also have that $u(s_{n})\to v_{0}$ in $L^{\frac{2d+2}{d-1}}$. Especially, $E_a(v_{0})=E_a^{c}(W_a)$. Furthermore, we have   $\delta(v_{0})\geq c>0$ by \eqref{GlobalS} .
			
			From Corollary~\ref{Remark} and Lemma~\ref{GlobalS} we deduce that the solution $v$ to \eqref{NLS} with initial data $v_{0}$ is global and satisfies $\delta(v(t))>0$ for all $t\geq 0$. In addition, Lemma~\ref{StabilityNLS} implies that $u(s_{n}+\tau)\to v(\tau)$ in $\dot{H}_a^{1}$ and so 
			$\delta(u(s_{n}+\tau))\geq c_{0}$ for $n$ large. From the local 
			theory for \eqref{NLS} we have $\|u(t_{n})-u(s_{n}+\tau)\|_{\dot{H}_a^{1}}\to 0$
			as $n\to +\infty$. This implies $\delta(t_{n})\geq c_{0}$ for some $c_{0}>0$ and all $n$ large.
			
			Since $\delta(t_{n})\geq c_{0}$,   we deduce that $t_{n} \to t_{*}\in [0, \infty)$ along some subsequence from Lemma~\ref{DeltaZero}.  By continuity  and compactness, we obtain that
			\begin{align}\label{Cqe}
				& u(t_{n}) \to u(t_{*})\neq 0 \qtq{ strongly in $H_a^{1}$,}\\  \label{lmn}
				& \lambda(t_{n})^{-{\frac{d-2}{2}}}u(t_{n}, \tfrac{\cdot}{\lambda(t_{n})}) \to \zeta\neq 0
				\qtq{ strongly in $\dot{H}_a^{1}$.}
			\end{align}
			As $\lambda(t_{n})\to \infty$, \eqref{Cqe} and \eqref{lmn} together yield a contradiction. Thus $\lambda(s_{n})\to \infty$ as $n\to \infty$ and hence Scenario I implies \eqref{step11}.

			\textsl{Step 2.}
			There exists $\delta_{1}>0$,  so that for any $T\geq 0$, either
			\begin{equation}\label{MinMax}
				\inf_{t\in [T, T+\tfrac{1}{\lambda(T)^{2}}]}\delta(t)\geq \delta_{1} \quad \text{or}\quad
				\sup_{t\in [T,T+\tfrac{1}{\lambda(T)^{2}}]}\delta(t)<\delta_{0}.
			\end{equation}
			Suppose instead that there exist $t_{n}^{*}\geq 0$ and
			$t_{n}$, $t^{\prime}_{n}\in  [t_{n}^{*}, t_{n}^{*}+\tfrac{1}{\lambda(t_{n}^{*})^{2}}]$ so that
			\begin{align}\label{ContraStep2}
				&\delta(t_{n})\to 0 \quad \text{and}\quad \delta(t^{\prime}_{n})\geq \delta_{0}.
			\end{align}
			
			Notice that $t_{n}\to \infty$. Indeed, if $t_{n}\leq C$ for all $n\in \N$, then   $t_{n}\to a\in [0, \infty)$ up to a  subsequence. By continuity and \eqref{ContraStep2}, we derive that $\delta(a)=0$, which contradicts the fact that  $\delta(t)>0$ for all $t\geq 0$.  Thus, as $t_{n}\to \infty$ and $\delta(t_{n})\to 0$, Lemma~\ref{DeltaZero} yields that
			\begin{align}\label{Inll}
				\lambda(t_{n})\to \infty  \quad\text{as\ } {n\to \infty}.
			\end{align}
			
			On the other hand, we claim that $\lambda(t^{\prime}_{n})$ is bounded. In fact, this follows from Lemma~\ref{DeltaZero} and \eqref{ContraStep2} if $t^{\prime}_{n}$ is unbounded, and by Step 1 above if $t^{\prime}_{n}$ is bounded. 
			
			The boundedness of $\lambda(t_n^\prime)$   implies  that $\lambda(t_{n})$ is bounded, which is a contradiction.

			\textsl{Step 3.} In this step, we will prove that  there exists a constant $C>0$ such that
			\begin{align}\label{BoundCenter}
				0\leq t_{1}\leq t_1'\leq t_2'\leq t_{2}=t_{1}+\tfrac{1}{C^{2}_{1}\lambda(t_{1})^{2}}
				\Rightarrow 
				\left|\tfrac{1}{\lambda(t_2')^{2}}-\tfrac{1}{\lambda(t_1')^{2}}\right|\leq C\int^{t_{2}}_{t_{1}}\delta(t)\,dt,
			\end{align}
			for some $C>0$. Here $C_{1}\geq 1$ is the constant given in Step 1.
			
			By Step 2, we can suppose that $\sup_{t\in [t_{1}, t_{2}]}\delta(t)<\delta_{0}$ or
			$\inf_{t\in [t_{1}, t_{2}]}\delta(t)\geq \delta_{1}$. In the former case, we deduce \eqref{BoundCenter} via the estimate  
			\[
			\left|\tfrac{\lambda^{\prime}(t)}{\lambda(t)^{3}}\right|\lesssim \delta(t)\qtq{for}\delta(t)<\delta_{0}.
			\]
			In the last case, we note that $\int^{t_{2}}_{t_{1}}\delta_{1}\,dt\geq \int^{t_{2}}_{t_{1}}\delta(t)\,dt$ and
			\[
			|t_1'-t_2'|\leq \tfrac{1}{C^{2}_{1}\lambda(t_{1})^{2}}\leq \tfrac{1}{\lambda(t_1')^{2}}.
			\]
			Thus, by Step 1 we obtain
			\[
			\left|\tfrac{1}{\lambda(t_2')^{2}}-\tfrac{1}{\lambda(t_1')^{2}}\right|
			\leq\tfrac{2 C^{2}_{1}}{\lambda(t_1')^{2}}
			\leq \tfrac{2 C^{4}_{1}}{\lambda({t_{1}})^{2}}
			=2C^{5}_{1}|t_{2}-t_{1}|\leq \tfrac{2C^{5}_{1}}{\delta_{1}}\int^{t_{2}}_{t_{1}}\delta(t)dt.
			\]
			Combining the inequalities above yields \eqref{BoundCenter}.
		\end{proof}
		
		We now rule out the possibility of infinite-time blowup:

		\begin{proof}[{Proof of Theorem~\ref{ITfinito22}}] Suppose $u$ is a solution as in Proposition~\ref{Compacness11} with $T^{*}=\infty$.  In particular, $E_a(u_{0})=E_a^{c}(W_a)$, $\| u_{0}\|^{2}_{\dot{H}_a^{1}}<\| W_a\|^{2}_{\dot{H}_a^{1}}$ and 
			there exists a function  $0<\lambda(t)<\infty$ for all $t\geq0$, such that the set
			\begin{equation*}
				\left\{\lambda(t)^{-{\frac{d-2}{2}}}u(t, \tfrac{x}{\lambda(t)}): t\in [0, \infty)\right\} \quad
				\text{is pre-compact in $\dot{H}_a^{1}(\R^{d})$}.
			\end{equation*}
			
			From Proposition~\ref{BoundedUN}, there exists an increasing sequence $t_{n}\to\infty$ with $\delta(t_{n})\to 0$. In particular, we can choose $t_{n}$ so that
			\[
			\delta(t_{0})\leq \tfrac{1}{2C^{*}} \qtq{and} \delta(t_{n})\to 0,
			\]
			where $C^{*}:=C\cdot\widetilde{C}$, with $C$ and $\widetilde{C}$ the constants in \eqref{BoundT11} and \eqref{BoundCenter1}. 
			
			Consider $0\leq a\leq b$.  For $n$ large we see that $b+\tfrac{1}{\lambda(b)^{2}}<t_{n}$. By estimates \eqref{BoundT11} and \eqref{BoundCenter1}, we find
			\[
			\left|\tfrac{1}{\lambda(b)^{2}}-\tfrac{1}{\lambda(t_{n})^{2}}\right|\leq C^{*}
			\sup_{a\leq t\leq t_{n}}\bigl(\tfrac{1}{\lambda(t)^{2}}\bigr)\left\{\delta(a)+\delta(t_{n})  \right\}.
			\]
			Sending $n\to\infty$, it follows that 
			\[
			\sup_{t\geq a}\tfrac{1}{\lambda(t)^{2}}\leq C^{*}\delta(a)\sup_{t\geq a}\tfrac{1}{\lambda(t)^{2}}.
			\]
			Choosing $a=t_{0}$, we get
			\[
			\sup_{t\geq t_{0}}\tfrac{1}{\lambda(t)^{2}}=0.
			\]
			Lemma~\ref{Lemma111} then implies that $\delta(t)=0$ for all $t\geq t_{0}$, a contradiction.\end{proof}
		
		Proposition~\ref{Compacness11} and Theorems~\ref{Tfinito}--\ref{ITfinito22} now yield Theorem~\ref{Threshold}(i).




	\end{document}